\documentclass[11pt,reqno]{amsart}

\usepackage{amsmath,amssymb,amsfonts,amsthm,latexsym,graphicx,multirow,hyperref,enumerate,tikz,lmodern}
\usepackage[utf8]{inputenc}
\usetikzlibrary{calc}

\newcommand{\Sym}{\mathrm{Sym}}
\newcommand{\Alt}{\mathrm{Alt}}
\newcommand{\Aut}{\mathrm{Aut}}
\newcommand{\Inn}{\mathrm{Inn}}
\newcommand{\Out}{\mathrm{Out}}
\newcommand{\Cay}{\mathrm{Cay}}
\newcommand{\Fix}{\mathrm{Fix}}
\newcommand{\rs}{\mathrm{+}}
\newcommand{\ls}{\mathrm{-}}
\newcommand{\HH}{\mathrm{H}}

\newcommand{\mg}{}
\newcommand{\DRR}{\text{DRR}}
\newcommand{\GRR}{\text{GRR}}
\newcommand{\Dic}{\textup{Dic}}
\newcommand{\D}{\mathrm{D}}
\newcommand{\DMSR}{\text{D$m$SR}}
\newcommand{\DMSRS}{\text{D$m$SRs}}
\newcommand{\GMSR}{\text{G$m$SR}}
\newcommand{\GMSRS}{\text{G$m$SRs}}
\newcommand{\HGR}{\mathcal{HGR}}
\newcommand{\UPHA}{\textup{HA}}
\newcommand{\UPHS}{\textup{HS}}
\newcommand{\UPHC}{\textup{HC}}
\newcommand{\UPSD}{\textup{SD}}
\newcommand{\UPCD}{\textup{CD}}
\newcommand{\UPTW}{\textup{TW}}
\newcommand{\UPAS}{\textup{AS}}
\newcommand{\UPPA}{\textup{PA}}
\newcommand{\Soc}{\mathrm{Soc}}
\newcommand{\PSL}{\mathrm{PSL}}
\newcommand{\AGL}{\mathrm{AGL}}
\newcommand{\M}{\mathrm{M}}

\newcommand{\Ker}{\mathrm{Ker}}

\newcommand{\odd}{\mathrm{odd}}

\newtheorem{theorem}{Theorem}[section]
\newtheorem{lemma}[theorem]{Lemma}
\newtheorem{proposition}[theorem]{Proposition}
\newtheorem{corollary}[theorem]{Corollary}
\newtheorem{problem}[theorem]{Problem}

\theoremstyle{definition}

\newtheorem{definition}[theorem]{Definition}
\newtheorem{notation}[theorem]{Notation}

\newtheorem{remark}[theorem]{Remark}

\begin{document}

\title[HGR]{Asymptotic enumeration of Haar graphical representations}

\author{Yunsong Gan}
\address{(Gan) School of Mathematics and Statistics\\Central South University\\Changsha, Hunan, 410083\\P.R. China}
\email{songsirr@126.com}

\author{Pablo Spiga}
\address{(Spiga) Dipartimento di Matematica e Applicazioni\\University of Milano-Bicocca\\Via Cozzi 55, 20125 Milano\\Italy}
\email{pablo.spiga@unimib.it}

\author{Binzhou Xia}
\address{(Xia) School of Mathematics and Statistics\\The University of Melbourne\\Parkville, VIC 3010\\Australia}
\email{binzhoux@unimelb.edu.au}


\begin{abstract}
This paper represents a significant leap forward in the problem of enumerating vertex-transitive graphs. Recent breakthroughs on symmetry of Cayley (di)graphs show that almost all finite Cayley (di)graphs have the smallest possible automorphism group. Extending the scope of these results, we enumerate (di)graphs admitting a fixed semiregular group of automorphisms with $m$ orbits. Moreover, we consider the more intricate inquiry of prohibiting arcs within each orbit, where the special case $m=2$ is known as the problem of finding Haar graphical representations (HGRs). We significantly advance the understanding of HGRs by proving that the proportion of HGRs among Haar graphs of a finite nonabelian group approaches $1$ as the group order grows. As a corollary, we obtain an improved bound on the proportion of DRRs among Cayley digraphs in the solution of Morris and the second author to the Babai-Godsil conjecture.

\textit{Key words:} Haar graph; Cayley graph; automorphism groups; asymptotic enumeration

\textit{MSC2020:} 20B25, 05E18, 05C30
\end{abstract}

\maketitle

\section{Introduction}\label{SEC1}

A \emph{Haar graph} of a group $G$ is a bipartite graph whose automorphism group has a subgroup isomorphic to $G$ that is semiregular on the vertex set with orbits giving a bipartition. For such a graph, we may identify the vertex set with $G\times\{1,2\}$ such that the parts of the bipartition are $G\times\{1\}$ and $G\times\{2\}$. Therefore, every Haar graph of $G$ can be determined uniquely by a subset $S$ of $G$ such that $(g,1)$ and $(h,2)$ are adjacent if and only if $hg^{-1}\in S$. We denote this Haar graph by $\HH(G,S)$. Introduced initially in~\cite{HMP2002},  Haar graphs have been studied extensively by several authors from different viewpoints~\cite{CEP2018,Dobsen2022,DFS2020,EP2016,FKWY2020,FKY2020,KK2014}.

A Haar graph of $G$ is called a \emph{Haar graphical representation} (HGR), if its automorphism group is isomorphic to $G$. When $G$ is abelian, an easy observation (see Definition~\ref{DEF005}) shows that, for each Haar graph of $G$, there exists an automorphism $\iota$ of order two such that the group $G\rtimes\langle\iota\rangle$  acts regularly on the vertex set, transforming each Haar graph into a Cayley graph. Hence no Haar graph of an abelian group is an HGR. It is then natural to ask: Which groups permit HGRs? In fact, this problem has been posed for finite groups in relevant research, for example~\cite{DFS2020,FKWY2020}.
Recently, Morris and the second author~\cite{MS2024} have classified finite groups admitting an HGR. They proved that, except for abelian groups and $22$ small groups, every finite group admits an HGR. This leads to the following problem.

\begin{problem}\label{PROB001}
For a finite nonabelian group $G$, count the HGRs of $G$.
\end{problem}

The first result of this paper addresses Problem~\ref{PROB001} by giving an upper bound on the number of subsets $S$ of a  nonabelian group $G$ such that $\HH(G,S)$ is not an HGR. For completeness, we also establish a similar result for abelian groups $G$ such that the automorphism group of $\HH(G,S)$ is not isomorphic to $G\rtimes\langle\iota\rangle$.

\begin{theorem}\label{THM001}
Let $\varepsilon\in\left(0,0.1\right]$, and let $n_\varepsilon$ be a positive integer such that for all $n\geq n_\varepsilon$,
\begin{equation}\label{eq032}
(6+2\log_2n)(n^{0.5-\varepsilon}+\log_2n)^2+(1-\log_2n)(n^{0.5-\varepsilon}+\log_2n)+\log_2^2n+2\log_2n<n^{1-\varepsilon}.
\end{equation}
Let $G$ be a finite group of order $n$, and let $f_\varepsilon(n)=\frac{n^{0.5-\varepsilon}}{24(\log_2n)^{2.5}}-\frac{3\log_2^2n}{4}-15$.
\begin{enumerate}[{\rm(a)}]
\item\label{THM001.1} If $G$ is nonabelian, then the number of subsets $S$ of $G$ such that $\HH(G,S)$ is not an HGR is less than $2^{n-f_\varepsilon(n)}$.
\item\label{THM001.2} If $G$ is abelian, then the number of subsets $S$ of $G$ such that the automorphism group of $\HH(G,S)$ is not isomorphic to $G\rtimes\langle\iota\rangle$ is less than $2^{n-f_\varepsilon(n)-1}$.
\end{enumerate}
\end{theorem}

%

Fix some $\varepsilon\in\left(0,0.1\right]$. Note that the left-hand side of~\eqref{eq032} is approximately $2\log_2n\cdot n^{1-2\varepsilon}$, which is smaller than $n^{1-\varepsilon}$ for sufficiently large $n$. Hence there exists an integer $n_\varepsilon$ such that~\eqref{eq032} holds for all $n\geq n_\varepsilon$. Since a group of order $n$ has exactly $2^n$ subsets, Theorem~\ref{THM001}\eqref{THM001.1} yields that, for a nonabelian group $G$ of order $n$ satisfying~\eqref{eq032}, the proportion of subsets $S$ of $G$ such that $\HH(G,S)$ is an HGR is larger than
\begin{equation}\label{eq004}
1-2^{-\frac{n^{0.5-\varepsilon}}{24(\log_2n)^{2.5}}+\frac{3\log_2^2n}{4}+15}.
\end{equation}
Since~\eqref{eq004} approaches $1$ as $n$ tends to infinity, this reveals that almost all Haar graphs of a finite nonabelian group are HGRs. Similarly, Theorem~\ref{THM001}\eqref{THM001.2} implies that almost all Haar graphs of a finite abelian group have automorphism group isomorphic to $G\rtimes\langle\iota\rangle$.

We remark that the graphs in Theorem~\ref{THM001} are labelled. For comparison, we offer an unlabeled  version as follows.

\begin{theorem}\label{THM005}
Let $\varepsilon\in\left(0,0.1\right]$, and let $n_\varepsilon$ be a positive integer such that~\eqref{eq032} holds for all $n\geq n_\varepsilon$. Let $G$ be a finite group of order $n$, and let
$h_\varepsilon(n)=\frac{n^{0.5-\varepsilon}}{24(\log_2n)^{2.5}}-\log_2^2(2n)-\frac{3\log_2^2n}{4}-2\log_2n-15$.
\begin{enumerate}[{\rm(a)}]
\item\label{THM005.1} If $G$ is nonabelian, then the proportion of HGRs of $G$ among Haar graphs of $G$, up to isomorphism, is greater than $1-2^{-h_\varepsilon(n)}$.
\item\label{THM005.2} If $G$ is abelian, then the proportion of Haar graphs of $G$ whose automorphism group is isomorphic to $G\rtimes\langle\iota\rangle$, among Haar graphs of $G$, up to isomorphism, is greater than $1-2^{-h_\varepsilon(n)}$.
\end{enumerate}
\end{theorem}

We now discuss the roots of Theorem~\ref{THM001} and present other results. The study of groups represented as the automorphism groups of (di)graphs of the same order commences with classical problems known as DRR and GRR problems. Let $G$ be a finite group and $S$ a subset of $G$. A \emph{Cayley digraph} $\Cay(G,S)$ of $G$ with \emph{connection set} $S\subseteq G$ is a digraph with vertex set $G$ such that $(g,h)$ is an arc if and only if $hg^{-1}\in S$. A Cayley digraph $\Cay(G,S)$ is a graph if and only if $S$ is inverse-closed. We call a Cayley (di)graph of $G$ a \emph{(di)graphical regular representation}, or GRR (DRR) for short, if its automorphism group is isomorphic to $G$.

A natural problem is to determine which finite groups admit a DRR or GRR. In~1980, Babai~\cite{Babai1980} proved that $C_2^2$, $C_2^3$, $C_2^4$, $C_3^2$ and $Q_8$ are the only groups without DRRs. Based on a series of partial results (see~\cite{Imrich1978,NW1972}, for example), the problem for GRRs was eventually solved by Godsil~\cite{Godsil1981} in 1981: Apart from abelian groups of exponent greater than~$2$, generalized dicyclic groups (see Section~\ref{SUBSEC7.2} for definition), and $13$ small solvable groups, every finite group admits a GRR. Following this, Babai and Godsil~\cite{BG1982,Godsil1981} conjectured in early~1980s that almost all finite Cayley digraphs are DRRs. This conjecture was confirmed by Morris and the second author by showing the following theorem in~\cite{MS2021}.

\begin{theorem}[Morris-Spiga]\label{THM002}
Let $G$ be a finite group of order $n$. When $n$ is sufficiently large, the proportion of subsets $S$ of $G$ such that $\Cay(G,S)$ is not a $\DRR$ is at most $2^{-\frac{bn^{0.499}}{4\log^3_2n}+2}$, where $b$ is an absolute constant.
\end{theorem}

We illustrate a connection between Theorem~\ref{THM002} and Theorem~\ref{THM001}. Let $\Gamma$ be a digraph with vertex set $V$. The \emph{standard double cover} $\D(\Gamma)$ of $\Gamma$ is the graph with vertex set $V\times\{1,2\}$ and edge set $\{\{(g,1),(h,2)\}\mid (g,h)\text{ is an arc of }\Gamma\}$. Observe that the automorphism group of $\Gamma$ also acts as a group of automorphisms of $\D(\Gamma)$ that stabilizes both of $V\times\{1\}$ and $V\times\{2\}$ (see~\cite[Section~3.1]{GLP2004}), and that the Haar graph $\HH(G,S)$ is precisely the standard double cover $\D(\Cay(G,S))$.
This implies that, if all the automorphisms of $\HH(G,S)$ that stabilizes both of the biparts constitute solely the group $G$, then $\Cay(G,S)$ is a DRR. Therefore, Theorem~\ref{THM001} leads to the following stronger assertion than Theorem~\ref{THM002} regarding the Babai-Godsil conjecture.

\begin{corollary}\label{CORO001}
For each $\varepsilon\in\left(0,0.1\right]$, there exists $n_\varepsilon>0$ such that for a finite group $G$ of order $n\geq n_\varepsilon$, the proportion of subsets $S$ of $G$ with $\Cay(G,S)$ not a $\DRR$ is less than $2^{-n^{0.5-\varepsilon}}$.
\end{corollary}
For the proportion of GRRs among Cayley graphs, Babai, Godsil, Imrich and Lov\'{a}sz conjectured that, except for the groups $G$ that are abelian of exponent greater than $2$ or generalised dicyclic, almost all finite Cayley graphs of $G$ are GRRs~\cite[Conjecture~2.1]{BG1982}. This was recently confirmed by the third author and Zheng~\cite{XZ2023}, following the framework developed in~\cite{MMS2022,Spiga2024}.

\begin{theorem}[Xia-Zheng]\label{THM003}
Let $G$ be a finite group of order $n$ such that $G$ is neither abelian of exponent greater than $2$ nor generalized dicyclic. When $n$ is sufficiently large, the proportion of inverse-closed subsets $S$ of $G$ such that $\Cay(G,S)$ is not a $\GRR$ is at most $2^{-\frac{n^{0.499}}{8\log^3_2n}+\log_{2}^2n+3}$.
\end{theorem}
It is evident that a digraph constitutes a Cayley digraph of $G$ if and only if its automorphism group contains a regular subgroup isomorphic to $G$. Thus, the essence of Theorems~\ref{THM002} and~\ref{THM003} lies in the observation that, except for  two special families outlined in Theorem~\ref{THM003}, nearly all (di)graphs possessing a regular group of automorphisms exhibit automorphism groups that are ``as small as possible''. This leads us to the following question: What about (di)graphs endowed with a semiregular group of automorphisms?

Let $G$ be a finite group. A digraph is termed an \emph{$m$-Cayley digraph} of $G$ if it admits a group $G$ of automorphisms acting semiregularly on the vertex set, featuring precisely $m$ orbits. Similarly, for each $m$-Cayley digraph of $G$, we can identify the $m$ orbits of $G$ with $m$ copies of $G$ and employ a set-matrix $\mathcal{S}$ of $G$ to delineate the arcs between them. Consequently, an $m$-Cayley digraph of $G$ can be conceptualized in terms of $G$ along with a set-matrix $\mathcal{S}$ of $G$ (see Section~\ref{SUBSEC2.1}), denoted by $\Cay(G,\mathcal{S})$. A digraph $\Cay(G,\mathcal{S})$ is a graph if and only if $\mathcal{S}$ is inverse-closed, as  defined in Definition~\ref{DEF001}.

An $m$-Cayley (di)graph of a finite group $G$ is called a \emph{(di)graphical $m$-semiregular representation}, abbreviated as $\GMSR$ ($\DMSR$), if its automorphism group is isomorphic to $G$. In their work~\cite{DFS20201}, Du, Feng, and the second author demonstrated that every group of order greater than $8$ possesses a $\GMSR$ and $\DMSR$ for each $m\geq2$. This naturally leads to the question: For a finite group $G$ of sufficiently large order, are almost all $m$-Cayley (di)graphs of $G$ $\GMSR$s ($\DMSR$s)? We provide an affirmative answer to this question by proving the following theorem.

\begin{theorem}\label{THM004}
Fix an integer $m\geq2$, and let $G$ be a finite group of order $n$. When $n$ is sufficiently large, the proportion of (inverse-closed) set-matrices $\mathcal{S}$ of $G$ such that $\mg\Cay(G,\mathcal{S})$ is a $\DMSR$ ($\GMSR$) is greater than $1-m^2/\sqrt{n}$.
\end{theorem}

To provide further background of Theorem~\ref{THM004}, it is worth remarking the relationship between Cayley (di)graphs, vertex-transitive (di)graphs, and $m$-Cayley (di)graphs. Let $\mathcal{V}$ and $\mathcal{C}_m$ be the set of vertex-transitive (di)graphs and $m$-Cayley (di)graphs, respectively. It is clear that $\mathcal{C}_1\subseteq\mathcal{V}$. However, determining $\mathcal{V}\setminus\mathcal{C}_1$ is difficult, and the famous McKay-Praeger conjecture~\cite{MP1994} states that $|\mathcal{C}_1|/|\mathcal{V}|$ approaches $1$ as the (di)graph order grows. Another related fascinating long-standing problem is the so-called \emph{Polycirculant conjecture}~\cite{Dragan81}, which asserts that
\begin{equation}\label{eq006}
\mathcal{C}_1\subseteq\mathcal{V}\subseteq\bigcup_{m\geq1}\mathcal{C}_m.
\end{equation}
We refer the reader to~\cite{AAS2019} for a survey of this conjecture. Given that the McKay-Praeger conjecture predicts the first inclusion in~\eqref{eq006} to be asymptotically tight, one might be interested in whether the second inclusion is tight. In this sense, Theorem~\ref{THM004} gives negative evidence by indicating that $\mathcal{V}\cap\mathcal{C}_m$ is extremely small relative to $\mathcal{C}_m$ for each $m\geq2$.

The validity of Theorem~\ref{THM004} relies significantly on the presence of arcs within each $G$-orbit in most $m$-Cayley (di)graphs of $G$. Indeed, a crucial step in establishing Theorem~\ref{THM004} is to utilize Theorem~\ref{THM003} (Theorem~\ref{THM002}) to assert that the induced sub(di)graph of a $G$-orbit typically conforms to a GRR (DRR) structure. Establishing an analog of Theorem~\ref{THM004} becomes notably more challenging if arcs within each $G$-orbit are prohibited. Even for $m=2$, corresponding to the examination of HGRs, Theorem~\ref{THM001} stands as the inaugural result enumerating them. The last main result of this paper is to provide an asymptotic result for the general case.

An $m$-Cayley graph $\mg\Cay(G,\mathcal{S})$ is $m$-partite with $G$-orbits as the $m$ parts if and only if the diagonal entries of $\mathcal{S}$ are empty sets. Such a set-matrix $\mathcal{S}$ is said to be \emph{skew} (see Definition~\ref{DEF001}). An $m$-Cayley graph $\mg\Cay(G,\mathcal{S})$ with skew $\mathcal{S}$ is called an \emph{$m$-partite graphical semiregular representation} ($m$-PGSR) of $G$, if its automorphism group is isomorphic to $G$. Notably, $2$-PGSRs are equivalent to HGRs. Similar to the asymptotic result in Theorem~\ref{THM001} for HGRs,  the following theorem demonstrates that the majority of skew set-matrices $\mathcal{S}$ make $\mg\Cay(G,\mathcal{S})$ an $m$-PGSR.

\begin{theorem}\label{THM009}
Fix an integer $m\geq3$, and let $G$ be a finite group of order $n$. When $n$ is sufficiently large, the proportion of skew set-matrices $\mathcal{S}$ of $G$ such that $\mg\Cay(G,\mathcal{S})$ is an $m$-PGSR is greater than $1-m^2/\sqrt{n}$.
\end{theorem}

The subsequent sections of this paper unfold as follows. In the next section, we introduce some notations and preliminary results. Following that, Section~\ref{SEC3} outlines the strategy employed to prove Theorem~\ref{THM001} and streamlines the proof to the enumeration of primitive groups harboring a ``large'' regular subgroup. This enumeration is carried out in Section~\ref{SEC4}, facilitating the culmination of the proof for Theorems~\ref{THM001} and~\ref{THM005} in Section~\ref{SEC5}. The proof of Theorem~\ref{THM004}, divided into the cases of digraphs and graphs, is presented towards the conclusion of Sections~\ref{SEC6} and~\ref{SEC7}, respectively. We establish Theorem~\ref{THM009} in Section~\ref{SEC8}.

\section{Preliminaries}\label{SEC2}

For a graph $\Gamma$, we use $V(\Gamma)$ to denote its vertex set, and for a vertex $u$ of $\Gamma$, we use $N_\Gamma(u)$ to denote the neighbourhood of $u$ in $\Gamma$. For a finite group $X$, denote by $P(X)$ the minimal index of proper subgroups of $X$, and by $\Soc(X)$ the socle (product of the minimal normal subgroups) of $X$.
For a permutation $\alpha$ of a set $\Omega$, denote by $\Fix(\alpha)$ the set of elements in $\Omega$ fixed by $\alpha$.

\subsection{\boldmath{$m$}-Cayley (di)graphs}\label{SUBSEC2.1}

Let $m$ be a positive integer. In Introduction, an $m$-Cayley digraph of a group $G$ is defined as a digraph whose automorphism group has a subgroup isomorphic to $G$ that is semiregular on the vertex set with exactly $m$ orbits. Now we give an equivalent definition of $m$-Cayley (di)graphs.
Let
\[
\mathcal{S}=(S_{i,j})_{m\times m}=
\begin{pmatrix}
  S_{1,1} & S_{1,2} & \cdots & S_{1,m} \\
  S_{2,1} & S_{2,2} & \cdots & S_{2,m} \\
  \vdots & \vdots & \ddots & \vdots \\
  S_{m,1} & S_{m,2} & \cdots & S_{m,m} \\
\end{pmatrix}
\]
be a \emph{set-matrix} of $G$, namely, a matrix whose entries are subsets of $G$. The \emph{$m$-Cayley digraph} $\mg\Cay(G,\mathcal{S})$ of $G$ with respect to $\mathcal{S}$ is the digraph with vertex set $G\times\{1,\ldots,m\}$ and arc set
\[
\bigcup_{i,j\in\{1,\ldots,m\}}\left\{\big((g,i),(sg,j)\big)\,\middle|\,s\in S_{i,j},\,g\in G\right\}.
\]

\begin{definition}\label{DEF001}
We call a set-matrix $\mathcal{S}$ \emph{inverse-closed} if  $S_{j,i}=S_{i,j}^{-1}$ for all $i,j\in\{1,\ldots,m\}$. For such $\mathcal{S}$, we call the digraph $\mg\Cay(G,\mathcal{S})$ an \emph{$m$-Cayley graph} as it is undirected. If a set-matrix $\mathcal{S}$ is inverse-closed and satisfies $S_{i,i}=\emptyset$ for each $i\in\{1,\ldots,m\}$, then $\mathcal{S}$ is said to be \emph{skew}.
\end{definition}

It is clear that each element $x$ of $G$ induces an automorphism $R(x)$ of $m$-Cayley digraphs of $G$ by mapping $(g,i)$ to $(gx,i)$ for each $g\in G$ and $i\in\{1,\ldots,m\}$. In this way, $R(G)$ is a semiregular subgroup of $\Aut(\mg\Cay(G,\mathcal{S}))$ with orbits $G\times\{1\},\ldots,G\times\{m\}$. For simplicity of notation, we identify $R(G)$ with $G$ and $R(x)$ with $x$ for each $x\in G$ when there is no confusion.

If a (di)graph has a group $G$ of automorphisms acting semiregularly on the vertex set with exactly $m$ orbits, then it is isomorphic to some $m$-Cayley (di)graph of $G$. Hence the two definitions of $m$-Cayley (di)graphs from Introduction and this section are equivalent.

\subsection{Haar graphs and odd-quotient graphs}\label{SUBSEC2.2}

For convenient, we adopt the following notations about Haar graphs in the rest of this paper.

\begin{notation}\label{NOT001}
Let $G$ be a finite group. Then each Haar graph $\Gamma$ of $G$ has vertex set $G\times\{1,2\}$. For each $g\in G$, denote $g_\rs=(g,1)$ and $g_\ls=(g,2)$. For $S\subseteq G$ and $\epsilon\in\{\rs,\ls\}$, denote $S_{\epsilon}=\{s_{\epsilon}\mid s\in S\}$. In particular,
$G_\rs=G\times\{1\}$ and $G_\ls=G\times\{2\}$. Let $\Aut^+(\Gamma)$ be the subgroup of $\Aut(\Gamma)$ fixing setwise  $G_\rs$ and $G_\ls$.
\end{notation}

The next definition reveals that each Haar graph of an abelian group has an automorphism that exchanges the parts $G_\rs$ and $G_\ls$.

\begin{definition}\label{DEF005}
Let $G$ be an abelian group, and let $\iota$ be the permutation on $G_\rs\cup G_\ls$ which maps $g_\epsilon$ to $(g^{-1})_{-\epsilon}$ for each $g\in G$ and $\epsilon\in\{\rs,\ls\}$. It is easy to check that, $\iota$ is a graph automorphism of each Haar graph of $G$ and normalizes the semiregular group $R(G)$ (via the right multiplication action on $G_\rs\cup G_\ls$). Hence each Haar graph of $G$ has a group of automorphisms isomorphic to $G\rtimes\langle\iota\rangle$. In particular, each Haar graph of an abelian group is a Cayley graph.
\end{definition}

The following concept, as defined in~\cite[Definition~6.1]{MS2021}, is used in the proof of Theorem~\ref{THM001}.

\begin{definition}\label{DEF002}
Let $\Gamma$ be a graph, and let $\mathcal{B}$ be a partition of $V(\Gamma)$ such that for any fixed $B,C\in\mathcal{B}$, all vertices in $B$ have the same number of neighbours, denoted by $e(B,C)$, in $C$. The \emph{odd-quotient digraph} $\Gamma^{\odd}_{\mathcal{B}}$ of $\Gamma$ with respect to $\mathcal{B}$ is the digraph with vertex set $\mathcal{B}$ such that $B$ is adjacent to $C$ if and only if
$e(B,C)$ is odd. If $\mathcal{B}$ is the orbit partition of some group $H\leq\Aut(\Gamma)$, we write $\Gamma^{\odd}_{\mathcal{B}}$ as $\Gamma^{\odd}_{H}$.
\end{definition}

In the above definition, if $|B|=|C|$, then a double-counting of the edges between $B$ and $C$ shows that $e(B,C)=e(C,B)$. This leads to the following remark.

\begin{remark}\label{RMK001}
In Definition~\ref{DEF002}, if all the sets in $\mathcal{B}$ have the same size, then $\Gamma_{\mathcal{B}}^{\odd}$ is undirected.
\end{remark}

\subsection{Simple arithmetic results}\label{SUBSEC2.3}

For a subset $S$ of a group, let $\mathcal{I}(S)$ be the set of elements in $S$ of order at most $2$, and let
\[
c(S)=\frac{|S|+|\mathcal{I}(S)|}{2}.
\]
The proof of the following lemma is straightforward (see~\cite[Lemma~2.2]{Spiga2021}, for example).

\begin{lemma}\label{LEM001}
Let $S$ be an inverse-closed subset of a finite group. Then the number of inverse-closed subsets of $S$ is $2^{c(S)}$.
\end{lemma}

We also require the following two useful lemmas, where the proof of the first is elementary and hence omitted.

\begin{lemma}\label{LEM002}
For a non-empty set, the number of subsets of odd size equals the number of subsets of even size.
\end{lemma}

\begin{lemma}\label{LEM003}
For a set of size $n$, the number of its subsets with a given size is at most $2^{n}/\sqrt{n}$.
\end{lemma}
\begin{proof}
Clearly, the lemma holds true for $n=1$, and so we assume $n\geq2$. According to Stirling's formula (see~\cite{Robbins1955}), for any positive integer $x$, the factorial $x!$ satisfies
\[
\sqrt{2\pi x}\left(\frac{x}{e}\right)^x <x!<\sqrt{2\pi x}\left(\frac{x}{e}\right)^x e^{\frac{1}{12x}}.
\]
Hence,
\[
\binom{2x}{x}=\frac{(2x)!}{x!\cdot x!}<
       e^{\frac{1}{24x}}
\frac  {\sqrt{4\pi x}\cdot (2x)^{2x}}
       {\big(\sqrt{2\pi x} \cdot x^{x}\big)^2}=\frac{e^{\frac{1}{24x}}}{\sqrt{\pi}}\cdot \frac{2^{2x}}{\sqrt{x}}
       < \frac{2^{2x}}{\sqrt{2x}}.
\]
Note that the maximum binomial coefficient in $\binom{n}{0},\ldots,\binom{n}{n}$ is $\binom{n}{\lfloor n/2\rfloor}$. If $n$ is even, the conclusion immediately follows from the above inequality by taking $2x$ as $n$. If $n$ is odd, we also obtain
\[
\binom{n}{\lfloor n/2\rfloor}=\frac{n}{(n+1)/2}\binom{n-1}{(n-1)/2}
<\frac{2n}{n+1}\cdot\frac{2^{n-1}}{\sqrt{n-1}}<\frac{2^{n}}{\sqrt{n}},
\]
as the lemma asserts.
\end{proof}

\subsection{Counting and group structure}\label{SUBSEC2.4}

For a positive $d\in\mathbb{R}$, a group $X$ is said to be \emph{$d$-generated} if it has a generating set of size at most $d$ (note that $\lfloor d\rfloor$ is not necessarily the minimum size of a generating set of $X$).
Since a chain of subgroups in $G$ has length at most $\log_2{|X|}$, it turns out that $X$ is always $(\log_2{|X|})$-generated.
\begin{lemma}\label{LEM004}
Let $X$ be a finite group. The following statements hold.
\begin{enumerate}[{\rm(a)}]
\item\label{LEM004.1} For a fixed positive number $m$, there are at most $|X|^{\log_2m}$ subgroups of order $m$ in $X$.
\item\label{LEM004.3} $|\Aut(X)|\leq|X|^{\log_2|X|}=2^{\log^2_2{|X|}}$.
\item\label{LEM004.4}$X$ has less than $2^{(\log_2^2n)/4+3}$ subgroups.
\end{enumerate}
\end{lemma}
\begin{proof}
Parts~\eqref{LEM004.1} and~\eqref{LEM004.3} follow from the previous paragraph, and part~\eqref{LEM004.4} is proved in~\cite{FusariSpiga}.
\end{proof}

The following deep result due to Lubotzky~\cite[Page~198,~(1)]{Lubotzky2001} estimates the number of finite groups of given order and given number of generators.

\begin{theorem}[Lubotzky]\label{LEM005}
The number of isomorphism classes of $d$-generated groups of order $n$ is at most $2^{2(d+1)\log_2^2n}$.
\end{theorem}

The next well-known result (see~\cite{LM1972} and~\cite[Theorem~4.1]{EN}, for example), along with its application Lemma~\ref{LEM007}, plays a crucial role in the proof of Theorem~\ref{THM001}.

\begin{lemma}\label{LEM006}
There is no finite nonabelian group with an automorphism inverting more than $3/4$ of its elements. Moreover, a finite group in which more than $3/4$ of the elements are involutions is an elementary abelian $2$-group.
\end{lemma}

\begin{lemma}\label{LEM007}
Let $X$ be a finite group, and let $G$ be a nonabelian subgroup of index $2$ in $X$. Then $|\mathcal{I}(X\setminus G)|\leq3|G|/4$.
\end{lemma}

\begin{proof}
Assume without loss of generality $\mathcal{I}(X\setminus G)\neq\emptyset$. Fix  $x\in\mathcal{I}(X\setminus G)$, so that $X=G\rtimes\langle x\rangle$. For each $g\in G$, we have $(gx)^2=1$ if and only if $x^{-1}gx=xgx=g^{-1}$, that is, $gx\in\mathcal{I}(X\setminus G)$ if and only if $x$ inverts $g$. As $G$ is nonabelian, we derive from Lemma~\ref{LEM006}
that there is no automorphism of $G$ inverting more than $3/4$ of its elements. Thus $|\mathcal{I}(X\setminus G)|\leq3|G|/4$.
\end{proof}

As usual, the symbol $\sqcup$ denotes disjoint union of sets, and $H\backslash G/K$ denotes the set of double cosets of $H$ and $K$ in $G$, for subgroups $H$ and $K$ of a group $G$.

\begin{lemma}\label{LEM008}
Let $G$ be a group with subgroups $H$ and $K$ of finite indices. Then either $|H\backslash G/K|\leq\frac{3}{4}\max\{|G\,{:}\,H|,|G\,{:}\,K|\}$, or $H=K$ is normal in $G$.
\end{lemma}

\begin{proof}
Without loss of generality, assume that $|H|\leq|K|$. Let $Hg_1K,\ldots,Hg_sK$ be the double cosets of size $|H|$. If $s\leq|G\,{:}\,H|/2$, then
\[
|H\backslash G/K|\leq s+\frac{|G\,{:}\,H|-s}{2}=\frac{|G\,{:}\,H|+s}{2}\leq\frac{3}{4}|G\,{:}\,H|=\frac{3}{4}\max\{|G\,{:}\,H|,|G\,{:}\,K|\},
\]
satisfying the conclusion of the lemma. Assume for the rest of the proof that $s>|G\,{:}\,H|/2$.

For each $i\in\{1,\ldots,s\}$, it follows from $Hg_iK=Hg_i$ that $K\subseteq g_i^{-1}Hg_i$, which together with $|H|\leq|K|$ implies $g_i^{-1}Hg_i=K$. In particular, $g_1^{-1}Hg_1=\cdots=g_s^{-1}Hg_s$. Hence every element of $Hg_1g_1^{-1}\sqcup\cdots\sqcup Hg_sg_1^{-1}$ normalizes $H$. This yields $|\mathbf{N}_G(H)|\geq|H|s>|H||G\,{:}\,H|/2=|G|/2$, which forces $\mathbf{N}_G(H)=G$, that is, $H$ is normal in $G$. Then we deduce from $g_1^{-1}Hg_1=K$ that $H=K$, proving the conclusion of the lemma.
\end{proof}

The following result follows from ~\cite[Lemma~2.3]{DX2000} and Lemma~\ref{LEM008}.

\begin{lemma}\label{LEM009}
Let $M$ be an intransitive permutation group with exactly two orbits $U$ and $W$, and let $\kappa$ be the number of double cosets of the stabilizers $M_w$ and $M_u$ in $M$, where $u\in U$ and $w\in W$. Then there are exactly $2^{\kappa}$ bipartite graphs $\Gamma$ with bipartition $\{U,W\}$ such that $M\leq\Aut(\Gamma)$. Moreover, if $|U|=|W|$ and $M$ is not semiregular, then $\kappa\leq\frac{3}{4}|U|$.
\end{lemma}

The final lemma is an immediate consequence of the Classification of Finite Simple Groups, and a detailed verification can be found in~\cite{Stefan}.

\begin{lemma}\label{LEM010}
For every finite nonabelian simple group $T$ we have $|\Out(T)|\leq\log_2|T|$.
\end{lemma}

\section{Strategy to prove Theorem~\ref{THM001}}\label{SEC3}

We first establish two reduction results for Theorem~\ref{THM001}, which will be presented in Sections~\ref{SUBSEC3.1} and~\ref{SUBSEC3.2}, respectively. Building upon  these results, the key to proving Theorem~\ref{THM001} lies in estimating the number of subsets $S$ of $G$ for which there exists a ``large'' subgroup $M$ of $\Aut^+(\HH(G,S))$ such that $G$ is maximal and core-free in $M$. We present this enumeration in Section~\ref{SUBSEC3.3} and dedicate Section~\ref{SEC4} to its proof. This proof, of independent interest, analyzes primitive groups with a ``large'' regular subgroup.
In essence, the analytical framework mirrors that of~\cite[Section~5]{MS2021}. Summarizing these results we arrive at an upper bound on the number of subsets $S$ of $G$ such that $\Aut^+(\HH(G,S))>G$ (see Proposition~\ref{PROP010}). Armed with this bound, we conclude the proof of Theorem~\ref{THM001} in Section~\ref{SEC5}, requiring only minimal additional effort.

Let us start with the following definition, which will also be used in Sections~\ref{SEC4} and~\ref{SEC5}.

\begin{definition}\label{DEF003}
Let $G$ be a group. For a subset $S$ of $G$ and a subgroup $X$ of $\Aut^+(\HH(G,S))$ such that $G<X$, we call $(S,X)$ an \emph{exceptional pair} of $G$ (with respect to $S$). If, in addition, $G$ is maximal in $X$, then we call $(S,X)$ a \emph{minimally exceptional pair}.
\end{definition}

\subsection{Babai-Godsil-like reduction}\label{SUBSEC3.1}

Let $X$ be a permutation group on $\Omega$. If a subset $\Delta$ of $\Omega$ is invariant under  $X$, we denote by $X|_\Delta$ the group induced by $X$ on $\Delta$. Moreover, if $\Delta$ is $\langle g\rangle$-invariant, where $g\in X$, then we write $g|_\Delta$ to denote the permutation induced by $g$ on $\Delta$.

\begin{lemma}\label{LEM011}
Let $G$ be a group of order $n$. For each positive integer $t$, the number of subsets $S$ of $G$ such that there exists a pair $(H,f)$ satisfying the following conditions~{\rm(a)} and~{\rm(b)} is less than $2^{n-\frac{n}{3t}\log_2\left(\frac{4}{3}\right)+\frac{\log_2^2n}{4}+\log_2n+2\log_2t+1}$.
\begin{enumerate}[{\rm(a)}]
\item\label{LEM011.1} $H$ is a nontrivial proper normal subgroup of $G$ with $|H|\leq t$;
\item\label{LEM011.2} $f\in\Aut^+(\HH(G,S))$ stabilizes every $H$-orbit on $V(\HH(G,S))$ and fixes $1_{\rs}$, and the induced permutation $f|_{(G\setminus H)_\rs}$ of $f$ on $(G\setminus H)_{\rs}$ is nontrivial.
\end{enumerate}
\end{lemma}

\begin{proof}
Let $\mathcal{H}$ be the set of nontrivial normal subgroups of $G$ with order at most $t$, and let
\[
\mathcal{S}=\{S\subseteq G\mid\text{there exists a pair $(H,f)$ satisfying both~\eqref{LEM011.1} and \eqref{LEM011.2}}\}.
\]
For $H\in\mathcal{H}$, let
\[
\mathcal{S}(H)=\{S\subseteq G\mid\text{there exists $f\in\Aut^{+}(\HH(G,S))$ satisfying \eqref{LEM011.2}}\}.
\]
Since Lemma~\ref{LEM004}\eqref{LEM004.4} implies that $|\mathcal{H}|\leq2^{(\log_2^2n)/4+3}$, we only need to prove
\begin{equation}\label{eq026}
|\mathcal{S}(H)|<2^{n-\frac{n}{3t}\log_2\left(\frac{4}{3}\right)+\log_2n+2\log_2t-2},
\end{equation}
for each $H\in\mathcal{H}$.

Fix some $H\in\mathcal{H}$ for the rest of the proof. Let $a=|H|\leq t$, $b=n/a$, and $H_1,\ldots,H_b$ be the right cosets of $H$ in $G$, where $H_1=H$. For each $i\in\{2,\ldots,b\}$, let
\[
\mathcal{S}_i=\{S\subseteq G\mid\text{there exists $f$ satisfying~\eqref{LEM011.2} with $f|_{(N_i)_{\rs}}\neq1$}\}.
\]
As $b\leq n/2=2^{\log_2n-1}$, to prove~\eqref{eq026}, it suffices to show
\begin{equation}\label{eq027}
|\mathcal{S}_i|<2^{n-\frac{n}{3t}\log_2\left(\frac{4}{3}\right)+2\log_2t-1}
\end{equation}
for each $i\in\{2,\ldots,b\}$.

From now on, fix some $i\in\{2,\ldots,b\}$. The right multiplication of any element in $H_i$ induces a permutation on $\{H_1,\ldots,H_b\}$ and hence on $\{1,\ldots,b\}$. Denote by $\sigma$ this induced permutation on $\{1,\ldots,b\}$, which does not depend on the choice of elements in $H_i$. Since $H$ is normal in $G$ and $H_i\neq H$, we have $j\neq j^\sigma$ for $j\in\{1,\ldots,b\}$.
Choose a subset $J\subseteq\{1,\ldots,b\}$ of maximal size satisfying $\{j,j^\sigma\}\cap
\{k,k^\sigma\}=\emptyset$ for all distinct $j,k\in J$. Then $j^{\sigma}\notin J$ for $j\in J$, that is, $J\cap J^{\sigma}=\emptyset$. If $|J|< b/3$, then
\[
\Bigg|\bigcup_{j\in J}\{j,j^\sigma,j^{\sigma^{-1}}\}\Bigg|\leq3|J|<b,
\]
and so there exists $\ell\in\{1,\ldots,b\}\setminus\bigcup_{j\in J}\{j,j^\sigma,j^{\sigma^{-1}}\}$. However, it follows that $\{j,j^\sigma\}\cap\{\ell,\ell^\sigma\}=\emptyset$ for each $j\in J$, which implies that the set $J\cup\{\ell\}$ satisfies $\{j,j^\sigma\}\cap\{k,k^\sigma\}=\emptyset$ for all distinct $j,k\in J\cup\{\ell\}$, contradicting the maximality of $J$. Therefore, $|J|\geq b/3$.

For each $S\subseteq G$, $x\in H_i$ and $j\in\{1,\ldots,b\}$, let $C(S,x,j)$ denote the set of common neighbours of $1_{\rs}$ and $x_{\rs}$ in $(H_{j^\sigma})_\ls$ within the graph $\HH(G,S)$. If $S\in\mathcal{S}_i$, then there exists $f\in\Aut^+(\HH(G,S))$ and distinct $x,y\in H_i$ such that $(1_\rs)^{f}=1_\rs$ and $(x_{\rs})^f=y_{\rs}$. Since $f$ fixes every $H$-orbit on $V(\HH(G,S))$, we obtain $$|C(S,x,j)|=|C(S,x,j)^f|=|C(S,y,j)|,$$ for each $j\in\{1,\ldots,b\}$. In particular, if $S\in\mathcal{S}_i$, then there exists a $2$-subset $\{x,y\}$ of $H_i$ such that $|C(S,x,j)|=|C(S,y,j)|$ for each $j\in J$. Denote
\begin{equation}\label{eq030}
\mathcal{S}_i(\{x,y\},j)=\{S\in \mathcal{S}_i\mid |C(S,x,j)|\equiv|C(S,y,j)|\ (\bmod\ 2)\},
\end{equation}
for each $\{x,y\}\in\binom{H_i}{2}$ and $j\in J$. Then
\[
\mathcal{S}_i\subseteq\bigcup_{\{x,y\}\in\binom{H_i}{2}}\Bigg(\bigcap_{j\in J}\mathcal{S}_i(\{x,y\},j)\Bigg).
\]
Since $|\binom{H_i}{2}|=\binom{a}{2}\leq\binom{t}{2}\le 2^{2\log_2t-1}$, to prove~\eqref{eq027}, it remains to show that
\begin{equation}\label{eq029}
\Bigg|\bigcap_{j\in J}\mathcal{S}_i(\{x,y\},j)\Bigg|\leq2^{n-\frac{n}{3t}\log_2\left(\frac{4}{3}\right)}.
\end{equation}

Fix some $2$-subset $\{x,y\}$ of $H_i$. For $g\in G$,
\begin{align*}
g_{\ls}\in C(S,x,j)&\Leftrightarrow\big(g\in S\big)\,\wedge\,\big(g\in Sx\big)\,\wedge\,\big(g\in H_{j^\sigma}\big)\\
&\Leftrightarrow\big(g\in Sx\cap H_{j^\sigma}\big)\,\wedge\,\big(g\in S\cap H_{j^\sigma}\big)\\
&\Leftrightarrow\big(g\in(S\cap H_j)x\big)\,\wedge\,\big(g\in S\cap H_{j^\sigma}\big).
\end{align*}
For each $S\subseteq G$, denote $S_j=S\cap H_j$. Then the above means
\begin{equation}\label{eq005}
C(S,x,j)=\big((S_j)x\cap S_{j^\sigma}\big)_\ls.
\end{equation}

Let $\mu_j$ be the number of pairs $(S_j,S_{j^\sigma})$ such that
\begin{equation}\label{eq008}
|(S_j)x\cap S_{j^\sigma}|
\equiv|(S_j)y\cap S_{j^\sigma}|\pmod{2}.
\end{equation}
We claim $\mu_j\le 3\cdot 2^{2a-2}$. To this end,
let $c_j$ be the number of subset $S_j$ of $H_j$ with $(S_j)x=(S_j)y$. Such an $S_j$ is a union of $\langle xy^{-1}\rangle$-orbits on $H_j$. The condition $x\neq y$ implies that $\langle xy^{-1}\rangle=\langle yx^{-1}\rangle$ is a nontrivial subgroup of $H$ and hence fixes setwise $H_j$. The semiregularity of $H$ implies
\[
c_j\leq2^{a/2}.
\]
Therefore, the number of pairs $(S_j,S_{j^\sigma})$ with $(S_j)x=(S_j)y$ is at most $c_j\cdot2^{a}$. Next we enumerate the pairs $(S_j,S_{j^\sigma})$ with $(S_j)x\neq(S_j)y$. For such a pair, both $(S_j)x\setminus(S_j)y$ and $(S_j)y\setminus(S_j)x$ are non-empty subset of $N_{j^\sigma}$. It follows from~\eqref{eq008} that $|\big((S_j)x\setminus(S_j)y\big)\cap S_{j^{\sigma}}|$ must have the same parity as $|\big((S_j)y\setminus(S_j)x\big)\cap S_{j^{\sigma}}|$. By Lemma~\ref{LEM002}, this gives $2^{a-1}$ choices for $S_{j^{\sigma}}$ for each fixed $S_j$. Therefore,
\[
\mu_j\leq c_j\cdot2^{a}+(2^a-c_j)\cdot2^{a-1}=
c_j\cdot2^{a-1}+2^{2a-1}\leq 2^{3a/2-1}+2^{2a-1}\leq3\cdot 2^{2a-2},
\]
as $a=|N|\geq 2$.

Since $J\cap J^{\sigma}=\emptyset$, we then conclude from~\eqref{eq030} and~\eqref{eq005} that
\[
\Bigg|\bigcap_{j\in J}\mathcal{S}_i(\{x,y\},j)\Bigg|\leq \Bigg(\prod_{j\in J}\mu_{j}\Bigg)\cdot2^{a(b-2|J|)}\leq(3\cdot 2^{2a-2})^{|J|}\cdot2^{n-2a|J|}=2^n\left(\frac{3}{4}\right)^{|J|}
\leq2^n\left(\frac{3}{4}\right)^{b/3},
\]
which leads to~\eqref{eq029}, as desired.
\end{proof}

We also need the following lemma for the first reduction.

\begin{lemma}\label{LEM012}
Let $G$ be a group of order $n$. The number of subsets $S$ of $G$ such that there exists a pair $(H,f)$ satisfying the following conditions~{\rm(a)}--{\rm(c)} is less than $2^{\frac{3}{4}n+\frac{\log_2^2n}{4}+2\log_2n+4}$.
\begin{enumerate}[{\rm(a)}]
\item\label{LEM012.1} $H$ is a nontrivial proper subgroup of $G$;
\item\label{LEM012.2} $f\in\Aut^+(\HH(G,S))$ stabilizes every $H$-orbit on $V(\HH(G,S))$;
\item\label{LEM012.3} $f$ induces a nontrivial permutation on $G_\rs$ fixing $(G\setminus H)_{\rs}$ pointwise.
\end{enumerate}
\end{lemma}

\begin{proof}
Since Lemma~\ref{LEM004}\eqref{LEM004.4} asserts that there are less than $2^{(\log_2^2{n})/4+3}$ subgroups of $G$, we only need to show that, for a fixed nontrivial $H<G$, the size of the following set is at most $2^{\frac{3}{4}n+2\log_2n+1}$:
\[
\mathcal{S}=\{S\subseteq G\mid\text{there exists $f\in\Aut^+(\HH(G,S))$ satisfying~\eqref{LEM012.2} and~\eqref{LEM012.3}}\}.
\]
We first estimate the size of
\begin{equation*}
\mathcal{T}:=\{S\subseteq G\mid\text{there exists $f\in\Aut^+(\HH(G,S))$ satisfying \eqref{LEM012.2} and \eqref{LEM012.3} such that $f|_{G_{\ls}}=1$}\}.
\end{equation*}
Let $S\in\mathcal{T}$ with a witness $f\in\Aut^+(\HH(G,S))$.
Since $f|_{H_{\rs}}\neq1$, there exist distinct $x,y\in H$ such that $(x_\rs)^f=y_{\rs}$. Write $\Gamma=\HH(G,S)$. Then since $(x_\rs)^{fy^{-1}}=(y_\rs)^{y^{-1}}=1_\rs$, we derive from $f|_{G_{\ls}}=1$ that
\[
S_\ls=N_\Gamma(1_\rs)=N_\Gamma\big((x_\rs)^{fy^{-1}}\big)=\big(N_\Gamma(x_\rs)\big)^{fy^{-1}}
=\big((Sx)_\ls\big)^{fy^{-1}}=\big((Sx)_\ls\big)^{y^{-1}}=(Sxy^{-1})_\ls.
\]
Hence $S=Sxy^{-1}$, which means that $S$ is a union of some left cosets of $\langle xy^{-1}\rangle$. Since $xy^{-1}\neq1$, it follows that, for a fixed non-identity element $xy^{-1}$, there are at most $2^{n/2}$ possibilities for $S$. Considering the choices for $xy^{-1}\in H$, we conclude that
\begin{equation}\label{eq040}
|\mathcal{T}|\leq(|H|-1)\cdot2^{\frac{n}{2}}<2^{\log_2|H|}\cdot2^{\frac{n}{2}}<2^{\frac{n}{2}+\log_2n}.
\end{equation}

Now assume $S\in\mathcal{S}\setminus\mathcal{T}$. Then there exists $f\in\Aut^{+}(\HH(G,S))$ satisfying~\eqref{PROP001.1} and~\eqref{PROP001.2} such that $f|_{G_{\ls}}\neq1$. In particular, there exist distinct elements $x,y\in G$ such that $(x_\ls)^f=y_\ls$ and so $(1_\ls)^{xfy^{-1}}=(x_\ls)^{fy^{-1}}=(y_\ls)^{y^{-1}}=1_\ls$. Since $(S^{-1})_\rs$ is the neighbourhood of $1_\ls$ in $\HH(G,S)$, it follows that $xfy^{-1}$ stabilizes $(S^{-1})_\rs$. Write $D=G\setminus H$. We claim that $xfy^{-1}$ also stabilizes $(Dx^{-1})_\rs$.
In fact, since~\eqref{LEM012.2} implies that $x_\ls$ and $y_\ls$ are in the same $H$-orbits, we have $(1_\ls)^{xH}=(x_\ls)^{H}=(y_\ls)^{H}=(1_\ls)^{yH}$. This together with the semiregularity of $G$ implies that $xH=yH$ and so $Hx^{-1}=Hy^{-1}$. Hence $Dx^{-1}=Dy^{-1}$, which leads to
\[
\big((Dx^{-1})_\rs\big)^{xfy^{-1}}=(D_\rs)^{fy^{-1}}=(D_\rs)^{y^{-1}}=(Dy^{-1})_\rs=(Dx^{-1})_\rs,
\]
as claimed. It follows that
\begin{align*}
(S^{-1}\cap Dx^{-1})_\rs&=\big((S^{-1}\cap Dx^{-1})_\rs\big)^{xfy^{-1}}\\
&=\big((S^{-1}x\cap D)_\rs\big)^{fy^{-1}}=\big((S^{-1}x\cap D)_\rs\big)^{y^{-1}}=\big((S^{-1}\cap Dx^{-1})xy^{-1}\big)_\rs.
\end{align*}
Consequently, $S^{-1}\cap Dx^{-1}=(S^{-1}\cap Dx^{-1})xy^{-1}$, which means that $S^{-1}\cap Dx^{-1}$ is a union of left cosets of $\langle xy^{-1}\rangle$. Since $xy^{-1}\neq1$, the condition $Dx^{-1}=Dy^{-1}$ indicates that $Dx^{-1}$ is a union of at most $|D|/2$ left cosets of $\langle xy^{-1}\rangle$. Therefore, for a fixed pair $(x,y)$ of elements in $G$ such that $xy^{-1}\in D$, there are at most $2^{|D|/2}$ choices for $S^{-1}\cap Dx^{-1}$ and hence at most
\[
2^{|H|}\cdot2^{\frac{|D|}{2}}=2^{|H|}\cdot2^{\frac{n-|H|}{2}}=2^{\frac{n+|H|}{2}}
\]
choices for $S$. Noting that $|H|\leq n/2$, we obtain
\[
|\mathcal{S}\setminus\mathcal{T}|\leq n^2\cdot2^{\frac{n+|H|}{2}}\leq2^{\frac{3}{4}n+2\log_2n}.
\]
Combining this with~\eqref{eq040}, we conclude that
\begin{equation*}
|\mathcal{S}|\leq|\mathcal{T}|+|\mathcal{S}\setminus\mathcal{T}|
<2^{\frac{n}{2}+\log_2n}+2^{\frac{3}{4}n+2\log_2n}
<2^{\frac{3}{4}n+2\log_2n+1},
\end{equation*}
completing the proof.
\end{proof}

We are now ready to state the first reduction result.

\begin{proposition}\label{PROP001}
Let $G$ be a group of order $n$. For each positive integer $t$, the number of subsets $S$ of $G$ such that there exists a pair $(H,f)$ satisfying the following conditions~{\rm(a)} and~{\rm(b)} is less than $2^{n-\frac{n}{3t}\log_2\left(\frac{4}{3}\right)+\frac{\log_2^2n}{4}+\log_2n+2\log_2t+5}$.
\begin{enumerate}[{\rm(a)}]
\item\label{PROP001.1} $H$ is a nontrivial proper normal subgroup of $G$ with $|H|\leq t$;
\item\label{PROP001.2} $f\in\Aut^+(\HH(G,S))\setminus G$ stabilizes every $H$-orbit on $V(\HH(G,S))$.
\end{enumerate}
\end{proposition}

\begin{proof}
Let $\mathcal{S}=\{S\subseteq G\mid \text{there exists a pair $(H,f)$ satisfying~\eqref{PROP001.1} and \eqref{PROP001.2}}\}$. For each $S\in\mathcal{S}$ and a witness pair $(H,f)$, the condition~\eqref{PROP001.2} implies that there exist unique elements $\alpha(S,H,f)$ and $\beta(S,H,f)$ in the right coset $Hf$ of $H$ in $\Aut^+(\HH(G,S))$ with $(1_{\rs})^{\alpha(S,H,f)}=1_{\rs}$ and $(1_{\ls})^{\beta(S,H,f)}=1_{\ls}$.

We first estimate the size of
\begin{equation}\label{eq031}
\mathcal{S}_1:=\{S\in\mathcal{S}\mid \text{there exists $(H,f)$ satisfying \eqref{PROP001.1} and \eqref{PROP001.2} such that $\alpha(S,H,f)|_{G_{\rs}}\neq1$}\}.
\end{equation}
Let $S\in\mathcal{S}_1$ and $(H,f)$ be a witness pair as in~\eqref{eq031}. Write $\alpha=\alpha(S,H,f)$ and $\beta=\beta(S,H,f)$. Then $\alpha$ stabilizes every $H$-orbit on $V(\HH(G,S))$ and $(1_{\rs})^{\alpha}=1_{\rs}$. Applying Lemmas~\ref{LEM011} and~\ref{LEM012} to the pair $(H,\alpha)$, we estimate the number of $S\in\mathcal{S}_1$ such that $\alpha|_{(G\setminus H)_{\rs}}\neq1$ or $\alpha|_{(G\setminus H)_{\rs}}=1$, respectively, and obtain
\begin{align}\label{eq002}
|\mathcal{S}_1|&<2^{n-\frac{n}{3t}\log_2\left(\frac{4}{3}\right)+\frac{\log_2^2n}{4}+\log_2n+2\log_2t+1}
+2^{\frac{3}{4}n+\frac{\log_2^2n}{4}+2\log_2n+4}.
\end{align}

With a similar argument we obtain the same upper bound as in~\eqref{eq002} for
\[
\mathcal{S}_2:=\{S\in\mathcal{S}\mid \text{there exists $(H,f)$ satisfying \eqref{PROP001.1} and \eqref{PROP001.2} such that $\beta(S,H,f)|_{G_{\ls}}\neq1$}\}.
\]
Now assume $S\in\mathcal{S}\setminus(\mathcal{S}_1\cup\mathcal{S}_2)$. Then we may take a pair $(H,f)$ satisfying~\eqref{PROP001.1} and~\eqref{PROP001.2} such that $\alpha|_{G_{\rs}}=1$ and $\beta|_{G_{\ls}}=1$, where $\alpha=\alpha(S,H,f)$ and $\beta=\beta(S,H,f)$. In particular, $\alpha\neq\beta^{-1}$ (the condition $\alpha\in Hf$ indicates $\alpha\neq1$). Since $\alpha\beta^{-1}\in H$, there exists a  non-identity element $\gamma\in H$ such that $(g_{\ls})^{\alpha\beta^{-1}}=(g\gamma)_{\ls}$ for each $g\in G$. Thus we deduce that
\[
(S_{\ls})^{\alpha}=(S_{\ls})^{\alpha\beta^{-1}}=(S\gamma)_{\ls}.
\]
As $S_{\ls}$ is the neighbourhood of $1_{\rs}$ in $\HH(G,S)$ while $\alpha\in\Aut^+(\HH(G,S))$ fixes $1_{\rs}$, it follows that $S_{\ls}=(S_{\ls})^{\alpha}=(S\gamma)_{\ls}$. Hence $S$ is a union of some left cosets of $\langle\gamma\rangle$. Therefore,
\[
|\mathcal{S}\setminus(\mathcal{S}_1\cup\mathcal{S}_2)|\leq
(|H|-1)\cdot2^{\frac{n}{2}}<2^{\frac{n}{2}+\log_2t}.
\]
Combining this with~\eqref{eq002}, we conclude that
\begin{align*}
|\mathcal{S}|&\leq|\mathcal{S}_1|+|\mathcal{S}_2|+|\mathcal{S}\setminus(\mathcal{S}_1\cup\mathcal{S}_2)|\\
&<2^{n-\frac{n}{3t}\log_2\left(\frac{4}{3}\right)+\frac{\log_2^2n}{4}+\log_2n+2\log_2t+2}
+2^{\frac{3}{4}n+\frac{\log_2^2n}{4}+2\log_2n+5}+2^{\frac{n}{2}+\log_2t}\\
&<2^{n-\frac{n}{3t}\log_2\left(\frac{4}{3}\right)+\frac{\log_2^2n}{4}+\log_2n+2\log_2t+5},
\end{align*}
where the last inequality follows with a computation.
\end{proof}

\subsection{Morris-Spiga-like reduction}\label{SUBSEC3.2}

In this section, we give two reduction results, namely, Propositions~\ref{PROP002} and~\ref{PROP003}, based on the following lemma.

\begin{lemma}\label{LEM013}
Let $\varepsilon\in(0,0.5)$, and let $n_\varepsilon$ be a positive integer such that~\eqref{eq032} holds for all $n\geq n_\varepsilon$.
Then for each semiregular subgroup $G$ of $\Sym(2n)$ with order $n\geq n_\varepsilon$, the number of subgroups $M$ of $\Sym(2n)$ satisfying the following conditions~\eqref{LEM013.1} and~\eqref{LEM013.2} is less than $2^{n^{1-\varepsilon}}$.
\begin{enumerate}[{\rm(a)}]
\item\label{LEM013.1} $G< M$ and $M$ has exactly two orbits;
\item\label{LEM013.2} $|M|\leq2^{n^{0.5-\varepsilon}+\log_2n}$ and $M$ is $(1+\log_2n)$-generated.
\end{enumerate}
\end{lemma}

\begin{proof}
Fix a semiregular subgroup $G$ of $\Sym(2n)$ with order $n\geq n_\varepsilon$. Then $G$ has exactly two orbits, acting regularly on each of them. Hence $G\leq\Sym(n)\times\Sym(n)$ and $\mathbf{C}_{\Sym(n)\times\Sym(n)}(G)=G\times G$. Since $|\Aut(G)|\leq2^{\log_2^2n}$ as Lemma~\ref{LEM004}\eqref{LEM004.3} asserts, it then follows that
\begin{equation}\label{eq003}
|\mathbf{N}_{\Sym(n)\times\Sym(n)}(G)|\leq|\mathbf{C}_{\Sym(n)\times\Sym(n)}(G)||\Aut(G)|\leq n^22^{\log_2^2n}=2^{\log_2^2n+2\log_2n}.
\end{equation}

Let $\mathcal{X}=\{[M]\mid M\text{ satisfies \eqref{LEM013.2}}\}$, where $[M]$ denotes the equivalence class of groups isomorphic to $M$. Write $d(n)=n^{0.5-\varepsilon}+\log_2n$. We conclude by Theorem~\ref{LEM005} that
\begin{equation*}
|\mathcal{X}|\leq2^{d(n)}\cdot2^{2((1+\log_2n)+1)d^2(n)}=2^{(4+2\log_2n)d^2(n)+d(n)}.
\end{equation*}

Given a group $X$ with a pair $(H,K)$ of subgroups and a group $Y$ with a pair $(P,Q)$ of subgroups, define $(X,H,K)\approx(Y,P,Q)$ if there is a group isomorphism $\varphi\colon X\to Y$ such that $\varphi(H)=P$ and $\varphi(K)=Q$. Clearly, $\approx$ is an equivalence relation. Let
\[
\mathcal{T}=\{(X,H,K)\mid[X]\in\mathcal{X},\,|X\,{:}\,H|=|X\,{:}\,K|=n,\,\mathrm{Core}_X(H)\cap\mathrm{Core}_X(K)=1\}
\]
and let $\mathcal T/\approx$ denote the $\approx$-equivalence classes in $\mathcal{T}$.
For each group $X$ with $[X]\in\mathcal{X}$, we derive from Lemma~\ref{LEM004}\eqref{LEM004.1} that there are at most
\[
|X|^{2\log_2(|X|/n)}=2^{2(\log_2|X|)(\log_2|X|-\log_2n)}\leq
2^{2d(n)(d(n)-\log_2n)}
\]
choices for a pair $(H,K)$ of subgroups of index $n$ in $X$. Hence
\begin{align*}
|\mathcal{T}/\approx|\leq|\mathcal{X}|\cdot2^{2d(n)(d(n)-\log_2n)}&\leq
2^{(4+2\log_2n)d^2(n)+d(n)}\cdot2^{2d(n)(d(n)-\log_2n)}\\
&=2^{(6+2\log_2n)d^2(n)+(1-2\log_2n)d(n)}.
\end{align*}

Let $\sim$ be the equivalence relation of conjugation of subgroups in $\Sym(n)\times\Sym(n)$.
Note that each triple $(X,H,K)\in\mathcal{T}$ gives rise to a subgroup $T(X,H,K)$  of $\Sym(n)\times\Sym(n)$ via the right multiplication action of $X$ on $[X\,{:}\,H]$ and $[X\,{:}\,K]$, and that triples in the same $\approx$-equivalence class give subgroups of $\Sym(n)\times\Sym(n)$ in the same $\sim$-equivalence class.

Let
\[
\mathcal{M}=\{M\leq\Sym(2n)\mid M\text{ satisfies \eqref{LEM013.1} and \eqref{LEM013.2}}\}.
\]
Since each $M\in\mathcal{M}$ lies in the same $\sim$-equivalence class with $T(X,H,K)$ for some $(X,H,K)\in\mathcal{T}$, we conclude that
\[
|\mathcal{M}/\sim|\leq|\mathcal{T}/\approx|\leq2^{(6+2\log_2n)d^2(n)+(1-2\log_2n)d(n)}.
\]

Suppose for a contradiction that $|\mathcal{M}|\geq2^{n^{1-\varepsilon}}$. Then by Pigeonhole Principal, there are pairwise distinct $M_1,\ldots,M_t\in\mathcal{M}$ in the same $\sim$-equivalence class with
\begin{equation}\label{eq001}
t\geq\frac{2^{n^{1-\varepsilon}}}{2^{(6+2\log_2n)d^2(n)+(1-2\log_2n)d(n)}}
>2^{(\log_2n)d(n)+\log_2^2n+2\log_2n},
\end{equation}
where the $>$ sign in~\eqref{eq001} follows from~\eqref{eq032}. We deduce that for each $i\in\{1,\ldots,t\}$ we have $M_1=x_i^{-1}M_ix_i$ for some $x_i\in\Sym(n)\times\Sym(n)$ (note that $x_i\neq x_j$ if $i\neq j$, as $M_1,\ldots,M_t$ are pairwise distinct). In particular, $x_1^{-1}Gx_1,\ldots,x_t^{-1}Gx_t$ are all subgroups of order $n$ in $M_1$, as $G<M_i$ for all $i\in\{1,\ldots,t\}$. However, by Lemma~\ref{LEM004}\eqref{LEM004.1}, $M_1$ has at most
\[
|M_1|^{\log_2n}\leq2^{(\log_2n)d(n)}
\]
subgroups of order $n$. Then again by Pigeonhole Principal, we deduce from~\eqref{eq001} that there exists $\{i_1\ldots,i_s\}\subseteq\{1,\ldots,t\}$ with $s>2^{\log_2^2n+2\log_2n}$ and $x_{i_1}^{-1}Gx_{i_1}=\cdots=x_{i_s}^{-1}Gx_{i_s}$. This shows the existence of $s$ distinct elements $x_{i_1}x^{-1}_{i_1},\ldots,x_{i_1}x^{-1}_{i_s}$ in $\Sym(n)\times\Sym(n)$ normalizing $G$, contradicting~\eqref{eq003}. Thus the proof is complete.
\end{proof}

Recall Definition~\ref{DEF003}, noting that for each exceptional pair $(S,X)$ there exists a minimally exceptional pair $(S,M)$ such that $M\leq X$.

\begin{proposition}\label{PROP002}
Let $\varepsilon\in(0,0.5)$, and let $n_\varepsilon$ be a positive integer such that~\eqref{eq032} holds for all $n\geq n_\varepsilon$.
Then for each group $G$ of order $n\geq n_\varepsilon$, the number of subsets $S$ of $G$ such that there exists an exceptional pair $(S,X)$ with $|X|\leq2^{n^{0.5-\varepsilon}+\log_2n}$ is less than $2^{\frac{3}{4}n+n^{1-\varepsilon}}$.
\end{proposition}

\begin{proof}
Fix a group $G$ with order $n\geq n_{\varepsilon}$. Suppose, for a contradiction, that there are $t$ subsets $S_1,\ldots,S_t$ of $G$, where $t>2^{\frac{3}{4}n+n^{1-\varepsilon}}$, such that there exists an exceptional pair $(S_i,X_i)$ with $|X_i|\leq2^{n^{0.5-\varepsilon}+\log_2n}$ for each $i\in\{1,\ldots,t\}$. Let $V$ be the (same) vertex set of
\[
\HH(G,S_1),\ldots,\HH(G,S_t),
\]
and let $U$ and $W$ be the two orbits of $G$ on $V$.

The definition of exceptional pair implies that $X_i\leq\Aut^{+}(\HH(G,S_i))$ and $X_i$ has exactly two orbits $U$ and $W$ on $V$. For each $i\in\{1,\ldots,t\}$, let $M_i$ be a subgroup  of $X_i$  such that $(S_i,M_i)$ is a minimally exceptional pair. Since $G$ is $(\log_2n)$-generated and is a maximal subgroup of $M_i$, it follows that $M_i$ is $(1+\log_2n)$-generated. Hence each $M_i$ satisfies both~\eqref{LEM013.1} and~\eqref{LEM013.2} of Lemma~\ref{LEM013}, and so there are at most $2^{n^{1-\varepsilon}}$ distinct ones among $M_1,\ldots,M_t$. By Pigeonhole Principle, there exists $\{i_1\ldots,i_s\}\subseteq\{1,\ldots,t\}$ with $s\geq t/2^{n^{1-\varepsilon}}>2^{3n/4}$ such that
\[
M_{i_1}=\cdots=M_{i_s}.
\]
Since $M_{i_1}$ is not semiregular on $V$, Lemma~\ref{LEM009} shows that there are at most $2^{3n/4}$ bipartite graphs with bipartition $\{U,W\}$ whose automorphism group contains $M_{i_1}$. This contradicts the condition that $M_{i_1}=M_{i_j}\leq\Aut^{+}(\HH(G,S_{i_j}))$ for each $j\in\{i_1,\ldots,i_s\}$.
\end{proof}

\begin{proposition}\label{PROP003}
Let $\varepsilon\in(0,0.5)$, and let $n_\varepsilon$ be a positive integer such that for all $n\geq n_\varepsilon$,
\begin{equation}\label{eq033}
\log^2_2n<n^{0.5-\varepsilon}.
\end{equation}
Then for each group $G$ of order $n\geq n_\varepsilon$, the number of subsets $S$ of $G$ such that there exists a minimally exceptional pair $(S,M)$ satisfying the following conditions~\eqref{PROP003.1} and~\eqref{PROP003.2} is less than $2^{n-\frac{n}{4\log_2n}\log_2\left(\frac{e}{2}\right)+\frac{\log^2_2n}{4}+
\frac{1}{2}\log_2\left(\frac{n}{4\log_2n}\right)+\log_2(24)}$.
\begin{enumerate}[{\rm(a)}]
\item\label{PROP003.1} $|M|>2^{n^{0.5-\varepsilon}+\log_2n}$;
\item\label{PROP003.2} $|\mathrm{Core}_M(G)|>8\log_2n$.
\end{enumerate}
\end{proposition}

\begin{proof}
Fix a group $G$ of order $n\geq n_{\varepsilon}$. Let $\{U,W\}$ be the bipartition for the Haar graphs of $G$. Let
\[
\mathcal{M}=\{M\leq\Sym(U)\times\Sym{(W)}\mid\text{$M$ satisfies \eqref{LEM013.1} and \eqref{LEM013.2}, $G$ is maximal in $M$}\}.
\]
The  proposition is an estimation on the number of $S\subseteq G$ such that there exists $M\in\mathcal{M}$ with $M\leq\Aut^{+}(\HH(G,S))$. By Lemma~\ref{LEM009}, we only need to prove
\begin{equation}\label{eq034}
|\mathcal{M}|<2^{\frac{n}{4}-\frac{n}{4\log_2n}\log_2\left(\frac{e}{2}\right)+\frac{\log^2_2n}{4}+\frac{1}{2}\log_2\left(\frac{n}{\log_2n}\right)+\log_2(24)}.
\end{equation}

Let $M\in\mathcal{M}$, and fix some $u\in U$. Then, by condition~\eqref{PROP003.1} and~\eqref{eq033}, the stabilizer $M_u$ satisfies
\[
|M_u|=\frac{|M|}{n}>\frac{2^{n^{0.5-\varepsilon}+\log_2n}}{n}=2^{n^{0.5-\varepsilon}}>2^{\log^2_2n}. \]
Let $C=\mathrm{Core}_M(G)$. Since $C$ is a subgroup of the semiregular group $G$, we have
\[
\mathbf{C}_{\Sym(2n)}(C)\cong C\wr\Sym(2n/|C|).
\]
By the inequality $x!<3\sqrt{x}(x/e)^{x}$ (Stirling's formula), we obtain
\begin{align}\label{eq035}
|\mathbf{C}_{\Sym(2n)}(C)|=|C|^{\frac{2n}{|C|}}\cdot\left(\frac{2n}{|C|}\right)!
&<|C|^{\frac{2n}{|C|}}\cdot3\left(\frac{2n}{|C|}\right)^{\frac{1}{2}}\left(\frac{2n}{e|C|}\right)^{\frac{2n}{|C|}}\nonumber\\
&=3\left(\frac{2n}{|C|}\right)^{\frac{1}{2}}\left(\frac{2n}{e}\right)^{\frac{2n}{|C|}}
<3\left(\frac{n}{4\log_2n}\right)^{\frac{1}{2}}\left(\frac{2n}{e}\right)^{\frac{n}{4\log_2n}},
\end{align}
where in the last inequality we have used condition~\eqref{PROP003.2}.

Since $M$ normalizes $C$, the group $M_u$ acts by conjugation as a group of automorphisms on $C$. Moreover, by Lemma~\ref{LEM004}\eqref{LEM004.3},
\[
\Aut(C)\leq2^{\log^2_2|C|}\leq
2^{\log^2_2n}<|M_u|.
\]
Hence the conjugation action of $M_u$ on $C$ is not faithful, and so there exists  $g\in\mathbf{C}_{M_u}(C)$ with $g\ne 1$. In particular, $g\notin G$, as $G\cap M_u=1$. Then it follows from the maximality of $G$ in $M$ that $M=\langle G,g\rangle$. Accordingly, $M$ is determined by a non-identity element of $\mathbf{C}_{\Sym(2n)}(C)$.
Considering the choices for the subgroup $C$ of $G$, we conclude by Lemma~\ref{LEM004}\eqref{LEM004.4} and~\eqref{eq035} that
\[
|\mathcal{M}|<2^{\frac{\log_2^2n}{4}+3}\cdot 3\left(\frac{n}{4\log_2n}\right)^{\frac{1}{2}}\left(\frac{2n}{e}\right)^{\frac{n}{4\log_2n}},
\]
This is equivalent to
\begin{align*}
\log_2|\mathcal{M}|&<\frac{\log^2_2n}{4}+3+\log_23+\frac{1}{2}\log_2\left(\frac{n}{4\log_2n}\right)
+\frac{n}{4\log_2n}\log_2\left(\frac{2n}{e}\right)\\
&= \frac{n}{4}-\frac{n}{4\log_2n}\log_2\left(\frac{e}{2}\right)+\frac{\log^2_2n}{4}+\frac{1}{2}\log_2\left(\frac{n}{4\log_2n}\right)+\log_2(24),
\end{align*}
proving~\eqref{eq034}, as required.
\end{proof}

\subsection{Critical pairs}\label{SUBSEC3.3}

Recall Definition~\ref{DEF003}. For each subset $S$ of a finite group $G$ such that $\Aut(\HH(G,S))>G$, there exists a minimally exceptional pair with respect to $S$.
As stated at the beginning of Section~\ref{SEC3}, to prove Theorem~\ref{THM001}, the primary task is to estimate the number of subsets $S$ of $G$ such that there exists a ``large'' subgroup $M$ of $\Aut^+(\HH(G,S))$ with $G$  maximal and core-free in $M$. We make this precise in the following definition.


\begin{definition}\label{DEF004}
Let $G$ be a finite group of order $n$, and let $\varepsilon\in\left(0,0.1\right]$. A minimally exceptional pair $(S,M)$ of $G$ is said to be \emph{$\varepsilon$-critical} if it satisfies
\begin{enumerate}
\item[(C1)] $|M|>2^{n^{0.5-\varepsilon}+\log_2n}$, or equivalently, $|M_{1_{\rs}}|=|M_{1_{\ls}}|>2^{n^{0.5-\varepsilon}}$;
\item[(C2)] $\mathrm{Core}_M(G)=1$.
\end{enumerate}
Clearly, an $\varepsilon$-critical pair is also an $\varepsilon'$-critical pair for each $\varepsilon'\geq\varepsilon$. We omit the label $\varepsilon$ when $\varepsilon=0.1$.
Denote $\mathcal{Z}(G,\varepsilon)=\{S\subseteq G\mid\text{there exists an $\varepsilon$-critical pair $(S,M)$ of $G$}\}$.
\end{definition}
The following is the desired estimation on the size of $\mathcal{Z}(G,\varepsilon)$.

\begin{proposition}\label{PROP004}
Let $G$ be a group of order $n\geq2^{57}$. Then for each $\varepsilon\in\left(0,0.1\right]$,
\[
|\mathcal{Z}(G,\varepsilon)|<2^{n-\frac{n^{0.5-\varepsilon}}{8\log_2^2n}+\frac{\log_2^2n}{2}+9}.
\]
\end{proposition}

Before proving Proposition~\ref{PROP004} in Section~\ref{SEC4}, here we briefly outline the main idea underlying its technical proof. For an $\varepsilon$-critical pair $(S,M)$, since $G$ is maximal in $M$ and $\mathrm{Core}_M(G)=1$, the group $M$ acts faithfully and primitively by right multiplication on the set $[M\,{:}\,G]$ of right cosets of $G$ in $M$. Moreover, as $M=M_{1_{\rs}}G=M_{1_{\ls}}G$ and $M_{1_{\rs}}\cap G=M_{1_{\ls}}\cap G=1$, the actions of $M_{1_{\rs}}$ and $M_{1_{\ls}}$ on $[M\,{:}\,G]$ are regular. Consequently, we obtain a primitive permutation group $M$ on $[M\,{:}\,G]$ with regular subgroups $M_{1_\rs}$ and $M_{1_\ls}$ satisfying
\begin{equation}\label{eq036}
|M_{1_{\rs}}|=|M_{1_{\ls}}|>2^{|G|^{0.5-\varepsilon}}.
\end{equation}
We follow the division in~\cite{Praeger1997} of primitive groups into eight types, namely, $\UPHS$, $\UPHC$, $\UPAS$, $\UPPA$, $\UPCD$, $\UPHA$, $\UPSD$ and $\UPTW$. For each type $\mathcal{T}$, let $\mathcal{Z}_{\mathcal{T}}(G,\varepsilon)$ be the set of $S\in\mathcal{Z}(G,\varepsilon)$ such that there exists an $\varepsilon$-critical pair $(S,M)$ with the action of $M$ on $[M\,{:}\,G]$ primitive of $\mathcal{T}$ type. We write $\mathcal{Z}_{\mathcal{T}}(G,0.1)$ as $\mathcal{Z}_{\mathcal{T}}(G)$ for simplicity. Since $\mathcal{Z}_{\mathcal{T}}(G,\varepsilon)\subseteq\mathcal{Z}_{\mathcal{T}}(G)$ for each $\varepsilon\in\left(0,0.1\right]$, we have
\[
\mathcal{Z}(G,\varepsilon)\subseteq
\mathcal{Z}_{\UPHS}(G)\cup\mathcal{Z}_{\UPHC}(G)\cup
\mathcal{Z}_{\UPAS}(G)\cup\mathcal{Z}_{\UPPA}(G)\cup
\mathcal{Z}_{\UPCD}(G)\cup\mathcal{Z}_{\UPHA}(G,\varepsilon)\cup
\mathcal{Z}_{\UPSD}(G,\varepsilon)\cup\mathcal{Z}_{\UPTW}(G,\varepsilon).
\]
In Section~\ref{SEC4}, we estimate $|\mathcal{Z}_{\mathcal{T}}(G)|$ for $\UPHS$, $\UPHC$, $\UPAS$, $\UPPA$, $\UPCD$ types and estimate $|\mathcal{Z}_{\mathcal{T}}(G,\varepsilon)|$ for $\UPHA$, $\UPSD$ and $\UPTW$ types. More precisely, we establish Propositions~\ref{PROP005},~\ref{PROP006},~\ref{PROP007},~\ref{PROP008} and~\ref{PROP009}, which respectively imply the following upper bounds under the assumption $|G|=n\geq2^{57}$:
\begin{align}
&\mathcal{Z}_{\UPHS}(G)\cup\mathcal{Z}_{\UPHC}(G)\cup\mathcal{Z}_{\UPPA}(G)=\emptyset,\nonumber\\
&|\mathcal{Z}_{\UPAS}(G)|<2^{3(\log_2n)+75},\label{eq061}\\
&|\mathcal{Z}_{\UPCD}(G)|<
2^{\frac{3}{4}n+2\log_2^4n+\log_2^3n+1702\log_2^2n+2\log_2n},\label{eq062}\\
&|\mathcal{Z}_{\UPHA}(G,\varepsilon)\cup\mathcal{Z}_{\UPSD}(G,\varepsilon)\cup\mathcal{Z}_{\UPTW}(G,\varepsilon)|
<2^{n-\frac{n^{0.5-\varepsilon}}{8\log_2^2n}+\frac{\log_2^2n}{2}+7}.\label{eq063}
\end{align}
Since the right-hand sides of~\eqref{eq061} and~\eqref{eq062} are both less than the right-hand side of~\eqref{eq063} for each $\varepsilon\in\left(0,0.1\right]$, we then derive
\begin{align*}
|\mathcal{Z}(G)|&\leq|\mathcal{Z}_{\UPAS}(G)|+|\mathcal{Z}_{\UPCD}(G)|+
|\mathcal{Z}_{\UPHA}(G,\varepsilon)\cup\mathcal{Z}_{\UPSD}(G,\varepsilon)\cup\mathcal{Z}_{\UPTW}(G,\varepsilon)|\\
&<3|\mathcal{Z}_{\UPHA}(G,\varepsilon)\cup\mathcal{Z}_{\UPSD}(G,\varepsilon)\cup\mathcal{Z}_{\UPTW}(G,\varepsilon)|
<2^{n-\frac{n^{0.5-\varepsilon}}{8\log_2^2n}+\frac{\log_2^2n}{2}+9},
\end{align*}
as Proposition~\ref{PROP004} asserts.

%

\section{Primitive permutation groups with a large regular subgroup}\label{SEC4}

In this section, we estimate $|\mathcal{Z}_{\mathcal{T}}(G,\varepsilon)|$ for various primitive types $\mathcal{T}$. The analysis depends on the structure of primitive groups with a ``large'' regular subgroup.

\subsection{Estimate  \boldmath{$|\mathcal{Z}_{\UPHS}(G)\cup\mathcal{Z}_{\UPHC}(G)|$}}

\begin{lemma}\label{LAM014}
Let $M$ be a primitive group of $\UPHS$ or $\UPHC$ type with stabilizer $G$ such that $|G|>132$. Then $M$ has no regular subgroup with order greater than $2^{|G|^{0.4}}$.
\end{lemma}

\begin{proof}
In both of these cases, the socle of $M$ has the form $H\times K$, where $H$ and $K$ are normal regular subgroups of $M$ with $H\cong K\cong T^\ell$ for some nonabelian simple group $T$ and some integer $\ell\geq1$. Then the stabilizer $G\gtrsim\Inn(H)\cong H$. Suppose on the contrary that $M$ has a regular subgroup $L$ with $|L|>2^{|G|^{0.4}}$. Then
\[
|G|\geq|\Inn(H)|=|H|=|L|>2^{|G|^{0.4}},
\]
which is impossible as $|G|>132$.
\end{proof}

The following proposition follows immediately from Lemma~\ref{LAM014} and~\eqref{eq036}.

\begin{proposition}\label{PROP005}
Let $G$ be finite group of order $n>132$. Then $\mathcal{Z}_{\UPHS}(G)\cup\mathcal{Z}_{\UPHC}(G)=\emptyset$.
\end{proposition}

\subsection{Estimate \boldmath{$|\mathcal{Z}_{\UPAS}(G)|$}}

\begin{lemma}\label{LAM015}
Let $M$ be a primitive permutation group of $\UPAS$ type with socle $N$ and the stabilizer $G$ with $|G|\geq2^{11}$. Suppose that $M$ has a regular subgroup $L$ with $|L|>2^{|G|^{0.4}}$. Then one of the following occurs for some odd prime $p$.
\begin{enumerate}[{\rm(a)}]
\item\label{LEM015.1} $N=\Alt(p)$, $N\cap G=p\cdot\frac{p-1}{2}$, and $L=\Sym(p-2)$ or $\Alt(p-2)\times C_2$;
\item\label{LEM015.2} $N=\Alt(p+1)$, $N\cap G=\PSL_2(p)$, and $L=\Sym(p-2)$ or $\Alt(p-2)\times C_2$;
\item\label{LEM015.3} $N=\Alt(p^2+1)$ with $p\equiv3\ (\bmod\ 4)$, $N\cap G=\PSL_2(p^2).2$, and $L=\Alt(p^2-2)$.
\end{enumerate}
\end{lemma}

\begin{proof}
The almost simple primitive groups with a regular subgroup are classified in~\cite{LPS2010}. By a computation for each case in~\cite[Table~16.1--16.3]{LPS2010}, we obtain that when $|G|\geq2^{11}$, the inequality $|L|>2^{|G|^{0.4}}\geq2^{|N\cap G|^{0.4}}$ holds only in one of the three cases stated in the lemma.
\end{proof}

\begin{proposition}\label{PROP006}
Let $G$ be finite group of order $n\geq2^{11}$. Then $|\mathcal{Z}_{\UPAS}(G)|<2^{3(\log_2n)+75}$.
\end{proposition}

\begin{proof}
Let $S\in\mathcal{Z}_{\UPAS}(G)$. Then there is a critical pair $(S,M)$ such that the action of $M$ on $[M\,{:}\,G]$ is primitive of  AS type. Let $N$ be the socle of $M$, and let $L=M_{1_\epsilon}$, where $\epsilon\in\{\rs,\ls\}$. Since $|L|>2^{n^{0.4}}$, the triple $(G,N,L)$ is described in Lemma~\ref{LAM015}. In particular, $N=\Alt(m)$ with $m\in\{p,p+1,p^2+1\}$ for some odd prime $p$, and
\begin{equation}\label{eq038}
\Alt(\{1,\ldots,m\}\setminus\Delta)\leq L\leq\Sym(\Delta)\times
\Sym(\{1,\ldots,m\}\setminus\Delta),
\end{equation}
for some subset $\Delta$ of $\{1,\ldots,m\}$ of size $2$ or $3$. As $n\ge 2^{11}$ by hypothesis, we get $|L|>2^{n^{0.4}}\geq2^{2^{4.4}}$. This implies $m>7$. Considering the choices for subsets $\Delta$ and for the groups $L$ satisfying~\eqref{eq038}, we obtain that the number of the choices for $L$ is at most
\[
\max\left\{\binom{m}{2}\cdot5,\,\binom{m}{3}\cdot16\right\}
=\binom{m}{3}\cdot16=\frac{8m(m-1)(m-2)}{3}.
\]
This implies that, for each $\epsilon\in\{\rs,\ls\}$, there are less than $8m(m-1)(m-2)/3$ choices $M_\epsilon$.
As $M\in\{\mathrm{Alt}(m),\mathrm{Sym}(m)\}$, it follows that there are at most  $2(8m(m-1)(m-2)/3)^2$ choices for the triple $(M,M_\rs,M_\ls)$. Since $\Alt(m)$ acts $6$-transitively on $\{1,\ldots,m\}$, the group $\Alt(\{4,\ldots,m\})$ has exactly
\[
\binom{3}{0}+\binom{3}{1}\cdot3+\binom{3}{2}\cdot3\cdot2+3\cdot2\cdot1=34
\]
orbits on the set of triples of pairwise distinct elements in $\{1,\ldots,m\}$. Since $\Alt(\{1,\ldots,m\})\leq M\leq\Sym(\{1,\ldots,m\})$ and $L\geq \Alt(\{1,\ldots,m\}\setminus\Delta)\gtrsim\Alt(\{4,\ldots,m\})$, there are at most $68$ double cosets of $M_{1_\rs}$ and $M_{1_\ls}$ in $M$. Therefore, in view of Lemma~\ref{LEM009}, there are less than
\[
2(8m(m-1)(m-2)/3)^2\cdot2^{68}<2^{72}(m(m-1)(m-2))^2
\]
choices for $S\in\mathcal{Z}^{\UPAS}(G)$. Note that, for each case arising in Lemma~\ref{LAM015}, we have
\[
n=\frac{|M|}{|M_{1_{\rs}}|}\geq\frac{|\Alt(m)|}{|\Sym(m-2)|}=\frac{m(m-1)}{2}.
\]
We conclude that
\[
|\mathcal{Z}^{\UPAS}(G)|<2^{72}(m(m-1)(m-2))^2<2^{72}(m(m-1))^3\leq2^{3(\log_2n)+75}, \]
as the proposition asserts.
\end{proof}

\subsection{Estimate \boldmath{$|\mathcal{Z}_{\UPPA}(G)|$}}\label{sec:43}

Let $M$ be a primitive permutation group of $\UPPA$ type with socle $N$ and point stabilizer $G$. Then we have $M\leq H\wr\Sym(\kappa)$ endowed with its natural wreath product action on $\Delta^{\kappa}$, where $H$ is a primitive group of $\UPAS$ type on $\Delta$ and $\kappa\geq2$. Let $T$ be the socle of $H$, and let $\delta\in\Delta$. Then $N=T^{\kappa}$ and $T_\delta\neq1$. Recall that $P(X)$ denotes the minimal index of proper subgroups of a finite group $X$.

\begin{lemma}\label{LAM016}
Let $H$ be a finite almost simple group with socle $T$, and let $B$ be a core-free subgroup of $H$ such that $H=KB$ for some core-free subgroup $K$ of $H$. Then either $|\Out(T)|\leq |T\cap B|$ or $T\in\{\PSL_2(9),\PSL_3(4)\}$. Moreover, if $B$ is maximal in $H$, then $|\Out(T)|\leq |T\cap B|$.
\end{lemma}

\begin{proof}
Suppose that $|\Out(T)|>|T\cap B|$. Then we derive from $H=KB$ that
\[
|\Out(T)|^2>|\Out(T)||T\cap B|\geq\frac{|H||T\cap B|}{|T|}=\frac{|H||B|}{|TB|}\geq|B|\geq\frac{|H|}{|K|}\geq\frac{|T|}{|T\cap K|}\geq|P(T)|.
\]
By~\cite[Table~4]{GMPS2015}, this implies $T\cong \PSL_2(9)$, $\PSL_2(27)$, $\PSL_3(4)$ or $\PSL_3(16)$. A computation in \textsc{Magma}~\cite{BCP1997} for these candidates of $T$ shows that there is no core-free subgroup $K$ of $H$ with $H=KB$ and $|\Out(T)|>|T\cap B|$ for $T\in\{\PSL_2(27),\PSL_3(16)\}$, and that $|\Out(T)|\leq |T\cap B|$ if $B$ is maximal.
\end{proof}

Before proceeding we need an elementary fact as follows.

\begin{lemma}\label{LEM028}
Let $\Alt(m)\leq X\leq\Sym(m)$ for some $m\geq7$, and let $Y$ be a maximal subgroup of $X$. Then $|Y\cap\Alt(m)|\geq3m$.
\end{lemma}

\begin{proof}
If $Y$ is intransitive on $\{1,\ldots,m\}$, then $Y\cap \mathrm{Alt}(m)\cong(\mathrm{Sym}(s)\times\mathrm{Sym}(m-s))\cap\mathrm{Alt}(m)$, for some $s\in \{1,\ldots,m-1\}$, and $|Y\cap\Alt(m)|=s!(m-s)!/2\ge 3m$.

If $Y$ is transitive and imprimitive on $\{1,\ldots,m\}$, then $Y\cap \mathrm{Alt}(m)\cong(\mathrm{Sym}(s)\wr\mathrm{Sym}(m/s))\cap\mathrm{Alt}(m)$, for some divisor $s$ of $m$ with $1<s<m$, and $|Y\cap\Alt(m)|=s!^{m/s}(m/s)!/2\ge 3m$.

Now assume $Y$ is primitive on $\{1,\ldots,m\}$. If the stabilizer $Y_1$ of the point $1$ is nonabelian, then $|Y_1|\geq6$ and hence $|Y\cap\mathrm{Alt}(m)|\geq 3m$. If $Y_1$ is abelian, then the primitivity of $Y$ implies that $Y$ is solvable (see~\cite[Exercise~3.4.7]{DM1996}). Therefore, $Y\cap\mathrm{Alt}(m)=\AGL_s(p)$ for some integer $s$ and prime $p$ such that $m=p^s$, and so $|Y\cap\mathrm{Alt}(m)|\geq m(m-1)/2\geq3m$.
\end{proof}

A subgroup $Y$ of a group $X$ is said to be $\max^+$ if $Y$ is maximal and core-free in $X$, and is said to be $\max^-$ if $Y$ is maximal among core-free subgroups of $X$. Clearly, $\max^+$ subgroups of $X$ are necessarily $\max^-$ in $X$, but the converse is not true.

\begin{lemma}\label{LAM017}
Let $M$ be a primitive permutation group of $\UPPA$ type with the notation at the beginning of Section~$\ref{sec:43}$. Suppose that $M$ has a regular subgroup $L$ of order $|L|>2^{|G|^{0.4}}$. Then $|T\,{:}\,T_\delta|^\kappa>2^{\left(\kappa|T_\delta|^{\kappa}\right)^{0.4}}$, and there exists a $\max^{-}$ subgroup $K$ of $H$ acting transitively on $\Delta$ with $|K|>2^{\frac{1}{\kappa}\left(\kappa|T_\delta|^{\kappa}\right)^{0.4}}$.
\end{lemma}

\begin{proof}
Since $GN/N$ is a transitive subgroup of $\Sym(\kappa)$, we have $|GN/N|\geq\kappa$, and so $|G|=|GN/N||G\cap N|\geq\kappa|T_\delta|^\kappa$. It follows from $|L|>2^{|G|^{0.4}}$ that
\[
|T:T_\delta|^\kappa=|L|>2^{|G|^{0.4}}\geq 2^{\left(\kappa|T_\delta|^{\kappa}\right)^{0.4}}.
\]
Since $T_\delta\neq1$, the group $L$ contains no simple direct factor of $N$. Consider the set $\mathcal{K}$ of core-free subgroups of $H$ that are transitive on $\Delta$. By~\cite[Theorem~1(i)]{LPS2000}, $\mathcal{K}\neq\emptyset$. Take a maximal one $K$ in $\mathcal{K}$. Then $K$ is $\max^{-}$ in $H$. Moreover, we derive that
\[
|K|\geq|\Delta|=|T:T_\delta|>2^{\frac{1}{\kappa}\left(\kappa|T_\delta|^{\kappa}\right)^{0.4}},
\]
completing the proof.
\end{proof}

Based on Lemmas~\ref{LAM016}--\ref{LAM017}, we now establish the main result of this section.

\begin{proposition}\label{PROP007}
Let $G$ be a finite group with $|G|\geq2^{57}$. Then $\mathcal{Z}_{\UPPA}(G)=\emptyset$.
\end{proposition}

\begin{proof}
Suppose for contradiction that $\mathcal{Z}_{\UPPA}(G)\neq\emptyset$. Then there exists a critical pair $(S,M)$ of $G$ such that $M$ is primitive of PA type on $[M\,{:}\,G]$ with a regular subgroup $L$ of order $|L|>2^{|G|^{0.4}}$ (recall~\eqref{eq036}). Adopt the notation at the beginning of Section~\ref{sec:43}. By Lemma~\ref{LAM017}, there exists a $\max^{-}$ subgroup $K$ of $H$ acting transitively on $\Delta$ such that
\begin{equation}\label{eq015}
|K|>2^{\frac{1}{\kappa}\left(\kappa|T_\delta|^{\kappa}\right)^{0.4}}.
\end{equation}
Since $K$ is a transitive subgroup of the primitive group $H$ on $\Delta$, we obtain $H=KH_\delta=TH_\delta$. Observe that as $H$ is primitive,  $H_\delta$ is $\max^+$ in $H$. Let
\[
X=KT\cap H_\delta T=KT\cap H=KT.
\]
Then $K=K\cap X$, $X_\delta=H_\delta\cap X$, and $X=KX_\delta$. By~\cite[Theorem]{LPS1996}, either both $K$ and $X_\delta$ are $\max^+$ in $X$, or $(T,K\cap T,T_\delta)$ lies in~\cite[Table~1]{LPS1996}. With a case-by-case analysis on the triples $(T,K\cap T,T_\delta)$ in ~\cite[Table~1]{LPS1996}, we see that  $|T\,{:}\,T_\delta|^\kappa<2^{\left(\kappa|T_\delta|^{\kappa}\right)^{0.4}}$, contradicting Lemma~\ref{LAM017}. Thus both $K$ and $X_\delta$ are $\max^+$ subgroups of $X$, and so $X$ acts faithfully and primitively by right multiplication on $[X\,{:}\,K]$. Consider this primitive action and let $d=|X\,{:}\,K|$. It follows from~\cite[Theorem~1.1]{Maroti2002} that one of the following holds
\begin{enumerate}[{\rm(i)}]
\item\label{PROP007.1} $|X|<d^{\lfloor\log_2d\rfloor+1}$;
\item\label{PROP007.2} $X$ is $\M_{11}$, $\M_{12}$, $\M_{23}$ or $\M_{24}$ in their $4$-transitive actions;
\item\label{PROP007.3} $(\Alt(m))^x\leq X\leq \Sym(m)\wr\Sym(x)$, where $m\geq5$, $x\geq1$, and the action of $\Sym(m)$ is on the set of $s$-subsets of $\{1,\ldots,m\}$ for some $1\leq s\leq m/2$.
\end{enumerate}
We start with case~\eqref{PROP007.1}. Since $X$ has socle $T$ and $X=KX_\delta$ with $X_\delta$ maximal in $H$, Lemma~\ref{LAM016} implies that $|\Out(T)|\leq|T_\delta|$, and so
\[
|H_\delta|=|T_\delta||H_\delta\,{:}\,T_\delta|\leq|T_\delta||\Out(T)|\leq|T_\delta|^2.
\]
Therefore, $|G|\leq|H_{\delta}\wr\Sym(\kappa)|\leq|T_{\delta}|^{2\kappa}\cdot\kappa!$ and $d=|X\,{:}\,K|\leq|X_\delta|\leq|H_\delta|\leq|T_\delta|^2$. Combining this with~\eqref{eq015} and~\eqref{PROP007.1}, we deduce that
\begin{equation}\label{eq045}
2^{\frac{1}{\kappa}\left(\kappa|T_\delta|^{\kappa}\right)^{0.4}}<|K|=|X|/d<d^{\lfloor\log_2d\rfloor}\leq |T_\delta|^{4\log_2|T_\delta|}.
\end{equation}
However, as $|T_{\delta}|^{2\kappa}\cdot\kappa!\geq|G|\geq2^{57}$, there is no solution for $(\kappa,|T_\delta|)$ in~\eqref{eq045}, a contradiction.

For each $X$ in case~\eqref{PROP007.2}, we have $X=T$, and the pair $(K,T_\delta)=(K,X_\delta)$ can be read off from the classification of factorizations $X=KX_\delta$ with $\max^+$ subgroups $K$ and $X_\delta$ (see~\cite[Table~6]{LPS1990}). However, there is no such pair satisfying~\eqref{eq015} as $\kappa\geq2$, a contradiction.

For the rest of the proof we consider case~\eqref{PROP007.3}. Since $X$ is almost simple, we have $\Alt(m)\leq X\leq\Sym(m)$ for some $m\geq5$, and $K$ is the stabilizer of an $s$-subset of $\{1,\ldots,m\}$ with $1\leq s\leq m/2$. In particular, $|K|\leq s!(m-s)!$. If $m\leq9$, then $|K|\leq8!$ and there is no solution for $(\kappa,|T_\delta|)$ in~\eqref{eq015} such that $|T_{\delta}|^{2\kappa}\cdot\kappa!\geq|G|\geq2^{57}$, a contradiction. Hence we assume $m\geq10$. Since $K$ stabilizes an $s$-subset, we derive from $X=KX_\delta$ that $X_\delta$ is $s$-homogeneous on $\{1,\ldots,m\}$. Suppose that $s\geq2$. Then, from the maximality of $X_\delta$ in $X$ and the classification of $s$-homogeneous groups~\cite{Kantor1972}, we have $|T_\delta|\geq\binom{m}{s}$. Combining this with~\eqref{eq015}, we obtain that
\[
\frac{1}{\kappa}\left(\kappa{\binom{m}{s}}^{\kappa}\right)^{0.4}
\leq\frac{1}{\kappa}\left(\kappa|T_\delta|^{\kappa}\right)^{0.4}
<\log_2|K|\le \log_2\big(s!\cdot(m-s)!\big).
\]
However, this holds only if $\kappa=2,$ $s=2$, $m<60$ and $|T_\delta|<2^{11}$. It follows that
\[
|G|\leq|H_{\delta}\wr\Sym(2)|\leq2(2|T_\delta|)^{2}<2^{25},
\]
a contradiction. Therefore, $s=1$ and we have $\log_2|K|\leq\log_2((m-1)!)<(m-1)\log_2(m-1)$. Since Lemma~\ref{LEM028} implies $|T_\delta|\geq3m$, the inequality~\eqref{eq015} yields
\[
\frac{1}{\kappa}\left(\kappa(3m)^{\kappa}\right)^{0.4}
\leq\frac{1}{\kappa}\left(\kappa|T_\delta|^{\kappa}\right)^{0.4}
<\log_2|K|<(m-1)\log_2(m-1).
\]
For $\kappa\geq3$, this holds only if $\kappa=3$ and $|T_\delta|<2^{17}$, which is impossible as $6(2|T_\delta|)^{3}\geq|G|\geq2^{57}$. Consequently, $\kappa=2$.

So far we have shown that $m\geq10$, $s=1$ and $\kappa=2$. Then $K=X\cap\Sym(m-1)$ and
\[
\Alt(m)\wr\Sym(2)\leq M\leq\Sym(m)\wr\Sym(2).
\]
For $i\in\{1,2\}$, let $\pi_i\colon N\cap L\to\Alt(m)$ be the projection of $N\cap L$ to the $i$th factor of the direct product $N=\Alt(m)\times\Alt(m)$, and let $C_i=\pi_i(N\cap L)$.

Suppose that one of $\pi_1$ or $\pi_2$, say, $\pi_1$, is surjective. Then as $\Ker({\pi_2})\trianglelefteq C_1=\Alt(m)$, we obtain $\Ker(\pi_2)=\Alt(m)$ or $1$. The former contradicts $L\cap G=1$. Thus $\mathrm{Ker}(\pi_2)=1$, that is, $\pi_2$ is injective. Moreover, since $\pi_1$ is surjective, we derive that $N\cap L$ is a diagonal subgroup of $\Alt(m)\times\Alt(m)$. It follows that $N\cap L$ has point stabilizer isomorphic to $\Alt(m-1)$, contradicting the regularity of $L$.

Thus neither $\pi_1$ nor $\pi_2$ is surjective. Note that $|M|\leq2(m!)^2$ and $|M|=|L||G|> 2^{|G|^{0.4}}\cdot|G|$. If $|G|\geq\binom{m}{2}^2$, then
\[
2(m!)^2> 2^{|G|^{0.4}}\cdot|G|\geq2^{\binom{m}{2}^{0.8}}\cdot\binom{m}{2}^2,
\]
which holds only if $m\leq146$ and $|G|<2^{28}$, a contradiction. Hence
\[
|G|<\binom{m}{2}^2.
\]
In view of $N\cap L\leq C_1\times C_2$, we obtain
\begin{equation*}
\frac{|\Alt(m)|}{|C_1|}\cdot\frac{|\Alt(m)|}{|C_2|}=\frac{|N|}{|C_1\times C_2|}\leq\frac{|N|}{|N\cap L|}\leq\frac{|M|}{|L|}
=|G|<\binom{m}{2}^2.
\end{equation*}
As a consequence, $|\Alt(m)\,{:}\,C_i|<\binom{m}{2}$ for some $i\in\{1,2\}$, say, $i=1$. Since $\pi_1$ is not surjective, $C_1$ is a proper subgroup of $\Alt(m)$. Recall $m\geq10$. According to~\cite[Theorem~5.2A]{DM1996}, the only proper subgroup of $\Alt(m)$ with index less than $\binom{m}{2}$ is $\Alt(m-1)$. Thus
\[
C_1=\Alt(m-1).
\]
Since $\Ker(\pi_2)\trianglelefteq C_1$,  it follows that $\Ker(\pi_2)=\Alt(m-1)$ or $1$. The latter implies $N\cap L\cong C_2$ and then
\[
|\Alt(m)|=\frac{|N|}{|\Alt(m)|}\leq\frac{|N|}{|C_2|}=\frac{|N|}{|N\cap L|}\leq\frac{|M|}{|L|}
=|G|<\binom{m}{2}^2,
\]
which is impossible. Therefore, $\Ker(\pi_2)=\Alt(m-1)=C_1$ and hence \[
C_2/\Ker(\pi_1)\cong C_1/\Ker(\pi_2)=1.
\]
Since $\Ker(\pi_2)\times\Ker(\pi_1)\leq N\cap L\leq C_1\times C_2$, we conclude that $N\cap L=C_1\times C_2$. Consequently,
\[
1=G\cap N\cap L=(G\cap N)\cap(N\cap L)=(T_\delta\times T_\delta)\cap(\Alt(m-1)\times C_2),
\]
and so $T_\delta\cap \Alt(m-1)=1$. This means that $T_\delta$ acts regularly on $\{1,\ldots,m\}$. In particular, $T_\delta$ is not maximal in $\Alt(m)$. Since $X_\delta$ is maximal in $X$, we deduce that $X\neq T$, and so $X=\Sym(m)$. Moreover, since $T_\delta\trianglelefteq X_\delta$, the maximality of $X_\delta$ in $X$ implies that $X_\delta=\mathbf{N}_X(T_\delta)=T_\delta\rtimes\Aut(T_\delta)$. As $|X_\delta|=2|T_\delta|$, this yields $|\Aut(T_\delta)|=2$, and so $|T_\delta|\leq6$. However, this leads to $|G|\leq2(2|T_\delta|)^{2}\leq2(2\cdot6)^{2}$, contradicting $|G|\geq2^{57}$. This completes the proof.
\end{proof}

\subsection{Estimate \boldmath{$|\mathcal{Z}_{\UPCD}(G)|$}}

The following result can be extracted from~\cite[Theorem~I]{PS1997}, where the constant $c$ in~\cite[Theorem~I]{PS1997} can be chosen to be $1700$.
\begin{lemma}[Pyber-Shalev]\label{LAM018}
The number of conjugacy classes of primitive subgroups of the symmetric group $\Sym(\ell)$ is at most $2^{1700\log_2^2\ell}$.
\end{lemma}

Let $S\in\mathcal{Z}_{\UPCD}(G)$, and let $(S,M)$ be a critical pair such that the action of $M$ on $[M\,{:}\,G]$ is primitive of CD type. Then $M$ is contained in $H\wr\Sym(\kappa)$ endowed with its natural product action on $\Delta^{\kappa}$, where $\kappa\geq2$ and $H$ is a primitive group of $\UPSD$ type on $\Delta$. Thus there is a positive integer $\ell$ and a nonabelian simple group $T$ such that $|\Delta|=|T|^{\ell-1}$ and
\[
T^{\ell}\leq H\leq T^{\ell}\cdot(\Out(T)\times\Sym(\ell)).
\]
Let $Q$ be the projection of $H$ to $\Sym(\ell)$. Then from the structure of primitive groups of CD type, we derive that
\[
n=|G|\geq\kappa|Q||T|^{\kappa}\geq\kappa\ell|T|^{\kappa}.
\]
Since $\ell,\kappa\geq2$ and $|T|\geq60$, this implies
\begin{equation}\label{eq042}
|T|\leq\frac{\sqrt{n}}{2},\ \ \ell\leq\frac{n}{7200},\ \ \kappa\leq\log_{60}\frac{n}{4}.
\end{equation}
From~\eqref{eq036}, $M_{1_\rs}$ is regular on $[M\,{:}\,G]$ with $|M_{1_\rs}|>2^{n^{0.4}}$. We conclude that
\[
|T|^{(\ell-1)\kappa}=|\Delta|^{\kappa}=|M_{1_\rs}|>2^{n^{0.4}}\geq2^{(\kappa|Q||T|^{\kappa})^{0.4}},
\]
which yields
\[
\frac{\ell-1}{|Q|^{0.4}}\geq\frac{|T|^{0.4\kappa}}{\kappa^{0.6}\log_2|T|}>1,
\]
and so $|Q|<\ell^{2.5}$. It is now not hard to prove the next proposition.

\begin{proposition}\label{PROP008}
Let $G$ be a finite group of order $n$. Then
\[
|\mathcal{Z}_{\UPCD}(G)|<
2^{\frac{3}{4}n+2\log_2^4n+\log_2^3n+1702\log_2^2n+2\log_2n}.
\]
\end{proposition}

\begin{proof}
We adopt the notation established after Lemma~\ref{LAM018}, and let $(S,M)$ be a critical pair with $M$ primitive on $[M\,{:}\,G]$ of CD type. Then there exist $T$, $\ell$, $\kappa$ satisfying~\eqref{eq042} and a primitive group $Q\leq\Sym(\ell)$ with $|Q|<\ell^{2.5}$ such that
\[
T^{\ell\kappa}=\Soc(M)=:N\trianglelefteq M\leq W:= \big(T^{\ell}\cdot(\Out(T)\times Q)\big)\wr\Sym(\kappa).
\]

We first estimate the number of choices for the tuple $(T,\ell,\kappa,W,M)$. Note that at most two nonabelian simple groups, up to isomorphism, have the same order (see~\cite{Lubotzky2001}). Considering the choices for $T$, $\ell$, $\kappa$ and $Q$, we derive from~\eqref{eq042} and Lemma~\ref{LAM018} that there are at most
\begin{equation}\label{eq043}
2\cdot\frac{\sqrt{n}}{2}\cdot\frac{n}{7200}\cdot\left(\log_{60}\frac{n}{4}\right)
\cdot2^{1700\log_2^2\ell}
<n^2\cdot2^{1700\log_2^2n}
=2^{1700\log_2^2n+2\log_2n}
\end{equation}
possibilities for the tuple $(T,\ell,\kappa,W)$. Moreover, in view of~\eqref{eq042} and Lemma~\ref{LEM010}, we have
\begin{align*}
|W\,{:}\,N|=(|\Out(T)||Q|)^\kappa\cdot\kappa!&\leq\left((\log_2|T|)\cdot\ell^{2.5}\right)^\kappa\cdot\kappa^\kappa\\
&=2^{\kappa(\log_2(\log_2|T|)+2.5\log_2\ell+\log_2\kappa)}<2^{5\kappa\log_2n}<2^{\log_2^2n},
\end{align*}
where the last two inequalities follow from (rather crude) estimates using~\eqref{eq042}. For a given pair $(N,W)$, recall that $M=GN$. Therefore, since $G$ has at most $\log_2n$ generators, the number of choices for $M=NG$ with $N\leq M\leq W$ is at most $|W/N|^{\log_2n}\leq2^{\log_2^3n}$. Combining this with~\eqref{eq043} and recalling that $N=T^{\ell\kappa}$, we conclude that the number of possible tuples $(T,\ell,\kappa,W,M)$ is at most $2^{\log_2^3n+1700\log_2^2n+2\log_2n}.$

Next we fix $(T,\ell,\kappa,W,M)$ to estimate the number of choices for $M_{1_\rs}$. Consider the right multiplication of $N$ on $[N\,{:}\,N_{1_\rs}]$. Let $C=\mathrm{Core}_N(N_{1_\rs})$ so that $N/C\cong T^s$ for some $s\leq\ell\kappa$. Since $|N\,{:}\,N_{1_\rs}|\leq|M\,{:}\,M_{1_\rs}|=n$, the group $N/C$ has a faithful transitive permutation representation on $[N\,{:}\,N_{1_\rs}]$ of degree at most $n$. By~\cite[Proposition~5.2.7(ii)]{KL1990}, the minimal degree of a faithful transitive permutation representation of $T^s$ is greater than $(\min\{|T|^{1/2},P(T)\})^s$, where $P(T)$ is the minimal degree of a nontrivial permutation representation of $T$. Clearly, $\min\{|T|^{1/2},P(T)\}\geq5$. Thus $n\geq(\min\{|T|^{1/2},P(T)\})^s\geq5^s$ and so $s\leq\log_2n/\log_25$. Note that, for a fixed $s\leq\ell\kappa$, there are at most $\binom{\ell\kappa}{s}<(\ell\kappa)^s$ possibilities for $C$. Therefore, it follows from~\eqref{eq042} that the number of choices for $C$ is at most
\begin{equation}\label{eq044}
\frac{\log_2n}{\log_25}\cdot(\ell\kappa)^{\frac{\log_2n}{\log_25}}
<\frac{\log_2n}{2}\cdot2^{(\log_2n)\frac{\log_2(\ell\kappa)}{2}}
<2^{(\log_2n)\big(\frac{\log_2(\ell\kappa)}{2}+1\big)}
<2^{\log_2^2n}.
\end{equation}
Now fix some $C\trianglelefteq N$ such that $|N/C|=|T|^s$ for some $s\leq\log_2n/\log_25$. Then
\[
|M\,{:}\,C|=\frac{|G||M_{1_\rs}|}{|C|}=\frac{n|T|^{(\ell-1)\kappa}}{|T|^{\ell\kappa-s}}
<n|T|^{s}
\leq n\left(\frac{\sqrt n}{2}\right)^\frac{\log_2n}{\log_25}
<2^{\log_2^2n}.
\]
Since $C\leq M_{1_\rs}\leq M$, this together with~\eqref{eq044} implies that the number of choices for $M_{1_\rs}$ is less than
\[
2^{\log_2^2n}\cdot|M\,{:}\,C|^{\log_2|M\,{:}\,C|}<
2^{\log_2^4n+\log_2^2n}.
\]

Similarly, the number of choices for $M_{1_\ls}$ for a fixed $(T,\ell,\kappa,W,M)$ is less than $2^{\log_2^4n+\log_2^2n}$. Consequently, the number of possible triples $(M,M_{1_\rs},M_{1_\ls})$ is less than
\[
2^{\log_2^3n+1700\log_2^2n+2\log_2n}\cdot2^{2\log_2^4n+2\log_2^2n}
=2^{2\log_2^4n+\log_2^3n+1702\log_2^2n+2\log_2n}.
\]
Then the proposition follows immediately from Lemma~\ref{LEM009}.
\end{proof}

\subsection{Estimate \boldmath{$|\mathcal{Z}_{\UPHA}(G,\varepsilon)\cup\mathcal{Z}_{\UPSD}(G,\varepsilon)\cup\mathcal{Z}_{\UPTW}(G,\varepsilon)|$}}

Before exhibiting the main result of this section, we prove two technical lemmas.

\begin{lemma}\label{LEM019}
Let $(S,M)$ be an $\varepsilon$-critical pair of a group $G$ of order $n$, let $N$ be the socle of $M$, and let $\tau\in\{\rs,\ls\}$.
\begin{enumerate}[{\rm(a)}]
\item\label{LEM019.1} If the primitive type of $M$ on $[M\,{:}\,G]$ is \textup{HA}, then $|(1_{\tau})^N|<n^{0.5+\varepsilon}\log_2n$.
\item\label{LEM019.2} If the primitive type of $M$ on $[M\,{:}\,G]$ is \textup{SD} or \textup{TW}, then $|(1_{\tau})^N|<n^{0.5+\varepsilon}\log_2^2n$.
\end{enumerate}
\end{lemma}

\begin{proof}
We only prove the result for $\tau=\rs$, as the argument for $\tau=\ls$ is the same.
Write $u=1_\rs$. Then we have
\begin{equation}\label{eq021}
|u^N|=|N\,{:}\,N_u|=|N\,{:}\,N\cap M_u|=|NM_u\,{:}\,M_u|.
\end{equation}
In the following, we discuss the cases $S\in\mathcal{Z}_{\UPHA}(G)$, $S\in\mathcal{Z}Y_{\UPSD}(G)$ and $S\in\mathcal{Z}_{\UPTW}(G)$ one by one. Note that in each case we have $|M_u|>2^{n^{0.5-\varepsilon}}$ from the definition of $\mathcal{Z}$.

\textsc{Case 1}: $S\in\mathcal{Z}_{\UPHA}(G)$.

In this case, $M=N\rtimes G$ with $N=C_p^\ell$ for some prime $p$ and positive integer $\ell$, and so $|M_u|=|M\,{:}\,G|=|N|=p^\ell$. Let $H=NM_u\cap G$. Then $H$ is a $p$-group, and
\begin{equation}\label{eq037}
NM_u=NM_u\cap M=NM_u\cap NG=N(NM_u\cap G)=NH=N\rtimes H.
\end{equation}
Hence~\eqref{eq021} turns out to be
\begin{equation}\label{eq022}
|u^N|=|NM_u\,{:}\,M_u|=|NH\,{:}\,M_u|=|NH\,{:}\,N|=|H|.
\end{equation}
If $M_u=N$, then $|u^N|=|u^{M_u}|=|\{u\}|=1<n^{0.5+\varepsilon}\log_2n$. Now suppose that $M_u\neq N$. Then~\eqref{eq037} shows that $H\neq1$, which implies $n=|G|\geq|H|\geq p$. Since $H$ is a $p$-subgroup of $G$, there exists a non-identity $x\in N\setminus\{1\}$ fixed by $H$. Therefore, $H\leq\mathbf{C}_{G}(x)$. Since $G$ acts irreducibly on $N$, the set $x^G$ spans $N$. So we obtain $\ell\leq|x^G|=|G\,{:}\,\mathbf{C}_{G}(x)|\leq|G\,{:}\,H|$, which means $|H|\leq n/\ell$. Moreover, it follows from  $p^\ell=|N|=|M_u|\geq2^{n^{0.5-\varepsilon}}$ that
\[
\ell>\log_p2^{n^{0.5-\varepsilon}}=\frac{n^{0.5-\varepsilon}}{\log_2p}\geq \frac{n^{0.5-\varepsilon}}{\log_2n}.
\]
Combining this with~\eqref{eq022} and $|H|\leq n/\ell$, we conclude that
\[
|u^N|=|H|\leq\frac{n}{\ell}<\frac{n\log_2n}{n^{0.5-\varepsilon}}=n^{0.5+\varepsilon}\log_2n.
\]

\textsc{Case 2}: $S\in\mathcal{Z}_{\UPSD}(G)$.

In this case, $N=T^\ell$ for some nonabelian simple group $T$ and integer $\ell\geq2$ such that $N\cap G\cong T$ and $GN/N$ is a transitive subgroup of $\Sym(\ell)$. Then $n=|G|\geq\ell|T|$, and
\[
|M_u|=|M\,{:}\,G|=|N\,{:}\,N\cap G|=|T|^{\ell-1}.
\]
Since~\cite[Proposition~5.5]{MS2021} implies that $|M_uN|\leq |T^{\ell}||\Out(T)|$, we infer from~\eqref{eq021} and Lemma~\ref{LEM010} that
\begin{equation}\label{eq023}
|u^N|=|NM_u\,{:}\,M_u|\leq |T||\Out(T)|\leq|T|\log_2|T|.
\end{equation}
As $|T|^\ell>|T|^{\ell-1}=|M_u|\geq 2^{n^{0.5-\varepsilon}}$, we have $\ell> n^{0.5-\varepsilon}/\log_2|T|$. This combined with~\eqref{eq023} and $|T|\leq n/\ell$ leads to
\[
|u^N|\leq|T|\log_2|T|\leq\frac{n}{\ell}\log_2|T|<\frac{n\log^2_2|T|}{n^{0.5-\varepsilon}}< n^{0.5+\varepsilon}\log^2_2n.
\]

\textsc{Case 3}: $S\in\mathcal{Z}_{\UPTW}(G)$.

In this case, $N=T^\ell$ for some nonabelian simple group $T$ and integer $\ell\geq6$ such that $|N|=|M\,{:}\,G|$ and $GN/N$ is a transitive subgroup of $\Sym(\ell)$. By~\cite[Proposition~5.6]{MS2021} we have $|M_u\,{:}\,N\cap M_u|\leq |\Aut(T)|$. Hence~\eqref{eq021} and Lemma~\ref{LEM010} yields that
\begin{equation*}
|u^N|=|N\,{:}\,N\cap M_u|=|M_u\,{:}\,N\cap M_u|\leq |\Aut(T)|=|T||\Out(T)|\leq|T|\log_2|T|.
\end{equation*}
Note that the induced permutation group $P$ of $G$ on the $\ell$ direct factors in $T^{\ell}$ is a transitive subgroup of $\Sym(\ell)$ with stabilizer $P_1\gtrsim T$ (see~\cite{LPS1988}). Thus, we obtain $|G|\geq\ell|T|$ and so $|T|\leq n/\ell$. Then the same argument as in Case~2 leads to $|u^N|< n^{0.5+\varepsilon}\log_2^2n$, as required.
\end{proof}

\begin{lemma}\label{LEM020}
Let $G$ be a group of order $n$. For each positive integer $t$, the number of subsets $S$ of $G$ such that there exist subgroups $H$ and $K$ of $G$ satisfying the following conditions~\eqref{LEM020.0}--\eqref{LEM020.2} is at most $2^{n-\frac{n}{4t}+\frac{\log_2^2n}{2}+6}$.
\begin{enumerate}[{\rm(a)}]
\item\label{LEM020.0} $\min\{|H|,|K|\}\leq t$;
\item\label{LEM020.1} either $H\neq K$, or $H$ is not normal in $G$;
\item\label{LEM020.2} $\{(Hg)_{\rs}\mid g\in G\}\cup\{(Kg)_{\ls}\mid g\in G\}$ is the set of orbits of some $X\leq\Aut(\HH(G,S))$.
\end{enumerate}
\end{lemma}

\begin{proof}
Let $H$ and $K$ be subgroups of $G$ satisfying~\eqref{LEM020.0}--\eqref{LEM020.2}, and write $\Gamma=\HH(G,S)$. Let $g\in G$ and $h\in H$.
As $(1_{\rs})^H=(1_{\rs})^X$, there exists $x\in X$ such that $(1_{\rs})^h=(1_{\rs})^x$, that is, $hx^{-1}$ fixes $1_{\rs}$. Then it follows from~\eqref{LEM020.2} that
\begin{align*}
|N_{\Gamma}(1_{\rs})\cap (Kg)_{\ls}|&=|(N_\Gamma(1_{\rs})\cap (Kg)_{\ls})^{hx^{-1}}|\\
&=|N_{\Gamma}\big((1_{\rs})^{hx^{-1}}\big)\cap \big((Kg)_{\ls}\big)^{hx^{-1}}|=|N_{\Gamma}(1_{\rs})\cap (Kgh)_{\ls}|.
\end{align*}
(Observe here that $(Kg)_{\ls}^{hx^{-1}}=(Kgh)_{\ls}^{x^{-1}}$; moreover, as $x\in X$, we see that $x$ fixes setwise the $K$-orbit $(Kgh)_{\ls}$.) Thus we obtain $|S\cap Kg|=|S\cap Kgh|$ for each $g\in G$ and $h\in H$.
Consequently, the intersections between $S$ and the right cosets $Kgh$ of $K$ with $h\in H$ have the same size, and in particular, the same parity. In other words, the intersections between $S$ and different right cosets of $K$ in $KgH$ have the same parity of size. Let $\Delta_1,\ldots,\Delta_\kappa$ be the double cosets of $K$ and $H$ in $G$, and fix a right coset $\Theta_i$ of $K$ in $\Delta_i$ for each $i\in\{1,\ldots,\kappa\}$. Then the intersection of $S$ with each right coset of $K$ in $G$ other than $\Theta_1\ldots,\Theta_\kappa$ must have fixed parity of size. Hence the number of choices for $S$ is at most
\begin{equation}\label{eq018}
2^{\kappa|K|}\cdot2^{(|G\,{:}\,K|-\kappa)(|K|-1)}=2^{|G|+\kappa-|G\,{:}\,K|}.
\end{equation}
Similarly, as $(1_{\ls})^K=(1_{\ls})^X$, the intersections between $S$ and different right cosets of $H$ in $HgK$ have the same parity of size, which implies that the number of choices for $S$ is at most $2^{|G|+\kappa-|G\,{:}\,H|}$.
Combining this with~\eqref{eq018}, we deduce that, for fixed $H$ and $K$, the number of choices for $S$ is at most
\[
2^{|G|+\kappa-m},
\]
where $m=\max\{|G\,{:}\,H|,|G\,{:}\,K|\}$. Since $H$ and $K$ satisfy~\eqref{LEM020.0}, we have $m\geq|G|/t$. Moreover, since $H$ and $K$ satisfy~\eqref{LEM020.1}, by Lemma~\ref{LEM008}, we have $\kappa\leq3m/4$. Therefore,
\[
2^{|G|+\kappa-m}\leq2^{|G|-\frac{m}{4}}\leq2^{n-\frac{n}{4t}}.
\]
Thus the conclusion of the lemma immediately follows, as there are at most $2^{\frac{\log_2^2n}{2}+6}$ choices for the pair $(H,K)$ of subgroups of $G$ by Lemma~\ref{LEM004}\eqref{LEM004.4}.
\end{proof}

We are now ready to give the main result of this section.

\begin{proposition}\label{PROP009}
Let $G$ be a finite group of order $n\geq2^{23}$, and let $\varepsilon\in\left(0,0.1\right]$. Then
\[
|\mathcal{Z}_{\UPHA}(G,\varepsilon)\cup\mathcal{Z}_{\UPSD}(G,\varepsilon)\cup\mathcal{Z}_{\UPTW}(G,\varepsilon)|
<2^{n-\frac{n^{0.5-\varepsilon}}{8\log_2^2n}+\frac{\log_2^2n}{2}+7}.
\]
\end{proposition}

\begin{proof}
Let $(S,M)$ be an $\varepsilon$-critical pair of $G$ such that the primitive type of $M$ on $[M\,{:}\,G]$ is HA, SD or TW, and let $N$ be the socle of $M$. Then $M$ preserves the partition into $N$-orbits of $G_{\rs}\cup G_{\ls}$. Let $H$ and $K$ be the stabilizers in $G$ of $(1_{\rs})^N$ and $(1_{\ls})^N$, respectively.
It follows from the semiregularity of $G$ that $|H|=|(1_{\rs})^N|$ and $|K|=|(1_{\ls})^N|$. Thus Lemma~\ref{LEM019} gives
\begin{equation}\label{eq019}
\max\{|H|,|K|\}\leq n^{0.5+\varepsilon}\log_2^2n.
\end{equation}
We estimate the number of $S$ in the following two categories:
\begin{enumerate}[{\rm(i)}]
\item\label{PROP009.1} $S$ is such that $H=K$ is normal in $G$;
\item\label{PROP009.2} $S$ is such that either $H\neq K$ or $H$ is not normal in $G$.
\end{enumerate}

First assume that~\eqref{PROP009.1} holds. Then it follows from the definition of $H$ and $K$ that $H$ has the same orbits on $V(\HH(G,S))$ as $N$. In particular, $H\neq1$. As the inequality $n^{0.5+\varepsilon}\log_2^2n<n$ holds for $n\geq2^{23}$, we see from~\eqref{eq019} that $H\neq G$. Moreover, since $N\nleq G$, there is $f\in N\setminus G$ stabilizing each $H$-orbit on $V(\HH(G,S))$. Applying Proposition~\ref{PROP001} with $t=n^{0.5+\varepsilon}\log_2^2n$ to the pair $(H,f)$, we conclude from~\eqref{eq019} that the number of $S$ in~\eqref{PROP009.1} is at most
\begin{equation}\label{eq020}
2^{n-\frac{n}{3n^{0.5+\varepsilon}\log_2^2n}\log_2\left(\frac{4}{3}\right)+\frac{\log_2^2n}{4}+\log_2n+
2\log_2\left(n^{0.5+\varepsilon}\log_2^2n\right)+5}<2^{n-\frac{n^{0.5-\varepsilon}}{8\log_2^2n}+\frac{\log_2^2n}{2}+6}.
\end{equation}

Next assume that~\eqref{PROP009.2} holds. Since $(1_\rs)^H=(1_\rs)^N$ and $(1_\ls)^H=(1_\ls)^N$, it follows that
\begin{align*}
\{(Hg)_{\rs}\mid g\in G\}\cup\{(Kg)_{\ls}\mid g\in G\}
&=\{(1_{\rs})^{Hg}\mid g\in G\}\cup\{(1_{\ls})^{Kg}\mid g\in G\}\\
&=\{(1_{\rs})^{Ng}\mid g\in G\}\cup\{(1_{\ls})^{Ng}\mid g\in G\}
\end{align*}
is the set of orbits of $N$. Then Lemma~\ref{LEM020} and~\eqref{eq019} imply that the number of $S$ in~\eqref{PROP009.2} is at most
\[
2^{n-\frac{n}{4n^{0.5+\varepsilon}\log_2^2n}+\frac{\log_2^2n}{2}+6}<2^{n-\frac{n^{0.5-\varepsilon}}{8\log_2^2n}+\frac{\log_2^2n}{2}+6}.
\]
This together with~\eqref{eq020} proves the proposition.
\end{proof}

\section{Proof of Theorems~\ref{THM001} and~\ref{THM005}}\label{SEC5}

In this section, we first prove Theorem~\ref{THM001} by using the results in Sections~\ref{SEC3}.
Then we apply Theorem~\ref{THM001} to show Theorem~\ref{THM005}, which gives an asymptotic enumeration for HGRs up to isomorphism.

\subsection{Proof of Theorem~\ref{THM001}}\label{SUBSEC5.1}

First recall Definition~\ref{DEF002} and Remark~\ref{RMK001} on odd-quotient digraphs.

\begin{lemma}\label{LEM021}
Let $G$ be a finite group, and let $C$ be a normal subgroup of $G$. Then there are exactly $2^{|G|-|G|/|C|}$ subsets $S$ of $G$ corresponding to the same $\HH(G,S)^{\odd}_{C}$.
\end{lemma}

\begin{proof}
Let $\mathcal{B}=\{(Cg)_\rs\mid g\in G\}\cup\{(Cg)_\ls\mid g\in G\}$. Then $\mathcal{B}$ is the set of $C$-orbits on the vertex set of $\HH(G,S)$ for every $S\subseteq G$. Note that $C_{\rs}$ is adjacent to $(Cg)_{\ls}$ if and only if $|S\cap Cg|$ is odd. By Lemma~\ref{LEM002}, for a fixed $\HH(G,S)^{\odd}_{C}$, there are exactly $\big(2^{|C|-1}\big)^{|G\,{:}\,C|}=2^{|G|-|G|/|C|}$ choices for $S$, as the lemma asserts.
\end{proof}

Recall also the concept of exceptional pairs and minimally exceptional pairs in~Definition~\ref{DEF003}. The following result is crucial in the proof of Theorem~\ref{THM001}.

\begin{proposition}\label{PROP010}
Let $\varepsilon\in\left(0,0.1\right]$, and let $n_\varepsilon$ be a positive integer such that~\eqref{eq032} holds for all $n\geq n_\varepsilon$. Let $G$ be a finite group of order $n$. Then the number of subsets $S$ of $G$ such that $\Aut^+(\HH(G,S))>G$ is less than $2^{n-\frac{n^{0.5-\varepsilon}}{24(\log_2n)^{2.5}}+\frac{3\log_2^2n}{4}+15}$.
\end{proposition}

\begin{proof}
Since~\eqref{eq032} holds, a direct computation shows that $n\geq n_\varepsilon\geq n_{0.1}>2^{67}$. Let $\mathcal{X}(G)=\{S\subseteq G\mid\Aut^+(\HH(G,S))>G\}$. Then for each $S\in\mathcal{X}(G)$, there exists a minimally exceptional pair $(S,M)$. Consider the following conditions:
\begin{enumerate}[(C1)]
\item $|M|>2^{n^{0.5-\varepsilon}+\log_2n}$;
\item $\mathrm{Core}_M(G)=1$;
\item $|\mathrm{Core}_M(G)|\leq8\log_2n$;
\item no $f$ in $M_{1_{\rs}}\setminus G$ stabilizes every $\mathrm{Core}_M(G)$-orbit on $G_\rs\cup G_\ls$.
\end{enumerate}
For $i\in\{1,2,3,4\}$, we say that $(S,M)$ satisfies condition~($\text{C}i$) if ($\text{C}i$) holds, and that $(S,M)$ satisfies condition~($\overline{\text{C}i}$) if~($\text{C}i$) does not hold.

Let $\mathcal{Y}(G)=\{S\subseteq G\mid\text{there is a minimally exceptional pair $(S,M)$ satisfying~($\overline{\text{C1}}$) or~($\overline{\text{C3}}$)}\}$. Note that~\eqref{eq033} holds as $n>2^{67}$ and $\varepsilon\in\left(0,0.1\right]$. Applying Propositions~\ref{PROP002} and~\ref{PROP003}, we obtain
\begin{equation}\label{eq047}
|\mathcal{Y}(G)|<2^{\frac{3}{4}n+n^{1-\varepsilon}}
+2^{n-\frac{n}{4\log_2n}\log_2\left(\frac{e}{2}\right)
+\frac{\log^2_2n}{4}+\frac{1}{2}\log_2\left(\frac{n}{4\log_2n}\right)+\log_2(24)}.
\end{equation}
Recall from Definition~\ref{DEF004} that $\mathcal{Z}(G,\varepsilon)$ is the set of $S\subseteq G$ such that there exists a minimally exceptional pair $(S,M)$ satisfying~(C1) and (C2). By Proposition~\ref{PROP004},
\begin{equation}\label{eq049}
|\mathcal{Z}(G,\varepsilon)|<2^{n-\frac{n^{0.5-\varepsilon}}{8\log_2^2n}+\frac{\log_2^2n}{2}+9}.
\end{equation}
Hence we only need to estimate the size of $\mathcal{X}(G)\setminus(\mathcal{Y}(G)\cup\mathcal{Z}(G,\varepsilon))$, that is, the number of $S\subseteq G$ such that there exists a minimally exceptional pair $(S,M)$ satisfying~(C1),~($\overline{\text{C}2}$) and~(C3).

Let $S\in\mathcal{X}(G)\setminus(\mathcal{Y}(G)\cup\mathcal{Z}(G,\varepsilon))$, and let $(S,M)$ be a minimally exceptional pair satisfying~(C1),~($\overline{\text{C2}}$) and~(C3). We first estimate the number of $S$ under the assumption that $(S,M)$ satisfies~($\overline{\text{C4}}$). Let $f$ be an element of $M_{1_\rs}\setminus G$ stabilizing every $\mathrm{Core}_M(G)$-orbit on $G_\rs\cup G_\ls$. Then applying Proposition~\ref{PROP001} with the normal subgroup $\mathrm{Core}_M(G)$ of $G$ and the automorphism $f$, we derive from~(C3) that the number of such subsets $S$ is at most
\begin{equation}\label{eq048}
2^{n-\frac{n}{24\log_2n}\log_2(\frac{4}{3})+\frac{\log_2^2n}{4}+\log_2n+2\log_2(8\log_2n)+5}.
\end{equation}

Now assume that $(S,M)$ satisfies~(C4). Write $\Gamma=\HH(G,S)$ and $C=\mathrm{Core}_M(G)$. Then it is clear from the semiregularity of $G$ on $V(\Gamma)$ that $\Gamma^{\odd}_{C}$ is a Haar graph of $G/C$. Thus $\Gamma^{\odd}_{C}=\HH(G/C,T)$ for some $T\subseteq G/C$. Since the kernel $N$ of the induced action of $M$ on $V(\Gamma^{\odd}_{C})$ is contained in $M_{1_\rs}C$ while $M_{1_\rs}\cap G=1$, it follows from~(C4) that $N=C$, and so $M/C$ acts on $V(\Gamma^{\odd}_{C})$ faithfully. Since every automorphism of $\Gamma$ induces an automorphism of $\Gamma^{\odd}_{C}$, we may regard $M/C$ as a subgroup of $\Aut(\Gamma^{\odd}_{C})$. Then $(T,M/C)$ is a minimally exceptional pair of $G/C$, as $G$ is maximal in $M$. The condition~(C1) implies that
\[
|M/C|\geq\frac{2^{n^{0.5-\varepsilon}+\log_2n}}{|C|}=|G\,{:}\,C|\cdot2^{n^{0.5-\varepsilon}}
\geq|G\,{:}\,C|\cdot2^{|G:C|^{0.5-\varepsilon}}.
\]
Moreover, it follows from $C=\mathrm{Core}_M(G)$ that $\mathrm{Core}_{M/C}(G/C)=1$. Hence $(T,M/C)$ is an $\varepsilon$-critical pair of $G/C$ (recall Definition~\ref{DEF004}), and so $T\in\mathcal{Z}(G/C,\varepsilon)$. Since $n>2^{67}$ and (C3) gives $|C|\leq8\log_2n$, we have $|G/C|\geq n/(8\log_2n)>2^{57}$. Counting the number of choices for subgroups $C$ of $G$ and the number of $T\in\mathcal{Z}(G/C,\varepsilon)$, we derive from Lemma~\ref{LEM004}\eqref{LEM004.4}, Proposition~\ref{PROP004} and Lemma~\ref{LEM021} that the number of choices for $S$ is at most
\begin{align*}
2^{\frac{\log_2^2n}{4}+3}\cdot|\mathcal{Z}(G/C,\varepsilon)|\cdot2^{n-\frac{n}{|C|}}
&<2^{\frac{\log_2^2n}{4}+3}\cdot2^{\frac{n}{|C|}-\frac{(n/|C|)^{0.5-\varepsilon}}{8\log_2^2(n/|C|)}+\frac{\log_2^2\left(n/|C|\right)}{2}+9}
\cdot2^{n-\frac{n}{|C|}}\nonumber\\
&<2^{n-\frac{(n/|C|)^{0.5-\varepsilon}}{8\log_2^2(n/|C|)}+\frac{3\log_2^2n}{4}+12}
<2^{n-\frac{n^{0.5-\varepsilon}}{24(\log_2n)^{2.5}}+\frac{3\log_2^2n}{4}+12}.
\end{align*}
Note that the above upper bound is greater than~\eqref{eq048}. Accordingly,
\begin{equation}\label{eq064}
|\mathcal{X}(G)\setminus(\mathcal{Y}(G)\cup\mathcal{Z}(G,\varepsilon))|<
2\cdot2^{n-\frac{n^{0.5-\varepsilon}}{24(\log_2n)^{2.5}}+\frac{3\log_2^2n}{4}+12}
=2^{n-\frac{n^{0.5-\varepsilon}}{24(\log_2n)^{2.5}}+\frac{3\log_2^2n}{4}+13}.
\end{equation}


Since $n>2^{67}$, the right-hand sides of~\eqref{eq047} and~\eqref{eq049} are both less than the right-hand side of~\eqref{eq064}. We conclude that
\begin{align*}
|\mathcal{X}(G)|&\leq|\mathcal{Y}(G)|+|\mathcal{Z}(G)|+|\mathcal{X}(G)\setminus(\mathcal{Y}(G)\cup\mathcal{Z}(G,\varepsilon))|\\
&<3|\mathcal{X}(G)\setminus(\mathcal{Y}(G)\cup\mathcal{Z}(G,\varepsilon))|
<2^{n-\frac{n^{0.5-\varepsilon}}{24(\log_2n)^{2.5}}+\frac{3\log_2^2n}{4}+15}.
\end{align*}
This completes the proof.
\end{proof}

For a subset $S$ of a group, recall that $\mathcal{I}(S)$ is the set of elements in $S$ of order at most $2$, and $c(S)=(|S|+|\mathcal{I}(S)|)/2$.

\begin{lemma}\label{LEM022}
Let $G$ be a nonabelian permutation group acting semiregularly on $2n$ points with exactly two orbits $U$ and $W$. Then for each regular group $X>G$, there are at most $2^{7n/8}$ bipartite graphs $\Gamma$ with bipartition $\{U,W\}$ such that $X\leq\Aut(\Gamma)$.
\end{lemma}

\begin{proof}
Each bipartite graph $\Gamma$ with bipartition $\{U,W\}$ such that $X\leq\Aut(\Gamma)$ is a Cayley graph $\Cay(X,S)$ for some inverse-closed subset $S$ of $X\setminus G$. By Lemma~\ref{LEM007},
\[
c(X\setminus G)=\frac{|X\setminus G|+|\mathcal{I}(X\setminus G)|}{2}
\leq\frac{1}{2}\left(n+\frac{3n}{4}\right)\leq\frac{7n}{8}.
\]
Thus Lemma~\ref{LEM001} implies that there are at most $2^{c(X\setminus G)}\leq2^{7n/8}$ choices for the inverse-closed subsets $S$ of $X\setminus G$ and hence for bipartite graphs $\Gamma$.
\end{proof}

We are now in a position to prove Theorem~\ref{THM001}.

\begin{proof}[Proof of Theorem~$\ref{THM001}$]
Recall Definition~\ref{DEF005} and note that, when $G$ is abelian and $S\subseteq G$, $\Aut(\HH(G,S))=G\rtimes\langle\iota\rangle$ if and only if $\Aut^+(\HH(G,S))=G$. Hence Theorem~\ref{THM001}\eqref{THM001.2} immediately follows from Proposition~\ref{PROP010}. Next let $G$ be nonabelian, and let $\mathcal{X}(G)=\{S\subseteq G\mid\Aut(\HH(G,S))>G\}$. Proposition~\ref{PROP010} gives an upper bound for the size of the subset $\mathcal{Y}(G):=\{S\subseteq G\mid\Aut^{+}(\HH(G,S))>G\}$ of $\mathcal{X}(G)$. Hence we just need to estimate $|\mathcal{X}(G)\setminus\mathcal{Y}(G)|$. Let $S\in\mathcal{X}(G)\setminus\mathcal{Y}(G)$. Then $\Aut(\HH(G,S))=X$ for some $X>G$ with $|X\,{:}\,G|=2$ such that $X$ is regular on $V(\HH(G,S))$. Since $G$ is normal and maximal in $X$, the group $X$ can be determined by an element of $\mathbf{N}_{\Sym(2n)}(G)$. Observe that since $G$ is semiregular with $2$ orbits, we have $|{\bf C}_{\mathrm{Sym}(n)}(C)|=2|G|^2=2n^2$. Now Lemma~\ref{LEM004}\eqref{LEM004.3} implies that
\begin{align*}
|\mathbf{N}_{\Sym(2n)}(G)|&=|\mathbf{C}_{\Sym(2n)}(G)|\cdot|\Aut(G)|\\
&=2n^2\cdot|\Aut(G)|\leq2n^2\cdot2^{\log_2^2n}
\leq2^{\log_2^2n+2\log_2n+1}.
\end{align*}
Hence there are at most $2^{\log_2^2n+2\log_2n+1}$ choices for  $X>G$ with $|X\,{:}\,G|=2$. Combining this with Lemma~\ref{LEM022}, we deduce that
\[
|\mathcal{X}(G)\setminus\mathcal{Y}(G)|\leq2^{\frac{7n}{8}+\log_2^2n+2\log_2n+1}.
\]
Together with Proposition~\ref{PROP010}, we conclude that
\begin{align*}
|\mathcal{X}(G)|=|\mathcal{Y}(G)|+|\mathcal{X}(G)\setminus\mathcal{Y}(G)|
&\leq2^{n-\frac{n^{0.5-\varepsilon}}{24(\log_2n)^{2.5}}+\frac{3\log_2^2n}{4}+15}+2^{\frac{7n}{8}+\log_2^2n+2\log_2n+1}\\
&<2^{n-\frac{n^{0.5-\varepsilon}}{24(\log_2n)^{2.5}}+\frac{3\log_2^2n}{4}+16},
\end{align*}
proving Theorem~\ref{THM001}\eqref{THM001.1}.
\end{proof}

\subsection{Unlabeled Haar graphs}\label{SUBSEC5.2}

For a group $G$, let $\mathcal{H}(G)$ be the set of Haar graphs of $G$ up to isomorphism. Elements of $\mathcal{H}(G)$ are called \emph{unlabelled} Haar graphs of $G$, and in contrast, we call a Haar graph $\HH(G,S)$ \emph{labeled} to indicate that it is not counted up to isomorphism. In this section, we prove Theorem~\ref{THM005}, which indicates that, up to isomorphism, almost all Haar graphs of a finite group have the smallest possible automorphism group. We first give the following lemma.

\begin{lemma}\label{LEM023}
Let $G$ be a finite group of order $n$.
\begin{enumerate}[{\rm(a)}]
\item\label{LEM023.1} If $G$ is nonabelian, then there are at most $2^{\log_2^2n+2\log_2n+1}$ labeled HGRs of $G$ isomorphic to each other.
\item\label{LEM023.2} If $G$ is abelian, then there are at most $2^{\log_2^2(2n)+2\log_2n+1}$ labeled Haar graphs of $G$ isomorphic to each other such that they have automorphism group $G\rtimes\langle\iota\rangle$.
\end{enumerate}
\end{lemma}

\begin{proof}
When $n=1$, the result is clear. Thus we suppose $n>1$.

First assume that $G$ is nonabelian. Fix a subset $S$ of $G$ such that $\HH(G,S)$ is an HGR. To prove~\eqref{LEM023.1}, we show that there are at most $2^{\log_2^2n+2\log_2n+2}$ subsets $T$ of $G$ such that $\HH(G,T)$ is isomorphic to $\HH(G,S)$. Let $T$ be such a subset, and write $\Gamma=\HH(G,S)$ and $\Sigma=\HH(G,T)$. Then there exists a graph isomorphism $\varphi$ from~$\Gamma$ to~$\Sigma$, and the mapping $\psi\colon g\mapsto\varphi^{-1}g\varphi$, for each $g\in G$, is a group isomorphism from $\Aut(\Gamma)$ to $\Aut(\Sigma)$. Since $N_{\Gamma}(1_\rs)=S_\ls$, we have
\begin{equation}\label{eq065}
N_{\Sigma}\big((1_\rs)^\varphi\big)=(S_\ls)^{\varphi}.
\end{equation}
Let $(1_\rs)^\varphi=x_\epsilon$ and $(1_\ls)^\varphi=y_\mu$ for some $x,y\in G$ and $\epsilon,\mu\in\{\rs,\ls\}$. Since $\Aut^+(\Gamma)=\Aut^+(\Sigma)=G$ and $|G|=n>1$, both the bipartite graphs $\Gamma$ and $\Sigma$ are connected, and so the partition $\{G_\rs,G_\ls\}$ is preserved by $\varphi$. This implies that $\mu=-\epsilon$. Then~\eqref{eq065} yields that
\[
N_{\Sigma}(x_\epsilon)=N_{\Sigma}\big((1_\rs)^\varphi\big)=(S_\ls)^{\varphi}
=(1_\ls)^{S\varphi}=(1_\ls)^{\varphi\varphi^{-1}S\varphi}=(y_{-\epsilon})^{\varphi^{-1}S\varphi}
=(y_{-\epsilon})^{S^{\psi}},
\]
where $S^{\psi}=\{s^{\psi}\mid s\in S\}$. Thus $(T^{\epsilon1}x)_{-\epsilon}=N_{\Sigma}(x_\epsilon)=(y_{-\epsilon})^{S^{\psi}}$, and so $T$ is uniquely determined by $(\epsilon,x,y,\psi)$. Considering the choices for $(\epsilon,x,y,\psi)$, we conclude from $\Aut(\Gamma)=\Aut(\Sigma)=G$ and Lemma~\ref{LEM004}\eqref{LEM004.3} that the number of choices for $T$ is at most
\[
2\cdot n^2\cdot2^{\log_2^2n}=2^{\log_2^2n+2\log_2n+1},
\]
proving~\eqref{LEM023.1}.

Next assume that $G$ is abelian. Fix a subset $S$ of $G$ such that $\Aut(\HH(G,S))=G\rtimes\langle\iota\rangle$.  We enumerate the subsets $T$ of $G$ such that $\HH(G,T)$ is isomorphic to $\HH(G,S)$. By the notation and similar argument as above (note that $|\Aut(\Gamma)|=2n$ in this case), there are at most
\[
2\cdot n^2\cdot2^{\log_2^2(2n)}=2^{\log_2^2(2n)+2\log_2n+1}
\]
choices for $T$, which proves~\eqref{LEM023.2}.
\end{proof}

We conclude this section by proving Theorem~\ref{THM005}.

\begin{proof}[Proof of Theorem~$\ref{THM005}$]
Use the notation established at the beginning of Section~\ref{SUBSEC5.2}, and let $\HGR(G)$ be the set of HGRs of $G$ up to isomorphism. First assume that $G$ is nonabelian. To prove~Theorem~\ref{THM005}\eqref{THM005.1}, we need to estimate the ratio $|\HGR(G)|/|\mathcal{H}(G)|$. Let
\[
a_\varepsilon(n)=\frac{n^{0.5-\varepsilon}}{24(\log_2n)^{2.5}}-\frac{3\log_2^2n}{4}-15
\ \text{ and }\ b(n)=\log_2^2n+2\log_2n+1.
\]
Then Theorem~\ref{THM001}\eqref{THM001.1} and Lemma~\ref{LEM023}\eqref{LEM023.1} imply that
\[
|\mathcal{H}(G)\setminus\HGR(G)|<2^{n-a_\varepsilon(n)}\ \text{ and }\ |\HGR(G)|>\frac{2^n-2^{n-a_\varepsilon(n)}}{2^{b(n)}}.
\]
Hence we deduce from $|\mathcal{H}(G)|=|\mathcal{H}(G)\setminus\HGR(G)|+|\HGR(G)|$ that
\[
\frac{|\mathcal{H}(G)|}{|\HGR(G)|}=1+\frac{|\mathcal{H}(G)\setminus\HGR(G)|}{|\HGR(G)|}
<1+\frac{2^{n+b(n)-a_\varepsilon(n)}}{2^{n}-2^{n-a_\varepsilon(n)}}
=1+\frac{2^{b(n)-a_\varepsilon(n)}}{1-2^{-a_\varepsilon(n)}},
\]
which yields
\begin{align*}
\frac{|\HGR(G)|}{|\mathcal{H}(G)|}=\left(\frac{|\mathcal{H}(G)|}{|\HGR(G)|}\right)^{-1}
&>\frac{1-2^{-a_\varepsilon(n)}}{1-2^{-a_\varepsilon(n)}+2^{b(n)-a_\varepsilon(n)}}\\
&=1-\frac{2^{b(n)-a_\varepsilon(n)}}{1+2^{b(n)-a_\varepsilon(n)}-2^{-a_\varepsilon(n)}}
>1-2^{b(n)-a_\varepsilon(n)}.
\end{align*}
Since $b(n)-a_\varepsilon(n)=-\frac{n^{0.5-\varepsilon}}{24(\log_2n)^{2.5}}+\frac{7\log_2^2n}{4}+2\log_2n+16
<-h_\varepsilon(n)$, Theorem~\ref{THM005}\eqref{THM005.1} follows.

Next assume that $G$ is abelian. Let $\mathcal{A}(G)$ be the set of Haar graphs of $G$ up to isomorphism that have automorphism group $G\rtimes\langle\iota\rangle$. By an argument similar to the one above and applying Theorem~\ref{THM001}\eqref{THM001.2} and Lemma~\ref{LEM023}\eqref{LEM023.2}, we obtain that
\[
\frac{|\mathcal{A}(G)|}{|\mathcal{H}(G)|}>1-2^{\left(\log_2^2(2n)+2\log_2n+1\right)
-\left(\frac{n^{0.5-\varepsilon}}{24(\log_2n)^{2.5}}-\frac{3\log_2^2n}{4}-14\right)}=1-2^{-h_\varepsilon(n)},
\]
as~Theorem~\ref{THM005}\eqref{THM005.2} asserts.
\end{proof}

\section{Asymptotic enumeration of $m$-Cayley digraphs}\label{SEC6}

For an $m$-Cayley digraph $\mg\Cay(G,\mathcal{S})$ of a group $G$, set $G_i=\{(g,i)\mid g\in G\}$ for $i\in\{1,\ldots,m\}$. Since $G$ acts transitively on each $G_i$, each vertex in $G_i$ has the same out-valency and the same in-valency, denoted as $d^+_i(\mathcal{S})$ and $d^-_i(\mathcal{S})$, respectively.
In this section, we prove the conclusion of Theorem~\ref{THM004} for digraphs, which shows that almost all $m$-Cayley digraphs of a finite group $G$ are $\DMSRS$.

\begin{proposition}\label{PROP015}
Fix an integer $m\geq2$, and let $G$ be a finite group of order $n$. When $n$ is sufficiently large, the number of set-matrices $\mathcal{S}$ of $G$ such that $\mg\Cay(G,\mathcal{S})$ is not a $\DMSR$ is less than $m^2\cdot2^{m^2n}/\sqrt{n}$.
\end{proposition}

\begin{proof}
Let $\mathcal{Z}$ be the set of set-matrices $\mathcal{S}$ of $G$ such that $\mg\Cay(G,\mathcal{S})$ is not a $\DMSR$.
We first estimate the size of
\[
\mathcal{Z}_1:=\{\mathcal{S}\in\mathcal{Z}\mid\text{there exists $i\in\{1,\ldots,m\}$ such that $\Aut(\mg\Cay(G,\mathcal{S}))$ does not stabilize $G_i$}\}.
\]
Note that, if a vertex of $G_i$ is mapped by some automorphism of $\mg\Cay(G,\mathcal{S})$ into $G_j$, then $d^+_i(\mathcal{S})=d^+_j(\mathcal{S})$ and $d^-_i(\mathcal{S})=d^-_j(\mathcal{S})$. This means that, for any $\mathcal{S}\in\mathcal{Z}_1$, there exists a pair $(i,j)$ with $i<j$ such that
\begin{equation}\label{eq051}
d^+_i(\mathcal{S})+d^-_i(\mathcal{S})=d^+_j(\mathcal{S})+d^-_j(\mathcal{S}).
\end{equation}
Now, fix $(i,j)$ with $i<j$ satisfying~\eqref{eq051}. Clearly,
\begin{align*}
&d^+_i(\mathcal{S})=|S_{i,i}|+|S_{i,j}|+\sum_{k\neq i,j}|S_{i,k}|,
&&d^-_i(\mathcal{S})=|S_{i,i}|+|S_{j,i}|+\sum_{k\neq i,j}|S_{k,i}|,\\
&d^+_j(\mathcal{S})=|S_{j,j}|+|S_{j,i}|+\sum_{k\neq i,j}|S_{j,k}|,
&&d^-_i(\mathcal{S})=|S_{j,j}|+|S_{i,j}|+\sum_{k\neq i,j}|S_{k,j}|.
\end{align*}
Substituting these into~\eqref{eq051}, we obtain
\[
2|S_{i,i}|+\sum\limits_{k\neq i,j}(|S_{i,k}|+|S_{k,i}|)=2|S_{j,j}|+\sum_{k\neq i,j}(|S_{j,k}|+|S_{k,j}|).
\]
This indicates that the size of $S_{i,i}$ is determined by $S_{k,\ell}$ with $(k,\ell)\neq(i,i)$. Hence, we derive from Lemma~\ref{LEM003} that the number of $\mathcal{S}$ satisfying~\eqref{eq051} is at most
\begin{equation*}
2^{(m^2-1)n}\cdot\frac{2^{n}}{\sqrt{n}}=\frac{2^{m^2n}}{\sqrt{n}}.
\end{equation*}
Since there are $\binom{m}{2}$ choices for $(i,j)$ with $i<j$, we conclude that
\begin{equation}\label{eq052}
|\mathcal{Z}_1|\leq\binom{m}{2}\frac{2^{m^2n}}{\sqrt{n}}.
\end{equation}

By the definition of $\mathcal{Z}_1$, every $\mathcal{S}\in\mathcal{Z}\setminus\mathcal{Z}_1$ is such that $\Aut(\mg\Cay(G,\mathcal{S}))$ stabilizes each $G_i$ and hence induces a subgroup $A_i$ of $\Aut(\Cay(G,S_{i,i}))$ on $G_i$. Let
\[
\mathcal{Z}_2=\{\mathcal{S}\in\mathcal{Z}\setminus\mathcal{Z}_1\mid\text{there exists $i\in\{1,\ldots,m\}$ such that $A_i\not\cong G$}\}.
\]
Then we derive from Theorem~\ref{THM002} that
\begin{equation}\label{eq053}
|\mathcal{Z}_2|\leq m\cdot  2^{(m^2-1)n}\cdot 2^{n-\frac{bn^{0.499}}{4\log^3_2n}+2}=m\cdot  2^{m^2n-\frac{bn^{0.499}}{4\log^3_2n}+2},
\end{equation}
where $b$ is an absolute constant.

Now consider $\mathcal{S}\in\mathcal{Z}\setminus(\mathcal{Z}_1\cup\mathcal{Z}_2)$. It follows from $\mathcal{S}\notin\mathcal{Z}_1\cup\mathcal{Z}_2$ and $\Aut(\mg\Cay(G,\mathcal{S}))>G$ that the action of $\Aut(\mg\Cay(G,\mathcal{S}))$ on each $G_i$ has kernel $K_i>1$. Take an arbitrary $i$. Then there exists $j$ distinct from $i$ such that $K_i$ acts nontrivially on $G_j$. This implies that $S_{i,j}$ is a union of some $K_i$-orbits on $G_j$. Since $A_j\cong G$ acts regularly on $G_j$, the action of $K_i$ on $G_j$ is semiregular, and so there are at most $n/2$ orbits of $K_i$ on $G_j$. Hence we derive from Lemma~\ref{LEM004}\eqref{LEM004.4} that the choices of $S_{i,j}$ is at most $2^{\log^2_2n}\cdot2^{n/2}$. Consequently,
\begin{equation}\label{eq054}
|\mathcal{Z}\setminus(\mathcal{Z}_1\cup\mathcal{Z}_2)|\leq
2^{(m^2-m)n}\cdot\left((m-1)\cdot2^{\log^2_2n+\frac{n}{2}}\right)^m=(m-1)^m\cdot2^{m^2n-\frac{mn}{2}+m\log^2_2n}.
\end{equation}

Noting $m\geq2$ and combining~\eqref{eq052}--\eqref{eq054}, we obtain for sufficiently large $n$ that
\[
|\mathcal{Z}|=|\mathcal{Z}_1|+|\mathcal{Z}_2|+|\mathcal{Z}\setminus(\mathcal{Z}_1\cup\mathcal{Z}_2)|< m^2\frac{2^{m^2n}}{\sqrt{n}},
\]
which completes the proof.
\end{proof}

\begin{proof}[Proof of Theorem~$\ref{THM004}$ for digraphs]
Since $G$ has order $n$, there are exactly $2^{m^2n}$ set-matrices of $G$. Thus, by Proposition~\ref{PROP015}, the proportion of set-matrices $\mathcal{S}$ of $G$ such that $\mg\Cay(G,\mathcal{S})$ is a $\DMSR$ is greater than
\[
1-\frac{m^2\cdot2^{m^2n}/\sqrt{n}}{2^{m^2n}}=1-m^2/\sqrt{n}.\qedhere
\]
\end{proof}

\section{Asymptotic enumeration of $m$-Cayley graphs}\label{SEC7}

In this section, we prove Theorem~\ref{THM004} for graphs, which implies that almost all $m$-Cayley graphs of a finite group $G$ are $\GMSRS$.
We divide the proof into three cases in the following three subsections, respectively, which will be summarized at the end of the section to complete the proof. Recall the notations $c(S)$ and $\mathcal{I}(S)$ in Section~\ref{SUBSEC2.3}.

Recall that a digraph $\mg\Cay(G,\mathcal{S})$ is undirected if and only if $\mathcal{S}$ is inverse-closed (see Definition~\ref{DEF001}). Such a set-matrix $\mathcal{S}=(S_{i,j})_{m\times m}$ is uniquely determined by the subsets $S_{i,j}$ for $1\leq i<j\leq m$ and the inverse-closed subsets $S_{i,i}$ for $i\in\{1,\ldots,m\}$. So the number of inverse-closed set-matrices $\mathcal{S}$ of $G$ is $2^d$, where
\[
d=\binom{m}{2}|G|+mc(G).
\]
For an $m$-Cayley graph $\mg\Cay(G,\mathcal{S})$, set $G_i=\{(g,i)\mid g\in G\}$ for $i\in\{1,\ldots,m\}$. Since $G$ acts transitively on each $G_i$, every vertex of $G_i$ has the same valency, denoted as $d_i(\mathcal{S})$.
\begin{lemma}\label{LEM024}
Fix an integer $m\geq2$. Let $G$ be a group of order $n$. Then the number of inverse-closed set-matrices $\mathcal{S}$ of $G$ such that $\Aut(\mg\Cay(G,\mathcal{S}))$ does not stabilize $G_i$ for some $i\in\{1,\ldots,m\}$ is at most $m(m-1)2^{d}/\sqrt{n}$, where $d=\binom{m}{2}n+mc(G)$.
\end{lemma}

\begin{proof}
Let $\mathcal{Z}$ be the set of inverse-closed set-matrices $\mathcal{S}$ of $G$ such that $\Aut(\mg\Cay(G,\mathcal{S}))$ does not stabilize $G_i$ for some $i\in\{1,\ldots,m\}$.
Note that, if a vertex of $G_i$ is mapped by some automorphism of $\mg\Cay(G,\mathcal{S})$ into $G_j$, then $d_i(\mathcal{S})=d_j(\mathcal{S})$. This means that, for any $\mathcal{S}\in\mathcal{Z}_1$, there exists a pair $(i,j)$ with $i<j$ such that $d_i(\mathcal{S})=d_j(\mathcal{S})$. Now, fix $(i,j)$ with $i<j$ that satisfies $d_i(\mathcal{S})=d_j(\mathcal{S})$. Then
\[
|S_{i,i}|+|S_{i,j}|+\sum_{k\neq i,j}|S_{i,k}|=d_i(\mathcal{S})=
d_j(\mathcal{S})=|S_{j,j}|+|S_{j,i}|+\sum_{k\neq i,j}|S_{j,k}|.
\]
For a subset $X$ of $G$, denote $\mathcal{N}(X)=X\setminus\mathcal{I}(X)$. Since $S_{i,j}=S_{j,i}^{-1}$, the equation above becomes
\begin{equation}\label{eq055}
|\mathcal{I}(S_{i,i})|+|\mathcal{N}(S_{i,i})|+\sum_{k\neq i,j}|S_{i,k}|=|S_{j,j}|+\sum_{k\neq i,j}|S_{j,k}|.
\end{equation}
When $G$ is not an elementary abelian $2$-group, the equation~\eqref{eq055} indicates that the size of $\mathcal{N}(S_{i,i})$ is determined by $\mathcal{I}(S_{i,i})$ and $S_{k,\ell}$ with $(k,\ell)\neq(i,i)$. In this case, we derive from Lemmas~\ref{LEM003} and~\ref{LEM006} that the number of $\mathcal{S}$ satisfying~\eqref{eq055} is at most
\begin{equation*}
\frac{2^{d}}{2^{c(G)}}\cdot 2^{|\mathcal{I}(G)|}\cdot \frac{2^{|\mathcal{N}(G)|}}{\sqrt{|\mathcal{N}(G)|}}=
\frac{2^{d}}{\sqrt{|\mathcal{N}(G)|}}\leq
\frac{2^{d}}{\sqrt{n/4}}=
\frac{2^{d+1}}{\sqrt{n}}.
\end{equation*}
When $G$ is an elementary abelian $2$-group, we have $\mathcal{I}(S_{i,i})=S_{i,i}$ and $\mathcal{N}(S_{i,i})=\emptyset$, and so~\eqref{eq055} implies that the size of $S_{i,i}$ is determined by the sets $S_{k,\ell}$ with $(k,\ell)\neq(i,i)$. In this case, the number of $\mathcal{S}$ satisfying~\eqref{eq055} is at most
\begin{equation*}
\frac{2^{d}}{2^{c(G)}}\cdot \frac{2^{c(G)}}{\sqrt{c(G)}}=
\frac{2^{d}}{\sqrt{c(G)}}=
\frac{2^{d}}{\sqrt{n}}<
\frac{2^{d+1}}{\sqrt{n}}.
\end{equation*}
Since there are $\binom{m}{2}$ choices for $(i,j)$ with $i<j$, summing up the two cases, we conclude that
\[
|\mathcal{Z}|\leq\binom{m}{2}\frac{2^{d+1}}{\sqrt{n}}=m(m-1)\frac{2^{d}}{\sqrt{n}},
\]
as the lemma asserts.
\end{proof}

\subsection{\boldmath{$m$}-Cayley graphs on abelian groups}\label{SUBSEC7.1}

In this section, we concentrate on the case when the group is abelian with exponent greater than $2$.
We first present a useful result and give a technical lemma as follows.

\begin{theorem}{\rm(\cite[Theorem~1.7]{DSV2016})}\label{THM006}
Let $G$ be an abelian group of order $n$, and let $\iota$ be the automorphism of $G$ defined by $\iota\colon g\mapsto g^{-1}$ for every $g\in G$. Then the number of inverse-closed subsets $S$ such that $\Aut(\Cay(G,S))\neq G\rtimes\langle\iota\rangle$ is at most $2^{c(G)-n/24+2\log^2_2n+2}$.
\end{theorem}

\begin{lemma}\label{LEM025}
Let $G$ be an abelian group of exponent greater than $2$, and let $\iota$ be the automorphism of $G$ defined by $\iota\colon g\mapsto g^{-1}$ for every $g\in G$. Then for each non-identity $b\in G\rtimes\langle\iota\rangle$, the number of orbits of $\langle b\rangle$ on $G$ is at most $5|G|/6$.
\end{lemma}

\begin{proof}
Since $G$ is abelian and of exponent greater than $2$, we may assume $G=C_k\times H$ for some $k>2$ and subgroup $H$. Note that, if $x\in G$ is fixed by $g\iota$, then
\[
x=x^{g\iota}=(xg)^{\iota}=g^{-1}x^{-1},
\]
which implies $x\in\{y\mid y^2=g^{-1}\}$. Moreover, we consider the homomorphism $\psi\colon G\to\{g^2\mid g\in G\}$, which maps $g$ to $g^2$. As $k>2$, the image $\psi(G)$ has size at least $k/2$. So, we obtain
\[
|\Ker(\psi)|=\frac{|G|}{|\psi(G)|}\leq\frac{2|G|}{k}\leq\frac{2}{3}|G|.
\]
This implies that $|\Fix(g\iota)|\leq\frac{2}{3}|G|$. Hence the number of orbits of $\langle g\iota\rangle$ on $G$ is at most
\[
|\Fix(g\iota)|+\frac{|G|-|\Fix(g\iota)|}{2}=\frac{|G|+|\Fix(g\iota)|}{2}\leq\frac{5}{6}|G|,
\]
as the lemma asserts.
\end{proof}

We conclude this section with the following proposition.

\begin{proposition}\label{PROP011}
Fix an integer $m\geq2$, and let $G$ be an abelian group with exponent greater than $2$ and order $n$. When $n$ is sufficiently large, the number of inverse-closed set-matrices $\mathcal{S}$ of $G$ such that $\mg\Cay(G,\mathcal{S})$ is not a $\GMSR$ is less than $m^22^{d}/\sqrt{n}$, where $d=\binom{m}{2}n+mc(G)$.
\end{proposition}

\begin{proof}
Let $\mathcal{Z}$ be the set of inverse-closed set-matrices $\mathcal{S}$ of $G$ such that $\mg\Cay(G,\mathcal{S})$ is not a $\GMSR$.
By Lemma~\ref{LEM024}, the size of the set
\[
\mathcal{Z}_1:=\{\mathcal{S}\in\mathcal{Z}\mid\text{there exists $i\in\{1,\ldots,m\}$ such that $\Aut(\mg\Cay(G,\mathcal{S}))$ does not stabilize $G_i$}\}.
\]
is at most $m(m-1)2^{d}/\sqrt{n}$. By the definition of $\mathcal{Z}_1$, every $\mathcal{S}\in\mathcal{Z}\setminus\mathcal{Z}_1$ is such that $\Aut(\mg\Cay(G,\mathcal{S}))$ stabilizes each $G_i$ and hence induces a subgroup $A_i$ of $\Aut(\Cay(G,S_{i,i}))$ on $G_i$. Let
\[
\mathcal{Z}_2=\{\mathcal{S}\in\mathcal{Z}\setminus\mathcal{Z}_1\mid\text{there exists $i\in\{1,\ldots,m\}$ such that $A_i\neq G\rtimes\langle\iota\rangle$}\}.
\]
Then we derive from Theorem~\ref{THM006} that
\begin{equation}\label{eq057}
|\mathcal{Z}_2|\leq m\cdot\frac{2^{d}}{2^{c(G)}}\cdot 2^{c(G)-\frac{n}{24}+2\log_2^2n+2}=m\cdot2^{d-\frac{n}{24}+2\log_2^2n+2}.
\end{equation}

Now consider
\[
\mathcal{Z}_3=\{\mathcal{S}\in\mathcal{Z}\setminus(\mathcal{Z}_1\cup\mathcal{Z}_2)
\mid\Aut(\mg\Cay(G,\mathcal{S}))\neq G\rtimes\langle\iota\rangle\}.
\]
Take $\mathcal{S}\in\mathcal{Z}_3$. It follows from $\mathcal{S}\notin\mathcal{Z}_1\cup\mathcal{Z}_2$ and $\Aut(\mg\Cay(G,\mathcal{S}))\neq G\rtimes\langle\iota\rangle$ that the action of $\Aut(\mg\Cay(G,\mathcal{S}))$ on each $G_i$ has kernel $K_i>1$. In particular, there exists $j\in\{2,\ldots,m\}$ such that some element of $K_1$ induces a nontrivial action $\alpha_{1j}\in G\rtimes\langle\iota\rangle$ on $G_j$. It follows that $S_{1,j}$ is a union of some $\langle\alpha_{1j}\rangle$-orbits on $G_j$. Moreover, according to Lemma~\ref{LEM025}, there are at most $5n/6$ orbits of $\langle\alpha_{1j}\rangle$ on $G_j$. Hence we deduce that the number of choices for $S_{1,j}$ is at most $2n\cdot2^{5n/6}$. Consequently,
\begin{equation}\label{eq010}
|\mathcal{Z}_3|\leq
(m-1)\frac{2^{d}}{2^{n}}\cdot
2n\cdot2^{\frac{5}{6}n}\leq (m-1)2^{d-\frac{n}{6}+\log_2n+1}.
\end{equation}

Finally, we estimate the cardinality of $\mathcal{Z}_4:=\mathcal{Z}\setminus(\mathcal{Z}_1\cup\mathcal{Z}_2\cup\mathcal{Z}_3)$. Let $\mathcal{S}\in\mathcal{Z}_4$. Then $\Aut(\mg\Cay(G,\mathcal{S}))= G\rtimes\langle\iota\rangle$. Clearly, the automorphism $\iota$ keeps each $1_i$ invariant for each $i\in\{1,\ldots,m\}$. Thus, for any pair $(i,j)$ with $i<j$, the subset $S_{i,j}$ is a union of $\langle\iota\rangle$-orbits on $G_j$.  Since $\langle\iota\rangle$ has $c(G)$ orbits on $G$, and Lemma~\ref{LEM006} implies $c(G)=(|G|+|\mathcal{I}(G)|)/2\le 7n/8$, we derive that that
\begin{equation}\label{eq056}
|\mathcal{Z}_4|=
2^{mc(G)}\cdot \left(2^{c(G)}\right)^{\frac{m(m-1)}{2}}=2^d\cdot(2^{c(G)-n})^{\frac{m(m-1)}{2}}
\le 2^d\cdot\left(2^{-\frac{n}{8}}\right)^{\frac{m(m-1)}{2}}=
2^{d-\frac{m(m-1)}{16}n}.
\end{equation}

Combining $|\mathcal{Z}_1|\leq m(m-1)2^{d}/\sqrt{n}$ and~\eqref{eq057}--\eqref{eq056}, we obtain for sufficiently large $n$ that
\[
|\mathcal{Z}|=|\mathcal{Z}_1|+|\mathcal{Z}_2|+|\mathcal{Z}_3|+|\mathcal{Z}_4|< m^2\frac{2^{d}}{\sqrt{n}},
\]
which completes the proof.
\end{proof}

\subsection{\boldmath{$m$}-Cayley graphs on generalised dicyclic groups}\label{SUBSEC7.2}

Let $A$ be an abelian group of even order and of exponent greater than $2$, and let $y$ be an involution of $A$. The \emph{generalised dicyclic group} $\Dic(A,y,x)$ is the group $\langle A,x\mid x^2=y,\,a^x=a^{-1},\,\forall a\in A\rangle$. In this section, we establish asymptotic results for $m$-Cayley graphs of $\Dic(A,y,x)$. The conclusions are divided into Propositions~\ref{PROP012} and~\ref{PROP013}, addressing two cases based on $\Dic(A,y,x)\cong Q_8\times C_2^{\ell}$ or not, where $C_2^{\ell}$ is an elementary abelian $2$-group. Let us start with the case $\Dic(A,y,x)\cong Q_8\times C_2^{\ell}$.


\begin{notation}\label{NOT002}
Let $G=Q_8\times E$ with $E=C_2^{\ell}$ for some $\ell\geq0$, and we label the element of $Q_8$ with $\{1,-1,\textbf{i},-\textbf{i},\textbf{j},-\textbf{j},\textbf{k},-\textbf{k}\}$ in the usual way. Define the following permutations of $G$: for $\textbf{u}\in\{\textbf{i},\textbf{j},\textbf{k}\}$, $\alpha_{\textbf{u}}$ is the involution which swaps $\textbf{u} e$ and $-\textbf{u} e$ for every $e\in E$ while fixes every other element of $G$. Let $M(G)=\langle G,\iota_\textbf{i},\iota_\textbf{j},\iota_\textbf{k}\rangle$, viewed as a permutation group on $G$ with $G$ acting regularly on itself by right multiplication.
\end{notation}

We need the following asymptotic result and technical lemma.

\begin{theorem}{\rm\cite[Theorems~1.6]{MSV2015}}\label{THM007}
Let $G$, $M(G)$ as in Notation\textup{~\ref{NOT002}} and $|G|=n$. Then the number of inverse-closed subsets $S$ of $G$ such that $\Aut(\Cay(G,S))\neq M(G)$ is at most $2^{c(G)-n/512+2\log^2_2n+2}$.
\end{theorem}

\begin{lemma}\label{LEM026}
Let $G=Q_8\times E$ and $M(G)$ as in Notation\textup{~\ref{NOT002}}. Then for any non-identity element $s\in M(G)$, the number of orbits of $\langle s\rangle$ on $G$ is at most $7|G|/8$.
\end{lemma}

\begin{proof}
Let $\alpha=\iota_{\textbf{i}}$, $\beta=\iota_{\textbf{j}}$ and $\gamma=\iota_{\textbf{k}}$. Since $|M(G)\,{:}\,G|=8$ (see~\cite[Lemma~4.3]{MSV2015}), and $\alpha,\beta,\gamma$ commute pairwise, the group $M(G)$ has the coset decomposition
\[
M(G)=\bigcup_{i,j,k\in\{0,1\}} G\alpha^{i}\beta^{j}\gamma^{k}.
\]
It is straightforward to verify that
\begin{align*}
&\Fix(qe\alpha^{i}\beta^{j}\gamma^{k})=\emptyset\, \text{ for each $q\in Q_8$ and non-identity $e\in E$},\\
&\Fix(q\alpha^{i}\beta^{j}\gamma^{k})\cap \left((Q_8\setminus\{\pm1\})\times E\right)=\emptyset\,\text{ for each non-identity $q\in Q_8$},\\
&\Fix(\alpha\beta^{j}\gamma^{k})\cap \left(\{\pm\alpha\}\times E\right)=\emptyset,\\
&\Fix(\alpha^i\beta\gamma^{k})\cap \left(\{\pm\beta\}\times E\right)=\emptyset,\\
&\Fix(\alpha^{i}\beta^{j}\gamma)\cap \left(\{\pm\gamma\}\times E\right)=\emptyset.
\end{align*}
This implies that for any $s\in M(G)$ with $s\ne 1$, the fixed point ratio $|\Fix(s)|/|G|$ on $G$ is at most $3/4$. Thus, the number of orbits of $\langle s\rangle$ on $G$ is at most $7|G|/8$, as the lemma asserts.
\end{proof}

We are now ready to prove the asymptotic result for $m$-Cayley graphs of $Q_8\times C_2^{\ell}$.

\begin{proposition}\label{PROP012}
Fix an integer $m\geq2$, and let $G=Q_8\times C_2^{\ell}$ with order~$n$. When $n$ is sufficiently large, the number of inverse-closed set-matrices $\mathcal{S}$ of $G$ such that $\mg\Cay(G,\mathcal{S})$ is not a $\GMSR$ is less than $m^22^{d}/\sqrt{n}$, where $d=\binom{m}{2}n+mc(G)$.
\end{proposition}

\begin{proof}
We use Notation~\ref{NOT002}. Let $\mathcal{Z}$ be the set of inverse-closed set-matrices $\mathcal{S}$ of $G$ such that $\mg\Cay(G,\mathcal{S})$ is not a $\GMSR$,
and let
\[
\mathcal{Z}_1=\{\mathcal{S}\in\mathcal{Z}\mid\text{there exists $i\in\{1,\ldots,m\}$ such that $\Aut(\mg\Cay(G,\mathcal{S}))$ does not stabilize $G_i$}\}.
\]
By Lemma~\ref{LEM024}, the cardinality of $\mathcal{Z}_1$ satisfies
\begin{equation}\label{eq105}
|\mathcal{Z}_1|\leq m(m-1)2^{d}/\sqrt{n}.
\end{equation}
By the definition of $\mathcal{Z}_1$, every $\mathcal{S}\in\mathcal{Z}\setminus\mathcal{Z}_1$ is such that $\Aut(\mg\Cay(G,\mathcal{S}))$ stabilizes each $G_i$ and hence induces a subgroup $A_i$ of $\Aut(\Cay(G,S_{i,i}))$ on $G_i$. Let
\[
\mathcal{Z}_2=\{\mathcal{S}\in\mathcal{Z}\setminus\mathcal{Z}_1\mid\text{there exists $i\in\{1,\ldots,m\}$ such that $A_i\ncong M(G)$}\}.
\]
Then we derive from Theorem~\ref{THM007} that
\begin{equation}\label{eq106}
|\mathcal{Z}_2|\leq m\cdot\frac{2^{d}}{2^{c(G)}}\cdot 2^{c(G)-\frac{n}{512}+2\log^2_2n+2}=m\cdot2^{d-\frac{n}{512}+2\log^2_2n+2}.
\end{equation}

Now consider $\mathcal{S}\in\mathcal{Z}\setminus(\mathcal{Z}_1\cup\mathcal{Z}_2)$. As $\Aut(\mg\Cay(G,\mathcal{S}))>G$, we have $\Aut(\mg\Cay(G,\mathcal{S}))_x>1$ for any $x\in G_1$. Moreover, since $\mathcal{S}\in\mathcal{Z}\setminus(\mathcal{Z}_1\cup\mathcal{Z}_2)$, there exists $j\in\{2,\ldots,m\}$ such that some element of $\Aut(\mg\Cay(G,\mathcal{S}))_x$ induces a nontrivial permutation $\alpha_{j}\in M(G)$ on $G_j$. Hence $S_{1,j}$ is a union of some $\langle\alpha_{j}\rangle$-orbits on $G_j$. On the other hand, according to Lemma~\ref{LEM026}, there are at most $7n/8$ orbits of $\langle\alpha_{j}\rangle$ on $G_j$. Hence we deduce that the number of choices of $S_{1,j}$ is at most $2n\cdot2^{7n/8}$. Consequently,
\begin{equation}\label{eq128}
|\mathcal{Z}\setminus(\mathcal{Z}_1\cup\mathcal{Z}_2)|\leq
(m-1)\frac{2^{d}}{2^{n}}\cdot
2n\cdot2^{\frac{7}{8}n}\leq (m-1)2^{d-\frac{n}{8}+\log_2n+1}.
\end{equation}

Combining~\eqref{eq105}--\eqref{eq128}, we obtain for sufficiently large $n$ that
\[
|\mathcal{Z}|=|\mathcal{Z}_1|+|\mathcal{Z}_2|+|\mathcal{Z}\setminus(\mathcal{Z}_1\cup\mathcal{Z}_2)|< m^2\frac{2^{d}}{\sqrt{n}},
\]
which completes the proof.
\end{proof}

In the rest of this section, we handle the case $\Dic(A,y,x)\ncong Q_8\times C_2^{\ell}$.

\begin{notation}\label{NOT003}
Let $G=\Dic(A,y,x)$ be a generalised dicyclic group such that $G\not\cong Q_8\times C_2^{\ell}$, where $Q_8$ is the quaternion group. Let $\iota$ be the automorphism of $G$ fixing $A$ pointwise and mapping every element of $G\setminus A$ to its inverse, and let $M(G)=G\rtimes\langle\iota\rangle$.
\end{notation}

Similarly, we need the following asymptotic result and technical lemma.

\begin{theorem}{\rm(\cite[Theorems~1.5]{MSV2015})}\label{THM008}
Let $G=\Dic(A,y,x)$ be a generalised dicyclic group of order $n$ such that $G\not\cong Q_8\times C_2^{\ell}$ for any $\ell\geq0$. Then the number of inverse-closed subsets $S$ of $G$ with $\Aut(\Cay(G,S))\neq M(G)$ is at most $2^{c(G)-n/48+2\log^2_2n+5}$.
\end{theorem}

\begin{lemma}\label{LEM027}
Let $G=\Dic(A,y,x)$ be a generalised dicyclic group such that $G\not\cong Q_8\times C_2^{\ell}$. Then for each non-identity $s\in M(G)$, the number of orbits of $\langle s\rangle$ on $G$ is at most $3|G|/4$.
\end{lemma}

\begin{proof}
Note that $M(G)=G\cup G\iota$ and $G=A\cup Ax$. It is straightforward to verify that
\begin{align*}
&\Fix(g)=\emptyset\,\text{ for each $g\in G$},\\
&\Fix(ax\iota)\cap Ax=\emptyset\,\text{ for each $a\in A$},\\
&\Fix(a\iota)\cap A=\emptyset\,\text{ for each non-identity $a\in A$},\\
&\Fix(\iota)\cap Ax=\emptyset.
\end{align*}
This implies that for each $s\in M(G)$, we have $|\Fix(s)|\leq|G|/2$. Hence the number of orbits of $\langle s\rangle$ on $G$ is at most
\[
|\Fix(s)|+\frac{|G|-|\Fix(s)|}{2}=\frac{|G|+|\Fix(s)|}{2}\leq\frac{3}{4}|G|,
\]
as the lemma asserts.
\end{proof}

We close this section with the following proposition.

\begin{proposition}\label{PROP013}
Fix an integer $m\geq2$. Let $G$ be a generalised dicyclic group of order $n$ such that $G\not\cong Q_8\times C_2^{\ell}$ for each $\ell\geq0$. When $n$ is sufficiently large, the number of inverse-closed set-matrices $\mathcal{S}$ of $G$ such that $\mg\Cay(G,\mathcal{S})$ is not a $\GMSR$
is less than $m^22^{d}/\sqrt{n}$, where $d=\binom{m}{2}n+mc(G)$.
\end{proposition}

\begin{proof}
Let $\mathcal{Z}=\{\mathcal{S}\mid\mathcal{S}\text{ is inverse-closed},\,\Aut(\mg\Cay(G,\mathcal{S}))>G\}$ and let
\[
\mathcal{Z}_1=\{\mathcal{S}\in\mathcal{Z}\mid\text{there exists $i\in\{1,\ldots,m\}$ such that $\Aut(\mg\Cay(G,\mathcal{S}))$ does not stabilize $G_i$}\}.
\]
By Lemma~\ref{LEM024}, the cardinality of $\mathcal{Z}_1$ satisfies
\begin{equation}\label{eq058}
|\mathcal{Z}_1|\leq m(m-1)2^{d}/\sqrt{n}.
\end{equation}
By the definition of $\mathcal{Z}_1$, every $\mathcal{S}\in\mathcal{Z}\setminus\mathcal{Z}_1$ is such that $\Aut(\mg\Cay(G,\mathcal{S}))$ stabilizes each $G_i$ and hence induces a subgroup $A_i$ of $\Aut(\Cay(G,S_{i,i}))$ on $G_i$. Let
\[
\mathcal{Z}_2=\{\mathcal{S}\in\mathcal{Z}\setminus\mathcal{Z}_1\mid\text{there exists $i\in\{1,\ldots,m\}$ such that $A_i\ncong M(G)$}\}.
\]
Then we derive from Theorem~\ref{THM008} that
\begin{equation}\label{eq059}
|\mathcal{Z}_2|\leq m\cdot\frac{2^{d}}{2^{c(G)}}\cdot 2^{c(G)-\frac{n}{48}+2\log^2_2n+5}=m\cdot2^{d-\frac{n}{48}+2\log^2_2n+5}.
\end{equation}

Now consider $\mathcal{S}\in\mathcal{Z}\setminus(\mathcal{Z}_1\cup\mathcal{Z}_2)$. As $\Aut(\mg\Cay(G,\mathcal{S}))\neq G$, it holds $\Aut(\mg\Cay(G,\mathcal{S}))_v>1$ for any $v\in G_1$. Moreover, since $\mathcal{S}\in\mathcal{Z}\setminus(\mathcal{Z}_1\cup\mathcal{Z}_2)$, there exists $j\in\{2,\ldots,m\}$ such that some element of $\Aut(\mg\Cay(G,\mathcal{S}))_v$ induces a nontrivial permutation $\alpha_{j}\in M(G)$ on $G_j$. Hence $S_{1,j}$ is a union of some $\langle\alpha_{j}\rangle$-orbits on $G_j$. According to Lemma~\ref{LEM027}, there are at most $3n/4$ orbits of $\langle\alpha_{j}\rangle$ on $G_j$. Hence we derive that the choices of $S_{1,j}$ is at most $|M(G)|\cdot2^{3n/4}=2n\cdot2^{3n/4}$. Consequently,
\begin{equation}\label{eq060}
|\mathcal{Z}\setminus(\mathcal{Z}_1\cup\mathcal{Z}_2)|\leq
(m-1)\frac{2^{d}}{2^{n}}\cdot
2n\cdot2^{\frac{3}{4}n}\leq (m-1)2^{d-\frac{n}{4}+\log_2n+1}.
\end{equation}

Combining~\eqref{eq058}--\eqref{eq060}, we obtain for sufficiently large $n$ that
\[
|\mathcal{Z}|=|\mathcal{Z}_1|+|\mathcal{Z}_2|+|\mathcal{Z}\setminus(\mathcal{Z}_1\cup\mathcal{Z}_2)|< m^2\frac{2^{d}}{\sqrt{n}},
\]
which completes the proof.
\end{proof}

\subsection{\boldmath{$m$}-Cayley graphs on other groups}\label{SUBSEC7.3}

In view of~Sections~\ref{SUBSEC7.1} and~\ref{SUBSEC7.2}, we are left to deal with the case that $G$ is neither abelian with exponent greater than $2$ nor generalised dicyclic.

\begin{proposition}\label{PROP014}
Fix an integer $m\geq2$. Let $G$ be a group of order $n$ such that $G$ is neither abelian of exponent greater than $2$ nor generalised dicyclic. When $n$ is sufficiently large, the number of inverse-closed set-matrices $\mathcal{S}$ of $G$ such that
$\mg\Cay(G,\mathcal{S})$ is not a $\GMSR$
is less than $m^22^{d}/\sqrt{n}$, where $d=\binom{m}{2}n+mc(G)$.
\end{proposition}

\begin{proof}
Let $\mathcal{Z}=\{\mathcal{S}\mid\mathcal{S}\text{ is inverse-closed},\,\Aut(\mg\Cay(G,\mathcal{S}))>G\}$. We first estimate the size of
\[
\mathcal{Z}_1:=\{\mathcal{S}\in\mathcal{Z}\mid\text{there exists $i\in\{1,\ldots,m\}$ such that $\Aut(\mg\Cay(G,\mathcal{S}))$ does not stabilize $G_i$}\}.
\]
Since there are $\binom{m}{2}$ choices for $(i,j)$ with $i<j$, we derive from Lemma~\ref{LEM024} that
\begin{equation}\label{eq107}
|\mathcal{Z}_1|\leq m(m-1)2^{d}/\sqrt{n}.
\end{equation}

By the definition of $\mathcal{Z}_1$, every $\mathcal{S}\in\mathcal{Z}\setminus\mathcal{Z}_1$ is such that $\Aut(\mg\Cay(G,\mathcal{S}))$ stabilizes each $G_i$ and hence induces a subgroup $A_i$ of $\Aut(\Cay(G,S_{i,i}))$ on $G_i$. Let
\[
\mathcal{Z}_2=\{\mathcal{S}\in\mathcal{Z}\setminus\mathcal{Z}_1\mid\text{there exists $i\in\{1,\ldots,m\}$ such that $A_i\not\cong G$}\}.
\]
Then we derive from Theorem~\ref{THM003} and Lemma~\ref{LEM001} that
\begin{equation}\label{eq108}
|\mathcal{Z}_2|\leq m\cdot\frac{2^{d}}{2^{c(G)}}\cdot 2^{c(G)-\frac{n^{0.499}}{8\log^3_2n}+\log_{2}^2n+3}=m\cdot 2^{d-\frac{n^{0.499}}{8\log^3_2n}+\log_{2}^2n+3}.
\end{equation}

Now consider $\mathcal{S}\in\mathcal{Z}\setminus(\mathcal{Z}_1\cup\mathcal{Z}_2)$. It follows from $\mathcal{S}\notin\mathcal{Z}_1\cup\mathcal{Z}_2$ and $\Aut(\mg\Cay(G,\mathcal{S}))>G$ that the action of $\Aut(\mg\Cay(G,\mathcal{S}))$ on each $G_i$ has kernel $K_i>1$. In particular, there exists $j\in\{2,\ldots,m\}$ such that $K_1$ acts nontrivially on $G_j$. It follows that $S_{1,j}$ is a union of some $K_1$-orbits on $G_j$. Since $A_j\cong G$ acts regularly on $G_j$, the action of $K_1$ on $G_j$ is semiregular, and so there are at most $n/2$ orbits of $K_1$ on $G_j$. Hence we derive from Lemma~\ref{LEM004}\eqref{LEM004.4} that the choices of $S_{1,j}$ is at most $2^{\log^2_2n}\cdot2^{n/2}$. Consequently,
\begin{equation}\label{eq109}
|\mathcal{Z}\setminus(\mathcal{Z}_1\cup\mathcal{Z}_2)|\leq
(m-1)\frac{2^{d}}{2^{n}}\cdot
2^{\log^2_2n+\frac{n}{2}}\leq (m-1)2^{d-\frac{n}{2}+\log^2_2n}.
\end{equation}

Combining~\eqref{eq107}--\eqref{eq109}, we obtain for sufficiently large $n$ that
\[
|\mathcal{Z}|=|\mathcal{Z}_1|+|\mathcal{Z}_2|+|\mathcal{Z}\setminus(\mathcal{Z}_1\cup\mathcal{Z}_2)|< m^2\frac{2^{d}}{\sqrt{n}},
\]
which completes the proof.
\end{proof}

\begin{proof}[Proof of Theorem~$\ref{THM004}$ for graphs]
As stated at the beginning of this section, there are exactly $2^{d}$ inverse-closed set-matrices of $G$, where $d=\binom{m}{2}n+mc(G)$. Combining this with Propositions~\ref{PROP011},~\ref{PROP012},~\ref{PROP013} and~\ref{PROP014}, we conclude that the proportion of inverse-closed set-matrices $\mathcal{S}$ of $G$ such that $\mg\Cay(G,\mathcal{S})$ is not a $\GMSR$ is less than
\[
\frac{m^22^{d}/\sqrt{n}}{2^{d}}=m^2/\sqrt{n}.\qedhere
\]
\end{proof}

\section{$m$-partite graphical semiregular representations}\label{SEC8}

Recall from Definition~\ref{DEF001} that a set-matrix $\mathcal{S}=(S_{i,j})_{m\times m}$ is skew if it is inverse-closed and $S_{i,i}=\emptyset$ for each $i\in\{1,\ldots,m\}$. Such a set-matrix $\mathcal{S}=(S_{i,j})_{m\times m}$ is uniquely determined by the subsets $S_{i,j}$ for $1\leq i<j\leq m$, and so there are exactly $2^d$ skew set-matrices of $G$, where $d=\binom{m}{2}n$.

\begin{proof}[Proof of Theorem~$\ref{THM009}$]
Write $d=\binom{m}{2}n$, and let $\mathcal{Z}$ be the set of skew set-matrices $\mathcal{S}$ of $G$ such that $\Aut(\mg\Cay(G,\mathcal{S}))>G$. Let
\[
\mathcal{Y}:=\{\mathcal{S}\in\mathcal{Z}\mid\text{there exists $i\in\{1,\ldots,m\}$ such that $\Aut(\mg\Cay(G,\mathcal{S}))$ does not stabilize $G_i$}\}.
\]
Observe that if a vertex of $G_i$ is mapped by some automorphism of $\mg\Cay(G,\mathcal{S})$ into $G_j$, then $d_i(\mathcal{S})=d_j(\mathcal{S})$. This means that, for each $\mathcal{S}\in\mathcal{Y}$, there exists a pair $(i,j)$ with $i<j$ such that $d_i(\mathcal{S})=d_j(\mathcal{S})$. Fix some $(i,j)$ with $i<j$ that satisfies $d_i(\mathcal{S})=d_j(\mathcal{S})$. Then
\begin{equation}\label{eq471}
|S_{i,j}|+\sum_{k\neq i,j}|S_{i,k}|=d_i(\mathcal{S})=
d_j(\mathcal{S})=|S_{j,i}|+\sum_{k\neq i,j}|S_{j,k}|.
\end{equation}
Take $x\in\{1,\ldots,m\}\setminus\{i,j\}$.
Then~\eqref{eq471} indicates that the size of $S_{i,x}$ is determined by the sets $S_{k,\ell}$ with $(k,\ell)\neq(i,x)$. We derive from Lemma~\ref{LEM003} that the number of $\mathcal{S}$ satisfying~\eqref{eq471} is at most $2^{d-n}\cdot2^{n}/\sqrt{n}=2^{d}/\sqrt{n}$. Since there are $\binom{m}{2}$ choices for $(i,j)$ with $i<j$, we deduce that
\begin{equation}\label{eq252}
|\mathcal{Y}|\leq\binom{m}{2}\frac{2^{d}}{\sqrt{n}}.
\end{equation}

Now consider $\mathcal{S}\in\mathcal{Z}\setminus\mathcal{Y}$. Note that for each pair $(i,j)$, the parts $G_i$ and $G_j$ induce the Haar graph $\HH(G,S_{i,j})$. Since $\Aut(\mg\Cay(G,\mathcal{S}))$ stabilizes each part and $\Aut(\mg\Cay(G,\mathcal{S}))>G$, there exists a pair $(i,j)$ with $i<j$ such that $\Aut^{+}(\HH(G,S_{i,j}))>G$. Applying Proposition~\ref{PROP010} with $\varepsilon=0.1$, we obtain that, when $n$ is sufficiently large, there are at most
\[
2^{n-\frac{n^{0.4}}{24(\log_2n)^{2.5}}+\frac{3\log_2^2n}{4}+15}
\]
choices for $S_{i,j}$. Counting the choices for $(i,j)$ with $i<j$, we conclude that
\[
|\mathcal{Z}\setminus\mathcal{Y}|
\leq\binom{m}{2}2^{d-n}\cdot2^{n-\frac{n^{0.4}}{24(\log_2n)^{2.5}}+\frac{3\log_2^2n}{4}+15}
=\binom{m}{2}2^{d-\frac{n^{0.4}}{24(\log_2n)^{2.5}}+\frac{3\log_2^2n}{4}+15}.
\]
This together with~\eqref{eq252} yields that, for sufficiently large $n$,
\[
|\mathcal{Z}|=|\mathcal{Y}|+|\mathcal{Z}\setminus\mathcal{Y}|< m^2\frac{2^{d}}{\sqrt{n}}.
\]
Since there are exactly $2^{d}$ skew set-matrices of $G$, the theorem follows.
\end{proof}

\section*{Acknowledgments}

The first author was supported by the China Scholarship Council (202306370173). The work was done during a visit of the first author to The University of Melbourne.


\begin{thebibliography}{}

\bibitem{AAS2019}
M. Arezoomand, A. Abdollahi and P. Spiga, On problems concerning fixed-point-free permutations and on the polycirculant conjecture--a survey, \emph{Trans. Comb.} 8 (2019), no.~1, 15--40.

\bibitem{Babai1980}
L. Babai, Finite digraphs with given regular automorphism groups, \emph{Period. Math. Hung.} 11 (1980), 257--270.

\bibitem{BG1982}
L. Babai and C.~D. Godsil, On the automorphism groups of almost all Cayley graphs, \emph{European J. Combin.} 3 (1982), 9--15.

\bibitem{BCP1997}
W. Bosma, J. Cannon and C. Playoust, The magma algebra system I: The user language, \emph{J. Symbolic Comput.} 24 (1997), 235--265.


\bibitem{CEP2018}
M. Conder, I. Est\'{e}lyi and T. Pisanski, Vertex-transitive Haar graphs that are not Cayley graphs, \emph{Discrete geometry and symmetry}, 61–70, \emph{Springer Proc. Math. Stat.}, 234, Springer, Cham, 2018.

\bibitem{DM1996}
J.~D. Dixon and B. Mortimer, \emph{Permutation groups}, Springer-Verlag, New York, (1996).

\bibitem{Dobsen2022}
T. Dobson, On automorphisms of Haar graphs of abelian groups, \emph{Art Discrete Appl. Math.} 5 (2022), no.~3, Paper No.~3.06, 22~pp.

\bibitem{DSV2016}
E. Dobson, P. Spiga and G. Verret, Cayley graphs on abelian groups, \emph{Combinatorica} 36 (2016), 371--393.

\bibitem{DX2000}
S.~F. Du and M.~Y. Xu, A classification of semisymmetric graphs of order $2pq$, \emph{Comm. Algebra} 28 (2000), 2685--2715.

\bibitem{DFS20201}
J.-L. Du, Y.-Q. Feng and P. Spiga, A classification of the graphical $m$-semiregular representation of finite groups, \emph{J. Combin. Theory, Ser. A} 171 (2020).

\bibitem{DFS2020}
J.-L. Du, Y.-Q. Feng and P. Spiga, On Haar digraphical representations of groups, \emph{Discrete Math.} 343 (2020), 6~pp.




\bibitem{EN}
A.~L. Edmonds and Z.~B. Norwood, Finite groups with many involutions,
\href{https://arxiv.org/abs/0911.1154v1}{https://arxiv.org/abs/0911.1154v1}.

\bibitem{EP2016}
I. Est\'elyi and T. Pisanski,
Which Haar graphs are Cayley graphs?,
\emph{Electron.~J.~Combin.}, 23 (2016), no.~3, Paper 3.10, 13 pp.

\bibitem{FKWY2020}
Y.-Q. Feng, I. Kov\'{a}cs, J. Wang and D.-W. Yang, Existence of non-Cayley Haar graphs, \emph{European J. Combin.} 89 (2020), 12 pp.

\bibitem{FKY2020}
Y.-Q.~Feng, I.~Kova\' cs and D.-W.~Yang,
On groups all of whose Haar graphs are Cayley graphs,
\emph{J. Algebraic Combin.}, 52 (2020), no.~1, 59--76.

\bibitem{FusariSpiga}
M.~Fusari and P.~Spiga, On the maximum number of subgroups of a finite group, \emph{J. Algebra}, 635 (2023), 486--526.

\bibitem{GLP2004}
M. Giudici, C.~H. Li and C.~E. Praeger, Analysing finite locally $s$-arc transitive graphs, \emph{Trans. Amer. Math. Soc.} 356 (2004), no.~1, 291--317.

\bibitem{Godsil1981}
C.~D. Godsil, On the full automorphism group of a graph, \emph{Combinatorica} 1 (1981), 243--256.

\bibitem{GMPS2015}
S. Guest, J. Morris, C.~E. Praeger and P. Spiga, On the maximum orders of elements of finite almost simple groups and primitive permutation groups, \emph{Trans. Amer. Math. Soc.} 367 (2015), 7665--7694.

\bibitem{HMP2002}
M. Hladnik, D. Maru\v{s}i\v{c} and T. Pisanski, Cyclic Haar graphs, \emph{Discrete Math.} 244 (2002), no.~1-3, 137--152.

\bibitem{Imrich1978}
W. Imrich, Graphical regular representations of groups of odd order, in: Combinatorics (Proc.~Fifth Hungarian Colloq., Keszthely, 1976) Vol.~II, \emph{Colloq.~Math.~Soc.~J\'anos Bolayi}, 18 (1978), 611--621.

\bibitem{Kantor1972}
W.~M. Kantor, $k$-homogenous groups, \emph{Math. Z.} 124 (1972), 261--265.

\bibitem{KL1990}
P.~B. Kleidman and M.~W. Liebeck, \emph{The subgroup structure of the finite classical groups}, Cambridge University Press, Cambridge, 1900.

\bibitem{KK2014}
H. Koike and I. Kov\'{a}cs, Isomorphic tetravalent cyclic Haar graphs, \emph{Ars Math. Contemp.} 7 (2014), no.~1, 215--235.

\bibitem{LM1972}
H. Liebeck and D. MacHale, Groups with automorphisms inverting most elements, \emph{Math. Z.} 124 (1972), 51--63.


\bibitem{LPS1988}
M.~W. Liebeck, C.~E. Praeger and J. Saxl, On the O'Nan--Scott theorem for finite primitive permutation groups, \emph{J. Austral. Math. Soc.} 44 (1988), 389--396.

\bibitem{LPS1990}
M.~W. Liebeck, C.~E. Praeger and J. Saxl, The maximal factorizations of the finite simple groups and their automorphism groups, \emph{Mem. Amer. Math. Soc.} 86 (1990).

\bibitem{LPS1996}
M.~W. Liebeck, C.~E. Praeger and J. Saxl, On factorizations of almost simple groups, \emph{J. Algebra} 185 (1996), 409--416.

\bibitem{LPS2000}
M.~W. Liebeck, C.~E. Praeger and J. Saxl, Transitive subgroups of primitive permutation groups, \emph{J. Algebra} 234 (2000), 291--361.

\bibitem{LPS2010}
M.~W. Liebeck, C.~E. Praeger and J. Saxl, Regular subgroups of primitive permutation groups, \emph{Mem. Amer. Math. Soc.} 203 (2010).

\bibitem{Lubotzky2001}
A. Lubotzky, Enumerating boundedly Generated Finite Groups, \emph{J. Algebra} 238 (2001), 194--199.

\bibitem{Maroti2002}
A. Mar\'{o}ti, On the orders of primitive groups, \emph{J. Algebra} 258 (2002), 631--640.

\bibitem{Dragan81}
D. Maru\v{s}i\v{c}, On vertex symmetric digraphs, \emph{Discrete Math.} 36 (1981), 69--81.

\bibitem{MP1994}
B.D. McKay and C.E. Praeger, Vertex-transitive graphs which are not Cayley graphs, I, \emph{J. Austral. Math. Soc. Ser. A} 56 (1994), no.~1, 53--63.

\bibitem{MMS2022}
J. Morris, M. Moscatiello and P. Spiga, On the asymptotic enumeration of Cayley graphs, \emph{Ann. Mat. Pura Appl.} 201 (2022), 1417--1461.

\bibitem{MS2021}
J. Morris and P. Spiga, Asymptotic enumeration of Cayley digraphs, \emph{Israel J. Math.} 242 (2021), 401--459.

\bibitem{MS2024}
J. Morris and P. Spiga, Haar graphical representations of finite groups and an application to poset representations, \href{https://arxiv.org/abs/2404.12658}{https://arxiv.org/abs/2404.12658}.

\bibitem{MSV2015}
J. Morris, P. Spiga and G. Verret, Automorphisms of Cayley graphs on generalised dicyclic groups, \emph{European J. Combin.} 43 (2015), 68--81.


\bibitem{NW1972}
L.~A. Nowitz and M.~E. Watkins, Graphical regular representations of non-abelian groups. I, II, \emph{Canadian J.~Math.}, 24 (1972), 993--1018.



\bibitem{PS1997}
L. Pyber and A. Shalev, Asymptotic results for primitive permutation groups, \emph{J. Algebra} 188 (1997), 103--124.


\bibitem{Praeger1997}
C.~E. Praeger, Finite quasiprimitive graphs, \emph{Surveys in Combinatorics}, 1997 (London), Cambridge Univ. Press, (1997), 65--85.

\bibitem{Robbins1955}
H. Robbins, A remark on Stirling's formula, \emph{Amer. Math. Monthly} 62 (1955), 26--29.


\bibitem{Spiga2021}
P. Spiga, On the equivalence between a conjecture of Babai-Godsil and a conjecture of Xu concerning the enumeration of Cayley graphs, \emph{Art Discrete Appl. Math.} 4 (2021), no.~1, 1--10.

\bibitem{Spiga2024}
P. Spiga, Finite transitive groups having many suborbits of cardinality at most $2$ and an application to the enumeration of Cayley graphs, \emph{Canad. J. Math.} 76 (2024), no.~1, 345--366.

\bibitem{Stefan}
K. Stefan, A bound on the order of the outer automorphism group of finite simple group of given order, https://stefan-kohl.github.io/preprints/outbound.pdf.


\bibitem{XZ2023}
B. Xia and S. Zheng, Asymptotic enumeration of graphical regular representations, \emph{Proc. London Math. Soc.}(3) 127 (2023), 1424--1450.

\end{thebibliography}
\end{document}